\documentclass[11pt,letterpaper]{amsart}

\usepackage{baskervald}
\usepackage[all]{xy}
\usepackage{amssymb}
\usepackage{amscd}
\usepackage{amsmath,amsfonts,amsthm,amssymb,amscd}
\usepackage{mathrsfs}

\newcommand{\plim}[1][]{\mathop{\varprojlim}\limits_{#1}}
\usepackage{enumerate}

\newtheorem{Thm}{Theorem}[section]
\newtheorem{Cor}[Thm]{Corollaly}
\newtheorem{Prop}[Thm]{Proposition}
\newtheorem{Lem}[Thm]{Lemma}
\newtheorem{Rem}[Thm]{Remark}
\newtheorem{Def}[Thm]{Definition}
\newtheorem{Asum}[Thm]{Assumption}

\def\slash#1{\not\!#1}

\numberwithin{equation}{section}

\newcommand{\ord}{{\rm{ord}}}
\newcommand{\divv}{{\rm{div}}}
\newcommand{\Det}{{\rm{det}}}
\newcommand{\Br}{{\rm{Br}}}
\newcommand{\CH}{{\rm{CH}}}
\newcommand{\Pic}{{\rm{Pic}}}
\newcommand{\Hom}{{\rm{Hom}}}
\newcommand{\Ker}{{\rm{Ker}}}
\newcommand{\CoKer}{{\rm{Coker}}}
\newcommand{\Tr}{{\rm{Tr}}}
\newcommand{\Gal}{{\rm{Gal}}}
\newcommand{\inv}{{\rm{inv}}}
\newcommand{\ind}{{\rm{ind}}}
\newcommand{\cor}{{\rm{cor}}}
\newcommand{\NS}{{\rm{NS}}}
\newcommand{\res}{{\rm{res}}}
\newcommand{\Nr}{{\rm{Nr}}}
\newcommand{\rat}{{\rm{rat}}}
\newcommand{\Spec}{{\rm{Spec}}}
\newcommand{\Cl}{{\rm{Cl}}}
\newcommand{\resp}{{\rm{resp}}}
\newcommand{\Imm}{{\rm{Im}}}
\newcommand{\Proj}{{\rm{Proj}}}
\newcommand{\dlog}{{\rm{dlog}}}
\newcommand{\tcf}{{\rm{the\;class\;of}}}
\newcommand{\val}{{\rm{val}}}
\newcommand{\tr}{{\rm{tr}}}
\newcommand{\gr}{{\rm{gr}}}
\newcommand{\Symb}{{\rm{Symb}}}
\newcommand{\sing}{{\rm{sing}}}
\newcommand{\id}{{\rm{id}}}
\newcommand{\loc}{{\rm{loc}}}
\newcommand{\Gys}{{\rm{Gys}}}
\newcommand{\lin}{{\rm{lin}}}
\newcommand{\Ch}{{\rm{Ch}}}
\newcommand{\Frac}{{\rm{Frac}}}
\newcommand{\codim}{{\rm{codim}}}
\newcommand{\ur}{{\rm{ur}}}
\newcommand{\cd}{{\rm{cd}}}
\newcommand{\pr}{{\rm{pr}}}
\newcommand{\Cone}{{\rm{Cone}}}

\begin{document}

\title[Duality for $p$-adic \'Etale Tate twists with modulus ]{Duality for $p$-adic \'Etale Tate twists with modulus}

\author[K. Yamamoto]{Kento Yamamoto}
\address{Department of Mathematics,\;Hokkaido University,\;Sapporo 060-0810\;Japan.}
\email{yamamoto@math.sci.hokudai.ac.jp}

\date{\today}

\begin{abstract}
In this paper, we define $p$-adic \'etale Tate twists for a modulus pair $(X, D)$, where $X$ is a regular 
semi-stable family and $D$ is an effective Cartier divisor on $X$ which is flat over a base scheme. The main result of this paper is an arithmetic duality of $p$-adic \'etale Tate twists for proper modulus pairs $(X, D)$, which holds as a pro-system with respect to the multiplicities of the irreducible components of $D$. 
\end{abstract}



\maketitle
\tableofcontents

\section{Introduction }
Let $p$ be a prime number.
Let $K$ be a $p$-adic field and let $\mathscr{O}_K$ be its valuation ring with $k$ the residue filed.
In recent years, theory of motives and the (higher) Chow group has been studied for a modulus pair $(X,D)$\;(\cite{KMSYI},\;\cite{KMSYII}, \cite{KMSYIII}, \cite{BS}, and \cite{RS} etc.).
One of the important question with respect to the theory of motives with modulus is $p$-adic realization of it. In \cite{Sa2} Conjecture 1.4.1 and \cite{Sa3} Remark 7.2 imply that $p$-adic \'etale Tate twists have a deeply
relation with motivic complexes. 
 
  In this paper, we introduce $p$-adic \'etale Tate twists $\mathfrak{T}_n(r)_{X|D}$ for a modulus pair $(X,D)$, where $X$ is a regular semi-stable family over $\mathscr{O}_K$ and $D$ is an effective Cartier divisor on $X$\;(we assume that $D$ is flat over $\mathscr{O}_K$). The object $\mathfrak{T}_n(r)_{X|D}$ is a generalization of $p$-adic \'etale Tate twists\;(cf.\;\cite{Sch},\;\cite{Sa2}). We prove arithmetic duality of $\mathfrak{T}_n(r)_{X|D}$ for  proper modulus pairs $(X,D)$, where $X$ is a regular semi-stable family and proper over $\mathscr{O}_K$.

We proved the following result in this paper: 
\begin{Thm}\label{Main Thm} Let $r \leq p-2$ and $\zeta_p\in K$. We put $Y^0$ a set of generic points of Y. We assume that $X$ is a semi-stable family and proper over $\mathscr{O}_K$ and $|D| \cap Y^0= \emptyset$. The pairing\\
\label{eeq1}\[\plim[m]H^i_{\acute{e}t}(X, \mathfrak{T}_n(r)_{X|mD}) \times H_{Y \cap U}^{2d+1-i}(U, \mathfrak{T}_n(d-r)_U) \longrightarrow \mathbb{Z}/p^n,\]\\
is a non-degenerate pairing for any $i\geq 0$ and $n \geq 0$.
\end{Thm}
When $D=\emptyset$ in Theorem \ref{Main Thm}, we have Theorem 1.2.2 in \cite{Sa2}. Here $\mathfrak{T}_n(d-r)_U$ is $p$-adic \'etale Tate twists, which is defined in \cite{Sa2}. 
 We prove Theorem$\ref{Main Thm}$ along the strategy of Theorem 1.2.2 in \cite{Sa2}. The key idea of the construction of pairing in Theorem \ref{Main Thm} is a Milne's pairing (a pairing of two-term complexes) in \cite{Mil}, p.175 or \cite{Mil2}, p.$217$. This construction idea is also applied to the proof of Theorem4.1.4 in \cite{JSZ}. We construct the pairing in Theorem \ref{Main Thm} by extending Milne's pairing to a pairing of multiple-term complexes. We introduce some logarithmic Hodge-Witt type sheaves for modulus pairs $(X,D)$ to show Theorem \ref{Main Thm} and show its duality property.

An application of Theorem \ref{Main Thm}, we have the following Corollary\;(Take $i=2d$ and $r=d$ in Theorem $\ref{Main Thm}$):
\begin{Cor}(cf. \cite{JSZ}) Under the same assumption of Theorem \ref{Main Thm}, we obtain a natural isomorphism 
\[\plim[m]H^{2d}(X, \mathfrak{T}_n(d)_{X|mD})\xrightarrow{\cong}\pi_1^{ab}(U)/p^n. \]
\end{Cor}
Here $\pi_1^{ab}(U)$ is the abelianized fundamental group of $U$. \\

\textbf{Notation and Conventions.}\begin{enumerate}[\upshape (i)]
\item Let $K$ be a $p$-adic field, $\mathscr{O}_K$ denotes the integer ring of $K$, and $\pi_K$ denotes the prime element of $\mathscr{O}_K$. We denote $k$ the residue field of $\mathscr{O}_K$ which is a perfect field of characteristic $p$.
\item Let $p$ be a prime number, and let $\zeta_p$ be a primitive $p$-the roots of unity.
\item Unless indicated otherwise, all cohomology groups of schemes are taken over the \'etale topology.
For a scheme $X$ and an integer $n \geq 2$, $D(X_{\acute{e}t}, \mathbb{Z}/n)\\({\rm resp.}\; D^b(X_{\acute{e}t}, \mathbb{Z}/n))$ denotes the derived category of (resp. bounded) complexes of \'etale $\mathbb{Z}/n$-sheaves on $X_{\acute{e}t}.$
\item\ Throughout this paper, we assume that a scheme $X$ is always separated over $\mathscr{O}_K$. 
\item Let $p$ be a prime number. For a scheme $X$, we put $X_n:=X \otimes \mathbb{Z}/p^n\mathbb{Z}$.
\item Let $k$ be a field, and let $X$ be a pure-dimensional scheme which is of finite type over $\Spec(k)$. We call $X$ $a\; normal\; crossing\; scheme$ over $\Spec(k)$, if it is evrywhere \'etale locally isomorphic to 
\[\Spec\big(k[T_0,...T_N]/(T_0\cdot\cdot\cdot T_N)\big),\] 
\noindent
for some integer $a$ with $0\leq a \leq N:=\dim(X)$.

\item Let $X$ be a pure-dimensional scheme which is flat of finite type over $\Spec(\mathscr{O}_K)$. We call X be a $regular\;semistable\;family$ over  $\Spec(\mathscr{O}_K)$, if it is regular and 
evrywhere \'etale locally isomorphic to 
\[\Spec\big(\mathscr{O}_K[T_0,...T_d]/(T_0\cdot\cdot\cdot T_a-\pi_K)\big),\] 
\noindent
for some $a$ such that $0\leq a \leq d:=\dim(X/\mathscr{O}_K)$.
\item For a positive integer $n$ invertible on a scheme $X$, $\mu_n$ denotes the \'etale sheaf of $n$-th roots of unity. 
\end{enumerate}
\mbox{}

\noindent
We assume the following situation\;($\clubsuit$)\;:
\begin{itemize}
\item Let $X$ be a $regular\;semistable\;family$ over  $\Spec(\mathscr{O}_K)$. We denote $Y:=X \otimes_{\mathscr{O}_K}k$ and $X_K:=X \otimes_{\mathscr{O}_K}K$. Let $D \subset X$ be an effective Cartier divisor on $X$ which is flat over $\Spec(\mathscr{O}_K)$ and $Y \cup D_{\rm{red}}$ has normal crossings on $X$.
\item Let $M_X$ be a logarithmic structure on $X$ associated with $D_{\rm{red}} \cup Y$.
 Let $M_D$ be a logarithmic structure on $D$ induced by the inverse image of $M_X$. For $n\geq 1$, we write $M_{X_n}$ for the inverse image of log structure of $M_X$ into $X_n$.
Let $(Y, M_Y)$ be the reduction of mod $\pi$ of $(X, M_X)$.
\end{itemize}
 Put $U:=X-D$, $V:=Y-(Y \cap D)$, $U_K:=X-(Y \cup D)$, and $V_X:=X-Y$, and we consider a diagram of immersions.
\[
 \xymatrix@M=7pt@C=50pt{
    V \ar@{^{(}->}[r]^{i'} \ar@{^{(}->}[d]_{\alpha'} & U \ar@{^{(}->}[d]_{\alpha}&  \ar@{^{(}->}[d] \ar@{_{(}->}[l]_{j'} U_K \ar@{^{(}->}[ld]^\psi \\
    Y \ar@{^{(}->}[r]^i & X& \ar@{_{(}->}[l]^j V_X
   }\]

%%%%%%%%%%%%%%%%%%%%%%%%%%%%%%%%%%

%%%%%%%%%%%%%%%%%%%%%%%%%%%%%%%%%%%%%%%%%%%%%%%%%%%%%%%%%%%%%%%%%%%%%%%%%%

%%%%%%%%%%%%%%%%%%%%%%%%%%
\subsection{Logarithmic Hodge-Witt type sheaves\;(\cite{Sa1},\cite{Sa2},\cite{Sa3})}\mbox{}

Throughout this subsection, $n$ denotes a non-negative integer and $r$ denotes a positive integer.
Let $k$ be a perfect field of positive characteristic $p$. Let $X$ be a pure-dimensional scheme of finite type over $\Spec(k)$. We define the \'etale sheaves $\nu_{X,r}^n$ and $\lambda_{X,r}^n$ on $X$ as follows:
\begin{equation}
\nu_{X,r}^n:=\Ker \Big(\bigoplus_{x \in X^{0}}{i_{x}}_*W_r\Omega_{x,\log}^n \stackrel{\partial^{\val}}{\longrightarrow }\bigoplus_{x \in X^{1}}{i_{x}}_*W_r\Omega_{x,\log}^{n-1}    \Big)
\end{equation}
\begin{equation}
\lambda_{X,r}^n:=\Imm \Big((\mathbb{G}_{m,X})^{\otimes n}\stackrel{\dlog}{\longrightarrow }\bigoplus_{x \in X^{0}}{i_{x}}_*W_r\Omega_{x,\log}^{n}    \Big)
\end{equation}

\noindent
where $i_x$ denote the canonical map $x \hookrightarrow X$ for a point $x \in X$ and $\partial^{\val}$ denote the sum of ${\partial^{\val}}_{y,x}$'s with $y \in X^0$ and $x \in X^1$(see \cite{Sa2}, $\S 1.8$). Here $W_r\Omega_{X,\log}^n$ denotes the \'etale subsheaf of $W_r\Omega_{X}^n$, which defined in \cite{B}  and \cite{Il} for a smooth variety over $\Spec(K)$.
It is obvious that $\lambda_{X,r}^n$ is  a subsheaf of $\nu_{X,r}^n$. If $X$ is smooth then both $\nu_{X,r}^n $ and  $\lambda_{X,r}^n$ agree with the sheaf $W_r\Omega_{X, \log}^n$.
We next review the duality result in \cite{Sa1}.  For integers $m,n \geq 0$, there is a natural pairing of sheaves 
\begin{equation}
\nu_{X,r}^n \times \lambda_{X,r}^m \longrightarrow \nu_{X,r}^n
\end{equation}
 (see loc.\;cit.,3.1.1).
\begin{Thm}\label{dualnu}(Duality, loc.\;cit., {\rm1.2.2}).
Let $k$ be a finite field, and let $X$ be a normal crossing scheme of dimension $N$ which is proper over $\Spec(k)$. Then for integers $q$ and $n$ with $0 \leq n \leq N$, the natural mapping
\begin{equation}
H^q(X, \nu_{X,r}^n )\times H^{N+1-q}(X,  \lambda_{X,r}^{N-n})\longrightarrow H^{N+1}(X, \nu_{X,r}^n )\stackrel{\tr_X}{\longrightarrow}\mathbb{Z}/p^r\mathbb{Z}
\end{equation}
is non-degenerate pairing of finite $\mathbb{Z}/p^r\mathbb{Z}$-modules.
\end{Thm}

\noindent
Here trace map $\tr_X$ is constructed in \cite{Sa1}, Proposition 2.5.9 and Corollary 2.5.11. In \cite{Sa3}, Sato have introduced a following two new logarithmic sheaves under the condition ''with Cartier divisor $D$":
\begin{Def}(\cite{Sa3}, {\rm \S2,  Definition 2.1})
\begin{equation}
\nu_{(Y, Y\cap D),n}^q:= \Ker \Big(\bigoplus_{y \in V^{0}}{i_{y}}_*W_r\Omega_{y,\log}^q \stackrel{\partial^{\val}}{\longrightarrow }\bigoplus_{y \in V^{1}}{i_{y}}_*W_r\Omega_{y,\log}^{q-1} \Big),
\end{equation}
\begin{equation}
 \lambda^q_{(Y,Y\cap D), n}:=\Imm\left(d\log: (\alpha'_*\mathscr{O}_V)^{\times}\longrightarrow \bigoplus_{y\in V^0} {i_y}_*W_n\Omega^q_{x,\log}\right).
\end{equation}
Here $V:=Y-(Y \cap D)$ and $i_y : y \rightarrow Y (y \in Y)$ denotes the composite map $y \hookrightarrow V \hookrightarrow Y$.
\end{Def}
We will introduce a new logarithmic sheaves as a subsheaf of these sheaves in section  below.  

\begin{Def}($p$-adic \'etale Tate twist\;\cite{Sa2}, Definition 4.2.4)\\
Sato have defined the $p$-adic \'etale Tate twist $\mathfrak{T}_n(r)_{X}$ by the following distinguished triangle:
\[i_*\nu^{r-1}_{Y,n}[-r-1]\longrightarrow \mathfrak{T}_n(r)_{X} \longrightarrow \tau_{\leq r}Rj_*\mu_{p^n}^{\otimes r} \longrightarrow i_*\nu^{r-1}_{Y,n}[-r]\]
in $D^b\left(X_{\acute{e}t}, \mathbb{Z}/p^n\mathbb{Z}\right)$.
\end{Def}
The $p$-adic \'etale Tate twist $\mathfrak{T}_n(r)_{X}$ has some good properties (trivialization, acyclicity, purity, compatibility, product structure, etc.). See \cite{Sa2} for the details of properties of $p$-adic \'etale Tate twist.

%%%%%%%%%%%%%%%%%%%%%%%%%%%%%%%%%%%%%%%%%%%%%%%%%%%%%%%%%%%%%%%%%%%%%%%%%%%%%%%%%%%%%%%%%%

%%%%%%%%%%%%%%%%%%%%%%%%%%%%%%%%%%%%%%%%%%%%%%%%%%%%%%%%%%%%%%%%%%%%%%%%%%%%%%%%%%%%%%%%%%%%%%%%%%%%
\section{$p$-adic \'etale Tate twists with modulus}\mbox{}
We assume that $0 \leq r \leq p-2$ and $X$ is in the setting ($\clubsuit$). The aim of this section is to construct $p$-adic \'etale Tate twists with modulus $\mathfrak{T}_n(r)_{X|D} \in D^b(Y_{\acute{e}t}, \mathbb{Z}/p^n\mathbb{Z})$. We first review the definition of the syntomic complexes with modulus $s_n(r)_{X|D}\in D^b(Y_{\acute{e}t}, \mathbb{Z}/p^n\mathbb{Z})$, which is constructed in \cite{Y}. To define the syntomic complex with modulus in local situation, we assume the existence of the following four datas:
\begin{Asum}\label{Asum*}
\begin{itemize}
\item There exists exact closed immersions \[\beta_n: (X_n,M_{X_n}) \hookrightarrow(Z_n,M_{Z_n})\ \ {\rm and}\ \ \beta_{n,D}: (D_n,M_{D_n}) \hookrightarrow(\mathscr{D}_n,M_{\mathscr{D}_n})\] of log schemes for $n \geq 1$ such that $(Z_n,M_{Z_n})$ and $(\mathscr{D}_n,M_{\mathscr{D}_n})$ are smooth over $W:=W(k)$, and such that the following diagram is cartesian :
\[
\xymatrix@M=10pt{
X_n \ar@{}[rd]|{\square} \ar@<-0.3ex>@{^{(}->}[r]^-{\beta_n}&  Z_n &\\
D_n \ar@<-0.3ex>@{^{(}->}[u]\ar@<-0.3ex>@{^{(}->}[r]^-{\beta_{D,n}}& \mathscr{D}_n \ar@<-0.3ex>@{^{(}->}[u]&  
}
\]

\item There exist a compatible system of liftings of  Frobenius endomorphisms $\{F_{Z_n}:(Z_n,M_{Z_n})\rightarrow (Z_n,M_{Z_n})\}$ and $\{F_{\mathscr{D}_n}:(\mathscr{D}_n,M_{\mathscr{D}_n})\rightarrow(\mathscr{D}_n,M_{\mathscr{D}_n})\}$ for each $n \in \mathbb{N}$ {\rm(\cite[p.71, (2.1.1)--(2.1.3)]{Tsu0})}.
\item\  The systems $\{F_{Z_n}:(Z_n,M_{Z_n})\rightarrow (Z_n,M_{Z_n})\}$ and $\{F_{\mathscr{D}_n}:(\mathscr{D}_n,M_{\mathscr{D}_n})\rightarrow (\mathscr{D}_n,M_{\mathscr{D}_n})\}$fit into the following commutative diagram for for each $n \in \mathbb{N}$: 
\[
\xymatrix@M=10pt{
(Z_n,M_{Z_n}) \ar[r]^-{F_{Z_n}}& (Z_n,M_{Z_n})&\\
(\mathscr{D}_n,M_{\mathscr{D}_n}) \ar[u]\ar[r]^-{F_{\mathscr{D}_n}}&(\mathscr{D}_n,M_{\mathscr{D}_n})  \ar[u]&  
}\]

\item  $\mathscr{D}_n$ is an effective Cartier divisor on $Z_n$ such that $\beta_{D,n}^*\mathscr{D}_n=D_n$ and $F_{Z_n}$ which induces a morphism $\mathscr{D}_n \longrightarrow \mathscr{D}_n$.
\end{itemize}
\end{Asum}

In \cite{Y}, the following complex was defined for a modulus pair $(X,D)$:

\begin{Def}{\rm(Definition2.11,\;\cite{Y}\;syntomic complex with modulus)}\\
We assume $r \leq p-1$. We define
\[
s_n(r)_{X|D, (Z_n,M_{Z_n}), (\mathscr{D}_n,M_{\mathscr{D}_n})}:={\rm Cone}\big(1-\varphi_{r}:  J_{\mathscr{E}_n}^{[r-\text{\large$\cdot$}]}\otimes_{\mathscr{O}_{Z_n}}\omega_{Z_n|\mathscr{D}_n}^{\text{\large$\cdot$}} \rightarrow   \mathscr{O}_{\mathscr{E}_n}\otimes_{\mathscr{O}_{Z_n}}\omega_{Z_n|\mathscr{D}_n}^{\text{\large$\cdot$}} \big)[-1],
\]
where $\varphi_r=\varphi_{r-q}\otimes \land^{q} \frac{d\varphi}{p}$ in degree $q$. We denote this complex by $s_n(r)_{X|D}$ for simplicity. We put $\omega_{Z_n}^q:=\Omega_{Z_n/W_n}^q\big(\log M_{Z_n}\big)$ and put $\omega_{Z_n|\mathscr{D}_n}^{q}:=\omega_{Z_n}^{q}\otimes_{\mathscr{O}_{Z_n}} \mathscr{O}_{Z_n}(-\mathscr{D}_n)\;(q\geq 0)$ which are locally free $\mathscr{O}_{Z_n}$-modules. 
\end{Def} When we consider the global situation of the syntomic complex with modulus, we can patch the another complex ${\mathscr{S}_n(q)}^{loc}_{X|D}$, which is quasi-isomorphic to $s_n(r)_{X|D}$ in local situation\;(see \cite{Y}, Lemma 2.12).
There is a product structure of syntomic complex $s_n(r)_{X|D}$:

\begin{Def}(Product structure of $s_n(r)_{X|D}$, cf.\;\S2.2\;\cite{Tsu1})\label{Prod}
For $q, q' \geq 0$, we define a product  morphism in $D^-(Y_{\acute{e}t}, \mathbb{Z}/p^n\mathbb{Z})$ as follows:
\begin{equation}
s_n(q)_{X|D} \otimes^{\mathbb{L}} s_n(q')_{X|D} \longrightarrow s_n(q+q')_{X|D}.
\end{equation}
\[(x, y)\otimes (x', y') \mapsto (xx', (-1)^j  xy' + y \varphi_{q'}(x')), \]
where\[(x,y) \in s^j_n(q)_{X|D}=\left(J^{[q-j]}_{\mathscr{E}_n} \otimes_{\mathscr{O}_{Z_n}} \omega^j_{Z_n|\mathscr{D}_n}\right)\oplus \left( \mathscr{O}_{\mathscr{E}_n}\otimes_{\mathscr{O}_{Z_n}} \omega^{j-1}_{Z_n|\mathscr{D}_n} \right)\] and 
\[(x', y') \in s^{j'}_n(q')_{X|D}\left(J^{[q'-j']}_{\mathscr{E}_n} \otimes_{\mathscr{O}_{Z_n}} \omega^{j'}_{Z_n|\mathscr{D}_n}\right)\oplus \left( \mathscr{O}_{\mathscr{E}_n}\otimes_{\mathscr{O}_{Z_n}} \omega^{j'-1}_{Z_n|\mathscr{D}_n} \right).\]
\end{Def}

\noindent
We define the following logarithmic Hodge-Witt type sheaf for a modulus pair $(X,D)$, which plays an important role in the proof of main result of this paper.
\begin{Def}
For $q \geq 1$, we define
\begin{equation}
\nu_{Y|D,n}^{q-1}:=\Imm\Big(\mathcal{H}^{q}\big(s_n(q)_{X|D}\big)\xrightarrow{\theta} i^*R^{q}\psi_*\mu_{p^n}^{\otimes q}\xrightarrow{\sigma_{(X,D)}} \nu_{(Y, D\cap Y),n}^{q-1}\Big).
\end{equation}
Here morphism $\theta$ is induced by the map ${s_n(q)}_{X|D}\longrightarrow {\mathcal{S}_n(q)}_{(X,M_X)}$ and the map $\sigma_{(X,D)}$ is defined in \cite{Sa3}.
\end{Def}
\begin{Rem}
When $D_s=\emptyset$, we have $\nu^{r-1}_{Y|D,n}=\nu^{r-1}_{Y,n}$, which is considered in Sato's paper \cite{Sa1}.
The last morphism $\sigma_{(X,D)}$ is surjective by \cite{Sa3}, Theorem 3.4.
When $Y \subset D_s$, we have $\nu_{Y|D,n}^{q-1}=0$. 

\end{Rem}

\noindent
Let $\sigma_{X|D,n}^r$ be the morphism $s_n(r)_{X|D}\longrightarrow \nu_{Y|D,n}^{r-1}[-r]$ in $D^b(Y_{\acute{e}t},\mathbb{Z}/p^n\mathbb{Z})$,
which naturally induced by the definition of $\nu_{Y|D}^{r-1}$. This is an analogue of the map $\sigma^r_{X,n}$ in (3.2.5), \cite{Sa2}.
\begin{Lem}Suppose $r \geq 1$, and let 
\begin{equation}
 \nu_{Y|D,n}^{r-1}[-r-1] \xrightarrow{g} \mathcal{K} \xrightarrow{t} s_n(r)_{X|D} \xrightarrow{\sigma_{X|D,n}(r)}\nu_{Y|D,n}^{r-1}[-r]
 \end{equation}
 \noindent
 be a distinguished triangle in $D^b(Y_{\acute{e}t},\mathbb{Z}/p^n\mathbb{Z})$. Then $\mathcal{K}$ is concentrated in $[0,r]$, the triple $(\mathcal{K},t,g)$ is 
 unique up to a unique isomorphism and $g$ is determined by the pair $(\mathcal{K}, t)$. 
\end{Lem}
\begin{proof}
The map $\sigma_{X|D,n}(r)$ is surjective by the definition of $\nu_{Y|D,n}^{r-1}$. Since $\mathcal{K}$ is acyclic outside of [$0,r$], by using Lemma 2.1.1 in \cite{Sa2},
we can compute $\Hom_{D^b(Y_{\acute{e}t},\mathbb{Z}/p^n\mathbb{Z})}\left(\mathcal{K}, \nu_{Y|D,n}^{r-1}[-r-1]\right)=0$. Then there is no non-zero morphism from $\mathcal{K}$ to $ \nu_{Y|D,n}^{r-1}[-r-1]$. The uniqueness assertion follows from this fact and  Lemma 2.12 (3) in \cite{Sa2}.
\end{proof}
\begin{Def}
For $r \geq 1$, we fix a pair $(\mathcal{K}, t)$ fitting into a distinguished triangle of the form (\ref{tsyn dist}) and define $\mathcal{K}=:\mathfrak{T}_n(r)_{X|D}^{syn}$. 
\end{Def}
\begin{Lem}(Product structure of $\mathfrak{T}_n(r)_{X|mD}^{syn}$, cf. \cite{Sa2})
For $r, r' \geq 0$, there is a unique morphism
\begin{equation}
\{\mathfrak{T}_n(r)_{X|mD}^{syn}\}_m \otimes^{\mathbb{L}} \{\mathfrak{T}_n(r')_{X|mD}^{syn}\}_m\longrightarrow \{\mathfrak{T}_n(r+r')_{X|mD}^{syn}\}_m \quad in\;D^-(Y_{\acute{e}t}, \mathbb{Z}/p^n\mathbb{Z})
\end{equation}
\begin{proof}We show the assertion by the same argument in Prop 4.2.6 in \cite{Sa2}.
Put $\mathfrak{P}:=\mathfrak{T}_n(r)_{X|mD}^{syn} \otimes^{\mathbb{L}} \mathfrak{T}_n(r')_{X|mD}^{syn}$. By definition of $\mathfrak{T}_n(r+r')_{X|mD}^{syn} $ and Lemma 2.1.2 (1) in \cite{Sa2}, it suffices to show that the following composite morphism is zero in $D^-(Y_{\acute{e}t}, \mathbb{Z}/p^n\mathbb{Z})$:
\begin{equation*}
\mathfrak{P}\rightarrow s_n(r)_{X|mD}\otimes^{\mathbb{L}} s_n(r')_{X|mD}\rightarrow s_n(r+r')_{X|mD} \xrightarrow{\sigma_{X|D,n}(r+r')} \nu_{Y|D,n}^{r+r'-1}[-r-r'],
\end{equation*}
where the second arrow is product structure of $s_n(r)_{X|mD}$(Lemma \ref{Prod}).  Because $\mathfrak{P}$ is concentrated in degrees $[0, r+r']$, this composite morphism is determined by the composite map of the ($r+r'$)-th cohomology sheaves(Lemma 2.1.1 in \cite{Sa2}). Then we consider the following composition map:
{\footnotesize \begin{equation}
\mathcal{H}^{r+r'}(\mathfrak{P})\rightarrow \mathcal{H}^{r}(s_n(r)_{X|mD})\otimes \mathcal{H}^{r'}(s_n(r')_{X|mD})\rightarrow \mathcal{H}^{r+r'}(s_n(r+r')_{X|mD}) \xrightarrow{\sigma_{X|D,n}(r+r')} \nu_{Y|D,n}^{r+r'-1}.
\end{equation}}
The image of $\mathcal{H}^{r+r'}(\mathfrak{P})(\cong F\mathscr{M}^r_m \otimes F\mathscr{M}^{r'}_m )$ into $ \mathcal{H}^{r+r'}(s_n(r+r')_{X|mD})$ is contained in $ F\mathscr{M}^{r+r'}_m$. Hence this composite map is zero by $(\flat)$ in Theorem \ref{StrM} below.
\end{proof}
\end{Lem}
We define a $p$-adic \'etale Tate twist for a modulus pair $(X,D)$ :
\begin{Def}( $p$-adic \'etale Tate twist with modulus )\\
We define $\mathfrak{T}_n(r)_{X|D}$  by the following homotopy square:
\[
\quad\quad\quad\quad\quad\quad\xymatrix@M=17pt@C=87pt{   
 \mathfrak{T}_n(r)_{X|D}\ar@{}[rd]|{\square}\ar[d]\ar[r]& \tau_{\leq r}R\psi_*\mu_{p^n}^{\otimes r}\ar[d] &\\
 i_*\mathfrak{T}_n(r)_{X|D}^{syn}\ar[r]&i_*\mathcal{S}_n(r)_{(X,M_X)},&
 }   
\]
where the right vertical morphism is composite adjunction morphism and isomorphism in \cite{Tsu2}\;(here we use the assumpution $0 \leq r \leq p-2$)\;{\rm:}
\[\tau_{\leq r}R\psi_*\mu_{p^n}^{\otimes r}\xrightarrow{adj.} \tau_{\leq r}(i_*i^*R\psi_*\mu_{p^n}^{\otimes r})\cong i_*\mathcal{S}_n(r)_{(X,M_X)}.\]
\end{Def}
\noindent
Note that by the definition of $\mathfrak{T}_n(r)_{X|D}$, there is a distinguished triangle:
\begin{equation}\label{tsyn dist}
\mathfrak{T}_n(r)_{X|D} \rightarrow i_*\mathfrak{T}_n(r)_{X|D}^{syn}\oplus \tau_{\leq r}R\psi_*\mu_{p^n}^{\otimes r}    \rightarrow i_*\mathcal{S}_n(r)_{(X,M_X)} \rightarrow \mathfrak{T}_n(r)_{X|D}[1].
\end{equation}
When $D=\emptyset$, we write $\mathfrak{T}_n(r)_{X}$ for $\mathfrak{T}_n(r)_{X|\emptyset}$, which agrees with that in \cite{Sa2} \S4. We call $\mathfrak{T}_n(r)_{X|D}$ a $p$-adic \'etale Tate twist with modulus.\\

We have the cohomology sheaf of  $\mathfrak{T}_n(r)_{X|D}^{syn}$: 
\begin{Lem}
\begin{equation}
\mathcal{H}^q\Big(\mathfrak{T}_n(r)_{X|D}^{syn}\Big) \cong  
 \begin{cases} \mathcal{H}^q\big(s_n(r)_{X|D}\big)  &(0 \leq q \leq r-1) ,\\
  \Ker\Big(\mathcal{H}^r\big(s_n(r)_{X|D}\big)\xrightarrow{\mathcal{H}^r\big(\sigma_{X|D,n}^r\big)} \nu_{Y|D, n}^{r-1}  \Big)& (q=r).\\
 \end{cases}
 \end{equation}
\end{Lem}
\begin{proof}
This follows form the long exact sequence of cohomology sheaves associated with a distingished triangle
(\ref{tsyn dist}). The details are streightforward and left to reader.
\end{proof}

%%%%%%%%%%%%%%%%%%%%%%%%%%%%%%%%%%%%%%%%%%%%%%%%%%%%%%%%%%%%%%%%%%%%%%%%%%%%%%%

\section{\textbf{Structure of $\Ker(\sigma_{X|mD,n}^r)$}}\mbox{}
 In \cite{Y}, Theorem 3.5, we have the theorem on the structure of $\mathcal{H}^r\big({s_n(r)}_{X|mD}\big)$.
We put $\mathscr{M}_{n,m}^r:=\mathcal{H}^r\big(s_n(r)_{X|mD}\big).$ 
We define the \'etale subsheaf $F\mathscr{M}_{n,m}^r$ of $\mathscr{M}_{n,m}^r$ as the part generated by the image of $U^1{\mathscr{M}^r_{n,m}}$ and $i^*\left(1+\mathscr{O}_X(-mD)\right)^{\times}\otimes i^*\left(\alpha_*\mathscr{O}_U^{\times}\right)^{\otimes r-1}$. Here $U^1 {\mathscr{M}^r_{n,m}}$ is defined in Definition 3.1, \cite{Y} as for \cite{Tsu1}. 

\begin{Thm}(cf. Theorem 3.4.2, \cite{Sa2})\label{StrM} The map $\sigma_{X|mD,n}^r$ induces an isomorphism 
\[(\flat)\quad{\{\mathscr{M}_{n,m}^r/F\mathscr{M}_{n,m}^r\}}_m \xrightarrow{\cong} {\{\nu^{r-1}_{Y|mD,n}\}}_m,\]
that is, ${\{F\mathscr{M}_{n,m}^r\}}_m = {\{\Ker(\sigma_{X|mD,n}^r)\}}_m$. Furthermore there is an exact sequence 
\[0 \longrightarrow {\{U^1{\mathscr{M}^r_{n,m}}\}}_m \longrightarrow {\{F\mathscr{M}_{n,m}^r\}}_m \longrightarrow {\{\lambda^{r}_{Y|mD}\}}_m \longrightarrow 0, \]
i.e. $(\flat\flat)\quad {\{F\mathscr{M}_{n,m}^r/U^1{\mathscr{M}^r_{n,m}}\}}_m \xrightarrow{\cong}  {\{\lambda^{r}_{Y|mD}\}}_m$. 
\end{Thm}
\begin{proof} By the Theorem 3.7 and 3.8 in \cite{Y}, we have $U^0\mathscr{M}_{n,m}^r=\mathscr{M}_{n,m}^r$ and a short exact sequence
\[0 \longrightarrow \{\gr^0_1\mathscr{M}_{n,m}^r\}_m \longrightarrow \{\gr^0_U\mathscr{M}_{n,m}^r\}_m \longrightarrow \{\gr^0_0\mathscr{M}_{n,m}^r\}_m \longrightarrow 0\]
as a pro-system. Then we have the natural adjunction map \[\{\mathscr{M}_{n,m}^r/U^1\mathscr{M}_{n,m}^r\}_m \longrightarrow \{\bigoplus_{y \in Y^0} i_{y*}{i_y}^* \left(\mathscr{M}_{n,m}^r/U^1\mathscr{M}_{n,m}^r \right)\}_m \] 
as a pro-system which is injective. We consider the following commutative diagram:
 {\footnotesize\[
\xymatrix@C=15pt@R=31pt{   
0\ar[r]&\{\Ker(\sigma_{X|D})/U^1\mathscr{M}_{n,m}^r\}_m\ar[r]\ar[d]^-{\tau_D}&\{\mathscr{M}_{n,m}^r/U^1\mathscr{M}_{n,m}^r\}_m\ar[r]\ar[d]&\bigoplus_{y \in Y^0}i_{y*}\Omega^{r-1}_{y,\log}\ar[r]\ar[d]^-{=}&0\\
0\ar[r]&\bigoplus_{y \in Y^0}i_{y*}\Omega^r_{y,\log}\ar[r]&\bigoplus_{y \in Y^0} i_{y*}{i_y}^* \{\left(\mathscr{M}_{n,m}^r/U^1\mathscr{M}_{n,m}^r \right)\}_m\ar[r]&\bigoplus_{y \in Y^0}i_{y*}\Omega^{r-1}_{y,\log}\ar[r]&0}
\]}
Here the bottom horizontal row induced by the direct decomposition {\footnotesize$\{\mathscr{M}_{n,m}^r/U^1\mathscr{M}_{n,m}^r\}_m \cong \{\Omega^r_{Y|mD_s,\log}\oplus \Omega^{r-1}_{Y|mD_s,\log}\}_m$} by Theorem 3.8 in \cite{Y} under the assumption Y is smooth over $\Spec(k)$.
Since the central map is an injective,  the map $\tau_D$ is also. We next  prove that the image of the map $\tau_D$ is contained in $\lambda^r_{Y|D}$. 
Here, we define the sheaf \[\lambda_{Y|D,n}^r:=\Imm\left((1+\mathscr{O}_Y(-D_s))^{\times} \otimes (\alpha'_*\mathscr{O}_V^{\times})^{\otimes r-1} \xrightarrow{{\rm d\log}}\displaystyle{ \bigoplus_{y \in V^0}} W_n\Omega_{y, \log}^{r} \right).\]
The exact sequence $(\mathfrak{V}2)$ in Proposition   below, we have an isomorphism \[\lambda^r_{Y|mD_s}\cong \Ker\left(a_{1*}\Omega^r_{Y^{(1)}|mD_s^{(1)}} \rightarrow a_{2*}\Omega^r_{Y^{(2)}|mD_s^{(2)}} \right).\] We consider the following commutative diagram as a pro-system\;:
 \[
\xymatrix@C=20pt@R=20pt{  
\Ker(\sigma_{X|D})\ar[r]\ar[d]&a_{1*}\Omega^r_{Y^{(1)}|mD_s^{(1)}}\ar[r]\ar[d]&a_{2*}\Omega^r_{Y^{(2)}|mD_s^{(2)}}\ar[d]\\
\Ker(\sigma_{(X,D)})\ar[r]&a_{1*}\Omega^r_{Y^{(1)}}\ar[r]&a_{2*}\Omega^r_{Y^{(2)}}}.
\]
Here the bottom horizontal row induced by Theorem  in \cite{Sa3}, which is a zero map. Since the right vertical map is injective, the upper horizontal row is zero map.
Therefore  we have an inclusion $\Imm(\tau_D) \subset \lambda^r_{Y|D_s}$. Finally, we show that the isomorphisms $(\flat)$ and $(\flat\flat)$. We have the following injective maps:
\[\left\{\frac{F\mathscr{M}_{n,m}^r}{U^1\mathscr{M}^r_{n,m}}\right\}_m\hookrightarrow \left\{\frac{\Ker(\sigma_{X|D})}{U^1\mathscr{M}_{n,m}^r}\right\}_m \hookrightarrow \left\{\lambda^r_{Y|mD_s}\right\}_m.\]
The sheaf $F\mathscr{M}_{n,m}^r/U^1\mathscr{M}^r_{n,m}$ is generated by symbols which form $i^*\left(1+\mathscr{O}_X(-mD)\right)^{\times}\otimes i^*\left(\alpha_*\mathscr{O}_U^{\times}\right)^{\otimes r-1}$. Since $\lambda^r_{Y|D_s}$ is generated by a symbol form $(1+\mathscr{O}_Y(-D_s))^{\times} \otimes (\alpha'_*\mathscr{O}_V^{\times})^{\otimes r-1}$, the two maps are bijective. Thus we have $\{F\mathscr{M}_{n,m}^r\}_m=\{\Ker(\sigma_{X|mD})\}_m$. This completes the proof.
\end{proof}

\begin{Rem}When $D=\emptyset$ in Theorem , we have Theorem {\rm 3.4.2}, \cite{Sa2}.
\end{Rem}

%%%%%%%%%%%%%%%%%%%%%%%%%%%%%%%%%%%%%%%%%%%%%%%%%%%%%%%%%%%%%%%%%%%%%%%%%%%%%%%%
\section{\textbf{Logarithmic Hodge-Witt type sheaves with modulus}}
In this section, we prove the structure and the duality theorem of $\nu^{r-1}_{Y|D_s,n}$ and introduce a sheaf $\lambda^r_{Y,n}$. The duality of $\nu^{r-1}_{Y|D_s,n}$ and $\lambda^r_{Y,n}$ plays an important role in our proof of the main result(see \S below). 
\begin{Lem} If $Y$ is smooth, we have $\nu^{r-1}_{Y|D_s,n}\cong W_n\Omega_{Y|D_s,\log}^{r-1}$ with respect to the multiplicity of the irreducible components of $D$.\end{Lem}
\begin{proof} We consider the following commutative diagram\;(as a pro-system):
 \[
\xymatrix@M=1.6pt@C=30pt{   
 i^*R^{r}\psi_*\mu_{p}^{\otimes r}\ar[r]^-{\phi_2}&\left(i^*R^{r}\psi_*\mu_{p^n}^{\otimes r}\right)/U^1(i^*R^{r}\psi_*\mu_{p^n}^{\otimes r})\ar[r]^-{\phi_3}&\alpha_*W_n\Omega^{r-1}_{V,\log}\\
\mathcal{H}^r(s_n(r)_{X|D})\ar[r]^-{\phi'_2}\ar[u]^-{\phi_1}&\left(\mathcal{H}^r(s_n(r)_{X|D})\right)/U^1\mathcal{H}^r(s_n(r)_{X|D})\ar[r]^-{\phi'_3}\ar[u]&W_n\Omega^{r-1}_{Y|D_s,\log}\ar[u]^-{\phi'_1}
 }   
\]
Here maps $\phi_3$ is obtained in Theorem   in \cite{Sa3}\;(under the assumption $Y$ is smooth). The map $\phi'_3$ is surjective by the induction on $n$ with applying Theorem 3.8 (1) in \cite{Y}.  By the definition, we have $\Imm(\phi_3\circ\phi_2\circ \phi_1)=\nu^{r-1}_{Y|D_s,n}$.  The image of the map $\phi'_1\circ \phi'_3 \circ \phi'_2$ is $W_n\Omega_{Y|D_s,\log}^{r-1}$. Then we obtain the assertion.
\end{proof}
\begin{Prop}\label{nu} Assume that $|D| \cap Y^0=\emptyset$. We have a short exact sequence 
\[0 \longrightarrow \nu^{r-1}_{Y|D_s,n} \longrightarrow \nu^{r-1}_{Y,n} \longrightarrow \beta'_*\nu^{r-1}_{D_s,n} \longrightarrow 0.\]
Here the last map is a restriction.
\end{Prop}
\begin{proof} We put $\mathcal{K}_{Y|D_s}^{r-1}:=\Ker( \nu^{r-1}_{Y,n} \longrightarrow \beta'_*\nu^{r-1}_{D_s,n})$. By the definiton of $\nu^{r-1}_{Y|D_s,n}$ and the assumption of $|D| \cap Y^0=\emptyset$, we have $ \nu_{(Y, |D|\cap Y),n}^{r-1}= \nu_{Y, n}^{r-1}$. Let $Y^{{\rm sing}}$ be the singular locus of $Y$ and let $\tilde{j}:\tilde{Y}:=Y\backslash Y^{{\rm sing}} \hookrightarrow Y$ be the open immersion. We define $\tilde{D}_s:=D_s \times_Y \tilde{Y}$. There is a commutative diagram of exact rows:
 \[
\xymatrix@M=1.6pt@C=30pt{   
0 \ar[r]&\mathcal{K}_{Y|D_s}^{r-1} \ar[r]\ar[d] & \nu^{r-1}_{Y,n}  \ar[r]\ar[d]& \beta'_*\nu^{r-1}_{D_s,n}\ar[d]\ar[r] &0\\
0 \ar[r]&  j_{2*}W_n\Omega_{\tilde{Y}|\tilde{D}_s, \log}^{r-1} \ar[r] & j_{2*}W_n\Omega_{\tilde{Y}, \log}^{r-1}\ar[r] &  j'_{2*}W_n\Omega_{\tilde{D}_s, \log}^{r-1}.&
 }   
\]Here the injectivity of the middle vertical (embedded) map and the right vertical map are due to \cite{Sa1}.
Then we have a relation $(\bigstar)\;\mathcal{K}_{Y|D_s}^{r-1}= \nu^{r-1}_{Y,n} \cap j_{2*}W_n\Omega_{\tilde{Y}|\tilde{D}_s, \log}^{r-1}\supset  \nu^{r-1}_{Y|D_s,n} \cap j_{2*}W_n\Omega_{\tilde{Y}|\tilde{D}_s, \log}^{r-1}=  \nu^{r-1}_{Y|D_s,n}$ by using the embedding $(*1)$ in Definition \ref{embed} below.
There is a commutative diagram:
 \[
\xymatrix@M=1.6pt@C=30pt{   
0 \ar[r]& \lambda^r_{Y,n} \ar[r]\ar[d] & M^r_n/U^1M^r_n \ar[r]^-{f}\ar[d] ^-{\cong}& \nu^{r-1}_{Y,n} \ar[d]^-{g}\ar[r] &0\\
0 \ar[r]& \Ker(\theta) \ar[r] & M^r_n/U^1M^r_n \ar[r]^-{\theta} &  \nu^{r-1}_{Y,n}/ \nu^{r-1}_{Y|D_s,n}\ar[r]&0.
 }   
\]
Here we define $\theta:=g\circ f$. The upper horizontal exact row is from Theorem 3.4.2 in \cite{Sa2}. By this diagram, we have a short exact sequence
 \[
\xymatrix@M=2pt@C=30pt{   
0 \ar[r]& \lambda^r_{Y.n} \ar[r]& \Ker(\theta) \ar[r] & \nu^{r-1}_{Y|D_s,n} \ar[r] &0. 
 }   
\]
We consider the following commutative diagram of exact sequences:
 \[
\xymatrix@M=3pt@C=20pt{   
&0\ar[d]&0\ar[d]\\
& \Ker(\tau) \ar[r]\ar[d]&  \lambda^{r}_{Y,n}\ar[d]\\
0 \ar[r]& \Ker(\theta) \ar[r]\ar[d]^-{\tau} & M^r_n/U^1M^r_n \ar[r]^-{\theta}\ar[d]^-{f} & \nu^{r-1}_{Y,n}/ \nu^{r-1}_{Y|D_s,n}\ar[r]\ar[d]^-{f'}&0\\
0 \ar[r]&\mathcal{K}^{r-1}_{Y|D_s} \ar[r] \ar[d]& \nu^{r-1}_{Y,n}\ar[r] \ar[d]& \beta'_*\nu^{r-1}_{D_s,n} \ar[r]&0\\
&0&0.&&
 }   
\]
The map $f'$ is defined as the composition $ \nu^{r-1}_{Y,n}/ \nu^{r-1}_{Y|D_s,n}\longrightarrow  \nu^{r-1}_{Y,n}/\mathcal{K}_{Y|D_s}^{r-1} \cong  \beta'_*\nu^{r-1}_{D_s,n}$. The well-definedness of this map is from $(\bigstar)$. By the fact that $\Ker(\tau)\cong\lambda^{r}_{Y,n}$, we have $\mathcal{K}_{Y|D_s}^{r-1}= \nu^{r-1}_{Y|D_s,n}$.  
\end{proof}
We will investigate the sheaf $\nu_{Y|D_s,n}^{r-1}$ by embedding into a differential module of a smooth variety as \S2.5 in \cite{Sa1}.
\begin{Def}\label{embed} We have a composite map: 
\[(*1)\quad \nu_{Y|D_s,1}^{r} \hookrightarrow \tilde{j}_*\tilde{j}^{-1}\nu_{Y|D_s,1}^{r}=\tilde{j}_*\Omega^r_{\tilde{Y}|\tilde{D}_s, \log} \hookrightarrow \tilde{j}_*\Omega^r_{\tilde{Y}|\tilde{D}_s}.\]By this embedding map $(*1)$, we define the \'etale sheaf  $\Xi^r_{Y|D_s}$ on $Y$ as the $\mathscr{O}_Y$-submodule of $\tilde{j}_*\Omega^r_{\tilde{Y}|\tilde{D}_s}$ generated by local sections of $\nu_{Y|D_s,1}^{r}$. When $D_s=\emptyset$, the \'etale sheaf $\Xi^r_{Y|\emptyset}$ is the same as Sato's differential module $\Xi^r_{Y}$ of Definiton 2.5.1 in \cite{Sa1}.
\end{Def} 
 \begin{Lem}\label{shex}(cf. Lemma 2.4.6, \cite{Sa1}) We define \[\mathcal{Z}^{r+1}_{\mathscr{Y}|\mathscr{D}}(\log Y):=\Ker\left(d: \Omega^{r+1}_{\mathscr{Y}|\mathscr{D}}(\log Y) \rightarrow \Omega^{r+2}_{\mathscr{Y}|\mathscr{D}}(\log Y)\right).\] Then there exists a short exact sequence on $\mathscr{Y}_{\acute{e}t}$:
 \[0 \longrightarrow j_*\Omega^{r+1}_{\mathscr{U}|(\mathscr{D}\backslash D_s), \log} \longrightarrow \mathcal{Z}^{r+1}_{\mathscr{Y}|\mathscr{D}}(\log Y) \xrightarrow{1-C} \Omega^{r+1}_{\mathscr{Y}|\mathscr{D}}(\log Y) \longrightarrow 0. \] 
 Here we put $\mathscr{U}:=\mathscr{Y}\backslash Y$.
 \end{Lem}
 \begin{proof} We have a commutative diagram with exact sequences:
 \[
\xymatrix@M=5pt@C=20pt{ 
&0\ar[d]&0\ar[d]&0\ar[d]&\\
0 \ar[r]& \Ker(1-C) \ar[r]\ar[d] & \mathcal{Z}^{r+1}_{\mathscr{Y}|\mathscr{D}}(\log Y)\ar[d]\ar[r]^-{1-C} &\Omega^{r+1}_{\mathscr{Y}|\mathscr{D}}(\log Y)\ar[r]\ar[d] &0\\
0\ar[r]&j_*\Omega^{r+1}_{\mathscr{U}, \log}\ar[d] \ar[r]&  \mathcal{Z}^{r+1}_{\mathscr{Y}}(\log Y)\ar[r]^-{1-C} \ar[d]&\Omega^{r+1}_{\mathscr{Y}}(\log Y)\ar[r]\ar[d]&0\\
0\ar[r]&\beta_*j_{D*}\Omega^{r+1}_{\mathscr{D}\backslash D_s, \log} \ar[r]\ar[d]& \beta_* \mathcal{Z}^{r+1}_{\mathscr{D}}(\log D_s)\ar[r]^-{1-C} \ar[d]&\beta_*\Omega^{r+1}_{\mathscr{D}}(\log D_s)\ar[r]\ar[d]&0\\
&0&0&0&}
\]
Here the two lower rows are from Lemma 2.4.6 in \cite{Sa1}. Since $\beta_*j_{D*}\Omega^{r+1}_{\mathscr{D}\backslash D_s, \log}\cong j_{*}\beta'_*\Omega^{r+1}_{\mathscr{D}\backslash D_s, \log}$, we have $ \Ker(1-C) \cong  j_*\Omega^{r+1}_{\mathscr{U}|(\mathscr{D}\backslash D_s), \log} $. Thus we have the assertion. \end{proof}

 \begin{Prop}\label{Poincare}(cf. \cite{Sa1}, Lemma 2.5.3)
 Let $Y$ be a simple and embedded into a smooth variety $\mathscr{Y}$ over $\Spec(k)$ as a simple normal crossing divisor.
 Let $\iota: Y \hookrightarrow \mathscr{Y}$ the closed immersion. Then we have an $\iota^{-1}\mathscr{O}_{\mathscr{Y}}$-linear isomorphism
 \[(\P)\quad\iota^{-1}\left(\Omega^{n+1}_{\mathscr{Y}|\mathscr{D}}(\log Y)/ \Omega^{n+1}_{\mathscr{Y}|\mathscr{D}}\right) \xrightarrow{\cong} \Xi^n_{Y|D_s}\]
 which induced by the Poincar\'e residue map
 \[\mathfrak{P}_{Y|D_s}: \Omega^{n+1}_{\mathscr{Y}|\mathscr{D}}(\log Y) \longrightarrow \iota_*j_{2*}\Omega^n_{Y|D_s}.\]
 
 \end{Prop}
 \begin{proof}We prove this Proposition by the same way as Lemma 2.5.3 in \cite{Sa1}.  We have a short exact sequences
\[0 \longrightarrow \Omega^r_{\mathscr{Y}, \log}\longrightarrow  \Omega^r_{(\mathscr{Y}, Y), \log}\longrightarrow \iota_*\nu^{r-1}_{Y,1} \longrightarrow 0,  \]
\[0 \longrightarrow \beta_*\Omega^r_{\mathscr{D}, \log}\longrightarrow  \Omega^r_{(\mathscr{D}, D_s), \log}\longrightarrow \iota_*\beta'_*\nu^{r-1}_{D_s,1} \longrightarrow 0.\]
by \cite{Sa3}, p.200. Then we have a short exact sequence
\[(*1)\quad\quad 0 \longrightarrow \Omega^r_{\mathscr{Y}|\mathscr{D}, \log}\longrightarrow  j_*\Omega^r_{\mathscr{U}|(\mathscr{D}\backslash D_s), \log}\xrightarrow{\mathfrak{r}} \iota_*\nu^{r-1}_{Y|D_s,1} \longrightarrow 0. \]
Here $j:\mathscr{U}:=\mathscr{Y}\backslash Y \hookrightarrow \mathscr{Y}$. We prove that  $\Imm(\mathfrak{P}_{Y|D_s})=\iota_*\Xi^n_{Y|D_s}$. We consider the following  anti-commutative diagram which considered in \cite{Sa1} under no modulus condition($D=\emptyset$):
\[
\xymatrix@M=5pt@C=50pt{ 
\iota^{-1}\mathscr{O}_{\mathscr{Y}} \otimes \iota^{-1} j_*\Omega^r_{\mathscr{U}|(\mathscr{D}\backslash D_s), \log} \ar[r]^-{\mathfrak{R}}\ar[d]& \mathscr{O}_Y \otimes \nu^{r-1}_{Y|D_s,1}\ar[d]\\
\iota^{-1}\Omega^{r+1}_{\mathscr{Y}|\mathscr{D}}(\log Y)\ar[r]^-{\iota^{-1}(\mathfrak{P}_{Y|D_s})}&j_{2*}\Omega^r_{Y|D_s}.
} 
\]
We denote by $\overline{\alpha}$ the residue class of $\alpha \in \iota^{-1}\mathscr{O}_{\mathscr{Y}}$. The upper horizontal map $\mathfrak{R}$ is defined by sending a local section $\alpha \otimes \beta$ to $\overline{\alpha}\otimes \mathfrak{r}(\beta)$, which is surjective.  Lemma \ref{shex} induces the left vertical map. The right vertical map is a product map. The image of this product map is $\Xi^n_{Y|D_s}$. By the surjectivity of the top horizontal map and left vertical map, we have $\Imm(\mathfrak{P}_{Y|D_s})\cong \iota_*\Xi^n_{Y|D_s}$. One can show that the invectivity of the map $(\P)$ by the same way as Lemma 2.5.3 in \cite{Sa1}. \end{proof}

We next introduce a differential sheaf which is defined for a modulus pair $(X,D)$.
\begin{Def}
We define 
\[\lambda_{Y|D_s,n}^r:=\Imm\left((1+\mathscr{O}_Y(-D_s))^{\times} \otimes (\alpha'_*\mathscr{O}_V^{\times})^{\otimes r-1} \xrightarrow{d\log}\displaystyle{ \bigoplus_{y \in V^0}} W_n\Omega_{y, \log}^{r} \right).\] When $D=\emptyset$, we have $\lambda^r_{Y|\emptyset,n}=\lambda^r_{Y,n}$, where $\lambda^r_{Y,n}$ is defined in \cite{Sa1}. We denote by $\mathbb{G}_{m, Y|D_s}^{\otimes r}:=(1+\mathscr{O}_Y(-D_s))^{\times} \otimes (\alpha'_*\mathscr{O}_V^{\times})^{\otimes r-1}$ for simplycity.
\end{Def}
We first consider the structure of $\lambda^r_{Y|D_s}$. We apply Lemma 3.2.2 in \cite{Sa1} to the sheaf $\mathcal{F}=W_n\Omega^r_{\bullet|(\bullet \times D_s), \log}$, $Z=\mathscr{Y}$, we obtain an exact sequence by using the fact that $W_n\Omega^r_{Y^{(q),\log}}=0$ for $q>d+1-r$:
\begin{align*}(\mathfrak{V}1):\;\iota^{-1}W_n\Omega_{\mathscr{Y}|\mathscr{D},\log}^r \longrightarrow &a_{1*}W_n\Omega^r_{Y^{(1)}|D^{(1)}_s,\log} \longrightarrow\\
& \cdots \longrightarrow a_{d+1-r*}W_n\Omega^r_{Y^{(d+1-r)}|D^{(d+1-r)}_s,\log}\longrightarrow 0.\end{align*}
Here $D_s^{(q)}$ is defined by $D_s^{(q)}:=D_s\times_YY^{(q)}$. See \S1.8 of \cite{Sa1} about the notation $Y^{(q)}$\;($q \geq 1$).
There is a commutative diagram:
\[
\xymatrix@M=5pt@C=50pt{ \iota^{-1}\mathbb{G}_{m,\mathscr{Y}|\mathscr{D}}^{\otimes r} \ar[r]\ar[d]^-{d\log} &\mathbb{G}_{m, Y|D_s}^{\otimes r}\ar[d]^-{d\log}\\
\iota^{-1}W_n\Omega_{\mathscr{Y}|\mathscr{D},\log}^r\ar[r]&a_{1*}W_n\Omega^r_{Y^{(1)}|D^{(1)}_s,\log}.
} 
\]
Here the upper horizontal map is surjective by Proposition 2.10 in \cite{RS}. The left vertical map are surjective by Theorem 1.1.5 in \cite{JSZ}; the left vertical arrow is a composition of maps
\[\iota^{-1}\mathbb{G}_{m,\mathscr{Y}|\mathscr{D}}^{\otimes r} \twoheadrightarrow \iota^{-1}\mathcal{K}^M_{r, \mathscr{Y}|\mathscr{D}}/\left(p^m\mathcal{K}^M_{r, \mathscr{Y}}\cap \mathcal{K}^M_{r, \mathscr{Y}|\mathscr{D}}\right)\xrightarrow{\cong} \iota^{-1}W_n\Omega_{\mathscr{Y}|\mathscr{D},\log}^r\]\;(the last isomorphism is due to Theorem 1.1.5 in \cite{JSZ}). By the definition of $\lambda^r_{Y|D_s,n}$, we have
 \[\lambda^r_{Y|D_s,n}=\Imm\left(\iota^{-1}W_n\Omega_{\mathscr{Y}|\mathscr{D},\log}^r \longrightarrow a_{1*}W_n\Omega^r_{Y^{(1)}|D^{(1)}_s,\log}\right).\]
 Then we have the following exact sequence by $(\mathfrak{V}1)$:
\begin{align*}(\mathfrak{V}2):\;0\longrightarrow \lambda^r_{Y|D_s,n}\longrightarrow & a_{1*}W_n\Omega^r_{Y^{(1)}|D^{(1)}_s,\log} \longrightarrow \cdots\\
& \longrightarrow a_{d+1-r*}W_n\Omega^r_{Y^{(d+1-r)}|D^{(d+1-r)}_s,\log}\longrightarrow 0.\end{align*}Taking $D=\emptyset$ in $(\mathfrak{V}2)$, we have Proposition 3.2.1 in \cite{Sa1}.
  \begin{Def}\label{embed;lambda} We have a composite map: 
\[(*1)\quad \lambda_{Y|D_s,1}^{r} \hookrightarrow \tilde{j}_*\tilde{j}^{-1}\lambda_{Y|D_s,1}^{r}=\tilde{j}_*\Omega^r_{\tilde{Y}|\tilde{D}_s, \log} \hookrightarrow \tilde{j}_*\Omega^r_{\tilde{Y}|\tilde{D}_s}.\]By this embedding map $(*1)$, we define the \'etale sheaf  $\Lambda^r_{Y|D_s}$ on $Y$ as the $\mathscr{O}_Y$-submodule of $\tilde{j}_*\Omega^r_{\tilde{Y}|\tilde{D}_s}$ generated by local sections of $\lambda_{Y|D_s,1}^{r}$. When $D_s=\emptyset$, the \'etale sheaf $\Lambda^r_{Y|\emptyset}$ is the same as Sato's differential module $\Lambda^r_{Y}$ of Definiton 3.3.1 in \cite{Sa1}.
\end{Def}

\begin{Lem}\label{Lem1}(cf.\;\cite{JSZ}, Theorem 1.1.6, \cite{Sa1}, Corollary 2.2.5 (2))
There are short exact sequences \begin{equation}\label{Eq1} 0 \longrightarrow \nu_{Y|[D_s/p],n-1}^{r-1} \xrightarrow{\underline{p}}  \nu_{Y|D_s,n}^{r-1}\xrightarrow{\mathcal{R}^{n-1}}   \nu_{Y|D_s,1}^{r-1} \longrightarrow 0, \end{equation}
\begin{equation}\label{Eq2} 0 \longrightarrow \lambda_{Y|[D_s/p],n-1}^{r-1} \xrightarrow{\underline{p}}  \lambda_{Y|D_s,n}^{r-1}\xrightarrow{\mathcal{R}^{n-1}}   \lambda_{Y|D_s,1}^{r-1} \longrightarrow 0 \end{equation}
Here $\mathcal{R}$ is the natural projection operator and $\underline{p}$ is the unique map such that $\underline{p}\circ \mathcal{R}=p$\;($p$ denotes a multiplication by $p$) and we put $[D_s/p]:=\sum_{\lambda \in \Lambda}[n_{\lambda}/p](D_s)_{\lambda}$ with $[n_{\lambda}/p]:=\min\{n\in\mathbb{Z}\mid pn\geq n_{\lambda} \}$\;($D_s:=\sum_{\lambda \in \Lambda}n_{\lambda}(D_s)_{\lambda}$)(This notation is the same as \cite{JSZ}).
\end{Lem}
\begin{proof}  
We first show that the short exact sequence $(\ref{Eq1})$. There is a commutative diagram
 \[
\xymatrix@M=2pt@C=40pt{ 
&0\ar[d]&0\ar[d]&0\ar[d]&\\
0 \ar[r]& W_{n-1}\Omega^r_{\mathscr{Y}|[\mathscr{D}/p],\log} \ar[r]^-{\underline{p}}\ar[d] &W_{n}\Omega^r_{\mathscr{Y}|\mathscr{D},\log}\ar[d]\ar[r]^-{\mathcal{R}^{r-1}} &W_1\Omega^r_{\mathscr{Y}|\mathscr{D},\log}\ar[r]\ar[d] &0\\
0\ar[r]&W_{n-1}\Omega^r_{\mathscr{Y},\log}\ar[d] \ar[r]^-{\underline{p}}& W_{n}\Omega^r_{\mathscr{Y},\log}\ar[r]^-{\mathcal{R}^{r-1}} \ar[d]& W_{1}\Omega^r_{\mathscr{Y},\log}\ar[r]\ar[d]&0\\
0\ar[r]& W_{n-1}\Omega^r_{[\mathscr{D}/p],\log} \ar[r]^-{\underline{p}}\ar[d]& W_{n}\Omega^r_{\mathscr{D},\log}\ar[r]^-{\mathcal{R}^{r-1}} \ar[d]& W_{1}\Omega^r_{\mathscr{D},\log}\ar[r]\ar[d]&0\\
&0&0&0.&}
\]
Here the upper horizontal row is given by \cite{JSZ}, Theorem 1.1.6, and the middle horizontal row is given by \cite{CTSS}, p.779, Lemma 3. Then the lower horizontal row is exact.
By the definition of $\nu^{r}_{D_s,m}$, we obtain a short exact sequence:
\[0 \longrightarrow \nu^{r}_{[D_s/p],n-1} \xrightarrow{\underline{p}} \nu^{r}_{D_s,n} \xrightarrow{\mathcal{R}^{r-1}} \nu^{r}_{D_s,1} \longrightarrow 0. \] 
We consider the following commutative diagram:
 \[
\xymatrix@M=3pt@C=40pt{ 
&0\ar[d]&0\ar[d]&0\ar[d]&\\
0 \ar[r]&\nu^r_{Y|[D_s/p],n-1} \ar[r]^-{\underline{p}}\ar[d] &\nu^r_{Y|D_s,n} \ar[d]\ar[r]^-{\mathcal{R}^{r-1}} & \nu^{r}_{Y|D_s,1}\ar[r]\ar[d] &0\\
0\ar[r]& \nu^{r}_{Y,n-1}\ar[d] \ar[r]^-{\underline{p}}& \nu^{r}_{Y,n} \ar[r]^-{\mathcal{R}^{r-1}} \ar[d]&  \nu^{r}_{Y,1}\ar[r]\ar[d]&0\\
0\ar[r]&  \nu^{r}_{[D_s/p],n-1} \ar[r]^-{\underline{p}}\ar[d]&  \nu^{r}_{D_s,n} \ar[r]^-{\mathcal{R}^{r-1}} \ar[d]&  \nu^{r}_{D_s,1}\ar[r]\ar[d]&0\\
&0&0&0,&}
\]
where the exactness of the middle horizontal row is given by \cite{Sa1}, Corollary 2.2.5 (2), and the vertical rows are exact by Proposition \ref{nu}.
Thus the upper horizontal row is an exact.  The exactness of $(\ref{Eq2})$ is given by  \cite{JSZ}, Theorem 1.1.6 and the exact sequence $(\mathfrak{V}2)$ of $\lambda^r_{Y|D_s,n}$.
This completes the proof of the assertion.\end{proof}

 We have the following duality which is a modulus version of Proposition 3.3.5 in \cite{Sa1}.
\begin{Prop}\label{dual}(cf. Proposition 3.3.5, \cite{Sa1})
Assume that $0 \leq m \leq d$. Let 
\[\Xi^m_{Y|D_s} \times \Lambda^{d-m}_{Y|-D_s} \longrightarrow \Xi^d_{Y}\quad\hbox{and}\quad \Lambda^m_{Y|D_s} \times \Xi^{d-m}_{Y|-D_s} \longrightarrow \Xi^d_{Y}\]
are the $\mathscr{O}_Y$-bilinear pairing obtained by the pairing $\nu_{Y|mD_s,1}^r \otimes \alpha'_*\lambda^{d-r}_{V,1} \longrightarrow \nu_{Y,1}^d$ which is defined below. Here $\Xi^r_{Y|-D_s}:=\Xi^r_{Y}\otimes \mathscr{O}_Y(D_s)$ and $\Lambda^r_{Y|-D_s}:=\Lambda^r_{Y}\otimes \mathscr{O}_Y(D_s)$.
Then the induced morphisms
\[(a1):\;\Lambda^{d-m}_{Y|-D_s} \longrightarrow {\rm R}\mathcal{H}om_{\mathscr{O}_Y}(\Xi^m_{Y|D_s},\; \Xi^d_{Y})  \]
\[(a2):\;\Xi^{d-m}_{Y|-D_s} \longrightarrow {\rm R}\mathcal{H}om_{\mathscr{O}_Y}(\Lambda^m_{Y|D_s},\; \Xi^d_{Y})  \]
are isomorphisms.
\end{Prop}
\begin{proof} We use the same argument of Proposition 3.3.5 in \cite{Sa1}. This problem is \'etale local on $Y$, we may suppose that $Y$ is embedded into a smooth variety $\mathscr{Y}$ over $s$ as a simple normal crossing divisor.
  We put $\iota: Y \hookrightarrow \mathscr{Y}$, $f: Y \rightarrow s$ and $h: \mathscr{Y} \rightarrow s$. Let $\mathscr{I}_Y$ be the definition ideal of $Y$. 
  We have a short exact sequence \[0 \longrightarrow \Omega_{\mathscr{Y}|-\mathscr{D}}^m(-\log Y) \longrightarrow\  \Omega_{\mathscr{Y}|-\mathscr{D}}^m \longrightarrow \iota_*\Lambda^m_{Y|-D_s} \longrightarrow 0\]
 by Theorem in \cite{Sa1}.  
 Here we put $\Omega_{\mathscr{Y}|-\mathscr{D}}^m(-\log Y):=\Omega_{\mathscr{Y}|-\mathscr{D}}^m(\log Y) \otimes_{\mathscr{O}_{\mathscr{Y}}} \mathscr{I}_Y$. 
 There are the $\mathscr{O}_{\mathscr{Y}}$-perfect pairings of locally free $\mathscr{O}_{\mathscr{Y}}$-modules
 \[\Omega^{m+1}_{\mathscr{Y}|\mathscr{D}}\times \Omega^{d-m}_{\mathscr{Y}|-\mathscr{D}} \longrightarrow \Omega^{d+1}_{\mathscr{Y}},\]
  \[\Omega^{m+1}_{\mathscr{Y}|\mathscr{D}}(\log Y)\times \Omega^{d-m}_{\mathscr{Y}|-\mathscr{D}}(-\log Y) \longrightarrow \Omega^{d+1}_{\mathscr{Y}}.\]
  We have an isomorphismsc
  \begin{align*}
  &R\iota_*{\rm R}\mathcal{H}om_{\mathscr{O}_Y}(\Xi^m_{Y|D_s},\; \Xi^d_{Y})\\
  & \cong  R\iota_*{\rm R}\mathcal{H}om_{\mathscr{O}_Y}(\Xi^m_{Y|D_s},\; f^{H}\mathscr{O}_s[-d]) \\
  & =R\iota_*{\rm R}\mathcal{H}om_{\mathscr{O}_Y}(\Xi^m_{Y|D_s},\; \iota^{H}h^{H}\mathscr{O}_s[-d])\\
  & = {\rm R}\mathcal{H}om_{\mathscr{O}_{\mathscr{Y}}}(\iota_*\Xi^m_{Y|D_s},\; h^{H}\mathscr{O}_s[-d])\\
  &\cong {\rm R}\mathcal{H}om_{\mathscr{O}_{\mathscr{Y}}}(\Omega^{m+1}_{\mathscr{Y}|\mathscr{D}}(\log Y)/\Omega^{m+1}_{\mathscr{Y}|\mathscr{D}},\; \Omega^{d+1}_{\mathscr{Y}})[1]\\
  &\cong \Omega^{m+1}_{\mathscr{Y}|-\mathscr{D}}/\Omega^{m+1}_{\mathscr{Y}|-\mathscr{D}}(-\log Y)=\iota_*(\Lambda^{m}_Y\otimes \mathscr{O}_{Y}(D_s)).
  \end{align*}
  Here we use a natural isomorphism (3.3.6) in the proof of Proposition 3.3.5 in \cite {Sa1} and projection formula\;(see Exercise 5.1 (d) in \cite{Har}) of $\iota_*$ for last equality. Then we obtain an isomorphism $(a1)$.
 We have an isomorphism $(a2)$ by the same argument. This completes the proof.
\end{proof}

The differential operator $d: \Omega^r_Y \rightarrow \Omega^{r+1}_Y$ induces a differential operators $d^r:\Xi^r_{Y|D_s,n} \rightarrow \Xi^{r+1}_{Y|D_s,n}$ and $d^r:\Lambda^r_{Y|D_s,n} \rightarrow \Lambda^{r+1}_{Y|D_s,n}$(cf.\;Corollary 2.5.6 and Lemma 3.3.2 in \cite{Sa1}). Then we can consider $Z\Xi^r_{Y|D_s}:=\Ker(d^r)$, $Z\Lambda^r_{Y|D_s}:=\Ker(d^r)$, $B\Xi^r_{Y|D_s}:=\Imm(d^{r-1})$, and $B\Lambda^r_{Y|D_s}:=\Imm(d^{r-1})$.

\begin{Lem}\label{Z/B}(cf.\;\cite{Sa1}, Lemma 2.5.7 and Lemma 3.3.4)\\
We have a Cartier isomorphisms
\[\hbox{{\rm (1)}}\quad C: Z\Xi^r_{Y|D_s}/B\Xi^r_{Y|D_s} \xrightarrow{\cong} \Xi^r_{Y|[D_s/p]}, \]
\[\hbox{{\rm (2)}}\quad C: Z\Lambda^r_{Y|D_s}/B\Lambda^r_{Y|D_s} \xrightarrow{\cong} \Lambda^r_{Y|[D_s/p]}\]
which are induced by the Cartier isomorphism $Z^n_Y/B^n_Y \cong \Omega^n_{Y}$.
\end{Lem}
\begin{proof}
The proofs are essentially same of Lemma 2.5.7 and Lemma 3.3.4 in \cite{Sa1}. We first show $(1)$. We assume that $Y$ is simple (if any irreducible component is smooth over $k$; this is the terminology of \cite{Sa1}) and embedded into a smooth variety $\mathscr{Y}/k$ as a simple normal crossing divisor since the problem is \'etale local. We put $\iota$ the closed immersion $\iota:Y \hookrightarrow \mathscr{Y}$.  We have an exact sequence of complexes
\[0 \longrightarrow \Omega^{\bullet}_{\mathscr{Y}|\mathscr{D}} \longrightarrow  \Omega^{\bullet}_{\mathscr{Y}|\mathscr{D}}(\log Y)\xrightarrow{\mathfrak{P}_{Y|D_s}} \iota_*\Xi^{\bullet-1}_{Y|D_s}\longrightarrow 0 \]
by Proposition 4.5. Then we taking the cohomology of this sequence, we get the following short exact sequence:
\[0\rightarrow\mathcal{H}^{q+1}\left(\Omega^{\bullet}_{\mathscr{Y}|\mathscr{D}}\right)\rightarrow  \mathcal{H}^{q+1}\left(\Omega^{\bullet}_{\mathscr{Y}|\mathscr{D}}(\log Y)\right)\rightarrow  \iota_*\mathcal{H}^{q}\left(\Xi^{\bullet}_{Y|D_s}\right) \rightarrow 0.\]
There are Cartier isomorphisms\;(see \cite{RS}, Theorem 2.16):
\[C:\mathcal{H}^{q+1}\left(\Omega^{\bullet}_{\mathscr{Y}|\mathscr{D}}\right) \xrightarrow{\cong} \Omega^{q+1}_{\mathscr{Y}|[\mathscr{D}/p]},\;C: \mathcal{H}^{q+1}\left(\Omega^{\bullet}_{\mathscr{Y}|\mathscr{D}}(\log Y)\right)\xrightarrow{\cong}  \Omega^{q+1}_{\mathscr{Y}|[\mathscr{D}/p]}(\log Y).\]
Hence we obtain an isomorphism in (1)\;:$\mathcal{H}^{q+1}\left(\Xi_{Y|D_s}^{\bullet}\right) \xrightarrow{\cong} \Xi_{Y|[D_s/p]}^{q+1}$ because the Poincar\'e residue map $\mathfrak{P}_{Y|D_s}$ is compatible with Cartier operator. The Cartier isomorphism in (2) is reduced to the smooth case by using an exact sequence $(\mathfrak{V}2)$ and Lemma 3.2.2 in \cite{Sa2}. \end{proof}

\begin{Lem}\label{Clin}(cf. \cite{Sa1}, {\rm p.731}) We consider  the following cartesian diagram:
\[
\xymatrix@M=5pt@C=70pt{   
Y' \ar[d]^-{f_1}\ar[r]^-{pr_2} & Y \ar[d]^-{f_2}\\
s\ar[r]^-{F_s} & s.
 }   
\]
Here $F_s$ denotes the absolute Frobenius automorphism of $s$. For $q \geq 0$, we have two isomorphisms:
\[(\mathfrak{C}1)\;C^D_{{\rm lin}}:F_{Y/s*}\left(Z\Xi^q_{Y|D_s}/B\Xi^q_{Y|D_s}\right) \xrightarrow{\cong} \Xi^q_{Y'|[D'_s/p]}, \]
\[(\mathfrak{C}2)\;C^D_{{\rm lin}}:F_{Y/s*}\left(Z\Lambda^q_{Y|D_s}/B\Lambda^q_{Y|D_s}\right) \xrightarrow{\cong} \Lambda^q_{Y'|[D'_s/p]}.\]
Here we define $D_s':=D_s \times_{Y} Y'$ and let $F_{Y/s*}:Y\rightarrow Y'$ be the unique finite morphism that the absolute Frobenius map $F_Y:Y\rightarrow Y$ factors through(see \cite{Sa1}). If $D_s=\emptyset$, we have an isomorphisms $(3.4.6)$ and $(3.4.7)$ in \cite{Sa1}, p.731.
\end{Lem}
\begin{proof}Let $F_{\mathscr{Y}/s}:\mathscr{Y} \rightarrow \mathscr{Y}'$ be a relative Frobenius morpshim. We have a short exact sequence by Proposition \ref{Poincare}
\[0 \longrightarrow F_{\mathscr{Y}/s*}\Omega^{\bullet}_{\mathscr{Y}|\mathscr{D}} \longrightarrow F_{\mathscr{Y}/s*}\Omega^{\bullet}_{\mathscr{Y}|\mathscr{D}}(\log Y)\xrightarrow{F_{\mathscr{Y}/s*}\mathfrak{P}_{Y|D_s}} F_{\mathscr{Y}/s*}\iota_*\Xi^{\bullet-1}_{Y|D_s}\longrightarrow 0.\]
We have $ F_{\mathscr{Y}/s*}\iota_*= \iota'_*F_{Y/s*}$, where $\iota'$ is an immersion $\iota':Y' \hookrightarrow \mathscr{Y}'$. 

We have isomorphisms \[C: \Omega^q_{\mathscr{Y}'|[\mathscr{D}'/p]}(\log Y') \xrightarrow{\cong} \mathcal{H}^q\left(F_{\mathscr{Y}/s*}(\Omega^{\bullet}_{\mathscr{Y}|\mathscr{D}}(\log Y))\right)\] and \[C: \Omega^q_{\mathscr{Y}'|[\mathscr{D}'/p]} \xrightarrow{\cong} \mathcal{H}^q\left(F_{\mathscr{Y}/s*}(\Omega^{\bullet}_{\mathscr{Y}|\mathscr{D}})\right)\] by the similar argument in the proof of Theorem 2.16, \cite{RS}. This isomorphisms are induced by the linear Cartier isomorphism given by Katz\;(\cite{Kat}, 7.2). 
Taking cohomology sheaves, we obtain a short exact sequence:
\begin{align*}0 \longrightarrow \mathcal{H}^{q+1}\left(F_{\mathscr{Y}/s*}\Omega^{\bullet}_{\mathscr{Y}|\mathscr{D}}\right) \longrightarrow \mathcal{H}^{q+1}\left(F_{\mathscr{Y}/s*}\Omega^{\bullet}_{\mathscr{Y}|\mathscr{D}}(\log Y)\right)\\ \xrightarrow{F_{\mathscr{Y}/s*}\mathfrak{P}_{Y|D_s}}  \iota'_*\mathcal{H}^{q}\left(F_{Y/s*}\Xi^{\bullet}_{Y|D_s}\right)\longrightarrow 0.\end{align*}

The assertion follows from the fact that the map $\mathfrak{P}_{Y|D_s}$ is compatible with Cartier operators. We obtain the isomorphism $(\mathfrak{C}2)$ by the similar argument as Lemma 3.3.4 in \cite{Sa1}.\end{proof}

The following is the main result of this chapter.
\begin{Thm}\label{PThm}(cf.\;\cite{Sa1}, Theorem 1.2.2 (2))\\
{\rm (1)}\;For any integers $i$ and $n$ with $0 \leq i \leq d$, there exist a natural pairings
\begin{equation}\label{pair1}\plim[m]H^i_{\acute{e}t}(Y, \nu_{Y|mD_s,n}^r) \times H_{\acute{e}t}^{d+1-i}(Y, \alpha'_*\lambda^{d-r}_{V,n})\longrightarrow \mathbb{Z}/p^n\mathbb{Z},\end{equation}
\begin{equation}\label{pair2}\plim[m]H^i_{\acute{e}t}(Y, \lambda_{Y|mD_s,n}^r) \times H_{\acute{e}t}^{d+1-i}(Y, \alpha'_*\nu^{d-r}_{V,n})\longrightarrow \mathbb{Z}/p^n\mathbb{Z}.\end{equation}
{\rm (2)}\; The pairings (\ref{pair1}), (\ref{pair2}) are non-degenerate pairings.
\end{Thm} 
\begin{proof} We first prove {\rm(1)}. It is enough to show that the case $n=1$ by using Lemma \ref{Lem1}. Consider a diagram
\[
\xymatrix@M=4pt@C=50pt{   
\nu_{Y|mD_s,1}^r \otimes \alpha'_*\lambda^{d-r}_{V,1}\ar@{.>}[d]\ar[r]^-{f} &\nu_{Y,1}^{r} \otimes \alpha'_*\lambda^{d-r}_{V,1}\ar[d]^-{g}&\\
\nu_{Y,1}^d\ar[r]&\alpha'_*\nu_{V,1}^d\ar[r]^-{\delta}&\beta'_*\nu_{D,1}^{d-1}\ar[r]&\nu_{Y,1}^d[1],&
 }   
\]
where $g$ is a composition of a map $\nu_{Y,1}^{r} \otimes \alpha'_*\lambda^{d-r}_{V,1} \rightarrow \alpha'_*\nu_{V,1}^{r} \otimes \alpha'_*\lambda^{d-r}_{V,1} \rightarrow \alpha'_*\nu_{V,1}^d$. Here last morphism is product morphism defined in \cite{Sa1}. Since  $\delta\circ g \circ f=0$ and $\Hom\left(\nu_{Y|mD_s,1}^r \otimes \alpha'_*\lambda^{d-r}_{V,1},\; \beta'_*\nu_{D,1}^{d-1}[-1]\right)=0$, we have a morphism $\nu_{Y|mD_s,1}^r \otimes \alpha'_*\lambda^{d-r}_{V,1} \longrightarrow \nu_{Y,1}^d$ by Lemma 2.1.2 (1) in \cite{Sa2}. By taking the cohomology of this morphism, we have a pairing (\ref{pair1}):
$\plim[m]H^i_{\acute{e}t}(Y, \nu_{Y|mD_s,1}^r) \times H_{\acute{e}t}^{d+1-i}(Y, \alpha'_*\lambda^{d-r}_{V,1})\longrightarrow H^{d+1}(Y,  \nu_{Y,1}^d) \xrightarrow{\tr}  \mathbb{Z}/p^n\mathbb{Z}$.
 We next construct the pairing $(\ref{pair2})$ by the similar argument of \S3 in \cite{Sa1}.  Since $\alpha'_*\nu^{d-r}_{V,1} \cong \displaystyle{\lim_{\stackrel{\longrightarrow}{m'}}}\; \nu^{d-r}_{Y,1} \otimes \mathscr{O}_{Y}(m'D_s)$, we define the map $\lambda_{Y|mD_s,n}^r\otimes \{\nu^{d-r}_{Y,1} \otimes \mathscr{O}_{Y}(m'D_s)\}_m\rightarrow \nu_{Y,1}^d$ of sheaves on $Y_{\acute{e}t}$. We consider the map \[\mathfrak{p}:\;\lambda_{Y|mD_s,n}^r\otimes\{ \nu^{d-r}_{Y,1} \otimes \mathscr{O}_{Y}(m'D_s)\}_m\rightarrow {\displaystyle \bigoplus_{y \in V^0}}i_{y*}\Omega^d_{y, \log}\]
  defined by $d\log(f_m) \otimes(\gamma_{m'}\otimes \omega) \mapsto \omega\cdot d\log((1+\gamma_{m'}z)\cdot f'_m)$, where $\gamma_{m'} \in \mathscr{O}_Y(m'D_s), \omega \in i_{y*}\Omega^{d-r}_{y, \log}, f_m=(1+z)\cdot f'_m \in (1+\mathscr{O}_Y(-mD_s))^{\times}\otimes (\alpha'_*\mathscr{O}_V)^{\otimes(r-1)}$. By the equation $\partial^{val}_{y,x}(\omega \cdot d\log(\gamma_{m'}\cdot f_m))=\partial^{val}_{y,x}(\omega)\cdot d\log(\overline{(1+\gamma_{m'}z)\cdot f'_m})$ in $i_{x*}\Omega^{d-1}_{x, \log}$, in Definition 3.1.1 (2),\cite{Sa1}, the image of the map $\mathfrak{p}$ lies in $\nu_{Y,1}^{d}$. Thus we obtain the map $\lambda_{Y|mD_s,n}^r\otimes \{\nu^{d-r}_{Y,1} \otimes \mathscr{O}_{Y}(m'D_s)\}_m\rightarrow \nu_{Y,1}^d$. We next prove {\rm (2)} assuming $n=1$. We show the non-degeneracy of the pairing $(\ref{pair1})$
\[\plim[m]H^i_{\acute{e}t}(Y, \nu_{Y|mD_s,1}^r) \times H_{\acute{e}t}^{d+1-i}(Y, \alpha'_*\lambda^{d-r}_{V,1})\longrightarrow H^{d+1}(Y,  \nu_{Y,1}^d) \xrightarrow{\tr}  \mathbb{Z}/p^n\mathbb{Z}.\]
It suffices to show that the natural two pairings
\begin{equation}\label{A1} \plim[m]H_{\acute{e}t}^i(Y, \Xi^r_{Y|D_s}) \times H_{\acute{e}t}^{d-i}(Y, \alpha'_*\Lambda^{d-r}_{V})\longrightarrow H_{\acute{e}t}^{d}(Y, \Xi_{Y}^d) \xrightarrow{\tr_1}  \mathbb{Z}/p^n\mathbb{Z}
\end{equation}
\begin{equation}\label{A2}\plim[m]H_{\acute{e}t}^i(Y, Z\Xi^r_{Y|D_s}) \times H_{\acute{e}t}^{d-i}(Y, \alpha'_*(\Lambda^{d-r}_{V}/B\Lambda^{d-r}_{V}))\longrightarrow H_{\acute{e}t}^{d}(Y, \Xi_{Y}^d/B\Xi_{Y}^d) \xrightarrow{\tr_2}  \mathbb{Z}/p^n\mathbb{Z}
\end{equation}
are non-degenerate pairings of $\mathbb{Z}/p^n\mathbb{Z}$-modules for any $i \in \mathbb{Z}$ by the following Lemma \ref{RSh} below.  Here $\tr_1$ and $\tr_2$ are defined as follows(see \cite{Sa1} p.731):
\[\tr_1:\;H_{\acute{e}t}^{d}(Y, \Xi_{Y}^d) \xrightarrow{\tr_f} k \xrightarrow{\tr_{k/\mathbb{F}_p}}   \mathbb{Z}/p\mathbb{Z},\quad \tr_2:\;H_{\acute{e}t}^{d}(Y, \Xi_{Y}^d/B\Xi_Y^d)\stackrel{ \xrightarrow{C}}{\cong}  H_{\acute{e}t}^{d}(Y, \Xi_{Y}^d) \xrightarrow{\tr_1}   \mathbb{Z}/p\mathbb{Z}.\]
\begin{Lem}\label{RSh}(cf. \cite{Sa1}, Lemma 2.5.7\;(2),\; Lemma 3.3.4\;(2))\\
(1)\;There exists a short exact sequence on $Y_{\acute{e}t}$
\[0 \longrightarrow \nu_{Y|D_s,1}^{q} \longrightarrow Z\Xi_{Y|D_s}^{q} \xrightarrow{1-C} \Xi_{Y|D_s}^{q} \longrightarrow 0\]
(2)\;There is a short exact sequence on $Y_{\acute{e}t}$
\[0 \longrightarrow \lambda_{Y|D_s,1}^{q} \longrightarrow \Lambda_{Y|D_s}^{q} \xrightarrow{1-C^{-1}} \Xi_{Y|D_s}^{q}/B\Xi_{Y|D_s}^{q} \longrightarrow 0\]
\end{Lem}
\begin{proof}
We first show that the assertion (1). There is a commutative diagram:
\[
\xymatrix@M=8pt@C=30pt{ 
           &0\ar[d] & 0\ar[d] &0\ar[d]\\
0 \ar[r] &\Omega_{\mathscr{Y}|\mathscr{D},\log}^{q+1}\ar[d]\ar[r] &Z^{q+1}_{\mathscr{Y}|\mathscr{D}}\ar[d]\ar[r]^-{1-C}&\Omega_{\mathscr{Y}|\mathscr{D}}^{q+1}\ar[d]\ar[r]&0\\
0 \ar[r] &j_*\Omega_{\mathscr{U}|\mathscr{D}\backslash D_s,\log}^{q+1}\ar[d]\ar[r] &Z^{q+1}_{\mathscr{Y}|\mathscr{D}}(\log Y)\ar[d]\ar[r]^-{1-C}&\Omega_{\mathscr{Y}|\mathscr{D}}^{q+1}(\log Y)\ar[d]\ar[r]&0\\
0 \ar[r] &\iota_*\nu_{Y|D_s,1}^{q}\ar[r]\ar[d] &\iota_*Z\Xi^{q}_{Y|D_s}\ar[d]\ar[r]^-{1-C}&\iota_*\Xi^{q+1}_{Y|D_s}\ar[d]\ar[r]&0\\
&0 & 0 &0
 }   
\]
Here the left vertical row is due to $(*1)$ in the proof of Proposition \ref{Poincare}, the middle and right vertical rows follows from an isomorphism $(\P)$ of Proposition \ref{Poincare}. The top and middle horizontal rows are due to Lemma \ref{shex}. Thus we obtain the lower horizontal row, which is an exact. We next prove the assertion (2). Since the problem are \'etale local on $Y$, we may assume that $Y$ is simple. By the definition of $\lambda^r_{Y|D_s}$ and $\Lambda_{Y|D_s}^r$, we have $\Lambda_{Y|D_s}^r=\Imm\left(\iota^{-1}\Omega_{\mathscr{Y}|\mathscr{D}}\xrightarrow{\iota^*} a_{1*}\Omega_{Y^{(1)}|D_s^{(1)}}\right)$. Here the map $\iota^*$ is the pull-back map of differential forms. Then we get exact sequences by similar argument as the proof of Lemma 3.3.4 in \cite{Sa1}:
\begin{align*}
&0 \rightarrow \Lambda^r_{Y|D_s} \xrightarrow{a_1^*} a_{1*}\Omega^n_{Y^{(1)}|D_s^{(1)}}\rightarrow \cdots \rightarrow a_{d+1-r*}\Omega^r_{Y^{(d+1-r)}|D_s^{(d+1-r)}}\rightarrow 0\\
&0 \rightarrow Z\Lambda^r_{Y|D_s} \xrightarrow{a_1^*} a_{1*}\mathcal{Z}^n_{Y^{(1)}|D_s^{(1)}}\rightarrow \cdots \rightarrow a_{d+1-r*}\mathcal{Z}^r_{Y^{(d+1-r)}|D_s^{(d+1-r)}}\rightarrow 0\\
&0 \rightarrow B\Lambda^r_{Y|D_s} \xrightarrow{a_1^*} a_{1*}\mathcal{B}^n_{Y^{(1)}|D_s^{(1)}}\rightarrow \cdots \rightarrow a_{d+1-r*}\mathcal{B}^r_{Y^{(d+1-r)}|D_s^{(d+1-r)}}\rightarrow 0
\end{align*}
By this exact sequences, we can reduce the assertion to the smooth case and  exact sequence $(\mathfrak{V}_2)$. The smooth case of the assertion due to Theorem 1.2.1 of \cite{JSZ}. This completes the proof. 
\end{proof}
We return to the proof of the non-degeneracy of the pairings $(\ref{A1})$ and $(\ref{A2})$.
The non-degeneracy of the pairing $(\ref{A1})$ follows from Proposition \ref{dual}, VII. 3, \cite{RD}, and an isomorphism $\alpha'_*\Lambda^{d-r}_{V} \cong \displaystyle{\lim_{\stackrel{\longrightarrow}{m'}}}\; \Lambda^{d-r}_{Y,1} \otimes \mathscr{O}_{Y}(m'D_s)$ because $\Lambda^{d-r}_{Y,1}$ is a coherent sheaf\;(see \cite{Sa1}, Lemma 3.3.2 (1)). We consider the non-degeneracy of $(\ref{A2})$. We can decompose the isomorphism $C$ into isomorphisms \begin{align*}Z\Xi^r_{Y|D_s}/B\Xi^r_{Y|D_s}\xrightarrow{F^{-1}_{Y/s}(C^D_{\lin})} F^{-1}_{Y/s}(\Xi^r_{Y'|[D'_s/p]}) \xrightarrow{(\pr_2^*)^{-1}} F_{Y/s}^{-1}\pr_2^{-1}(\Xi^r_{Y|[D_s/p]})\\=F_Y^{-1}\Xi^r_{Y|[D_s/p]}=\Xi^r_{Y|[D_s/p]}\end{align*} as $(3.4.8)$ in \cite{Sa1}. The pairing of pro-system with respect to $m$:
\[F_{Y/s*}(Z\Xi^r_{Y|mD_s})\times F_{Y/s*}(\Lambda^{d-r}_{Y|-mD_s}/B\Lambda^{d-r}_{Y|-mD_s})\rightarrow F_{Y/s*}(\Xi^d_{Y}/B\Xi^d_{Y})\xrightarrow{C_{\lin}}\Xi^d_{Y'}\]
induces an morphism \[\{F_{Y/s*}(\Lambda^{d-r}_{Y|-mD_s}/B\Lambda^{d-r}_{Y|-mD_s})\}_m\rightarrow \{R\mathscr{H}om_{\mathscr{O}_{Y'}}(F_{Y/s*}(Z\Xi^r_{Y|mD_s}), \Xi^d_{Y'})\}_m.\]
We have an isomorphism as a pro-system with respect to $m$:
\begin{align*}
\{F_{Y/s*}(Z\Lambda^{d-r}_{Y|-mD_s}/&B\Lambda^{d-r}_{Y|-mD_s})\}_m\cong \{\Lambda^{d-r}_{Y'|-pmD_s}\}_m\\
&\cong \{R\mathscr{H}om_{\mathscr{O}_{Y'}}(\Xi^{r}_{Y'|pmD_s},\; \Xi^d_{Y'})\}_m\\
&\cong \{R\mathscr{H}om_{\mathscr{O}_{Y'}}(F_{Y/s*}(Z\Xi^r_{Y|p^2mD_s}/B\Xi^r_{Y|p^2mD_s}),\; \Xi^d_{Y'})\}_m\\
&\cong \{R\mathscr{H}om_{\mathscr{O}_{Y'}}(F_{Y/s*}(Z\Xi^r_{Y|mD_s}/B\Xi^r_{Y|mD_s}),\; \Xi^d_{Y'})\}_m.
\end{align*}
Here the second isomorphism is due to Proposition \ref{dual}. For the third isomorphism, we use Lemma \ref{Clin}.
Thus, by using this fact  and an isomorphism $\{F_{Y/s*}(\Lambda^{d-r}_{Y|mD_s}/Z\Lambda^{d-r}_{Y|mD_s})\}_m \cong \{F_{Y/s*}B\Lambda^{d-r}_{Y|mD_s}\}_m$, we obtain an isomorphism of pro-systems \[\{F_{Y/s*}(\Lambda^{d-r}_{Y|-mD_s}/B\Lambda^{d-r}_{Y|-mD_s})\}_m\xrightarrow{\cong} \{R\mathscr{H}om_{\mathscr{O}_{Y'}}(F_{Y/s*}(Z\Xi^r_{Y|mD_s}), \Xi^d_{Y'})\}_m.\] We have a perfect pairing 
\[\plim[m]H^i_{\acute{e}t}(Y,Z \Xi_{Y|mD_s}^r) \times H_{\acute{e}t}^{d+1-i}(Y, \alpha'_*\Lambda^{d-r}_{V}/B\Lambda^{d-r}_{V})\longrightarrow H^{d+1}(Y',  \Xi_{Y'}^d) \xrightarrow{\tr_g} k.\]
The non-degeneracy of the pairing $(\ref{A2})$ is due to the commutative diagram in the proof of Theorem 1.2.2 (2) in \cite{Sa1}, p.732.
The proof of non-degeneracy of the pairing $(\ref{pair2})$ is similar argument as $(\ref{pair1})$'s one. This completes the proof.
\end{proof}
\begin{Rem}
When $Y$ is smooth in Theorem \ref{PThm}, we obtain the Theorem 4.1.4 in \cite{JSZ}. 
\end{Rem}

%%%%%%%%%%%%%%%%%%%%%%%%%%%%%%%%%%%%%%%%%%%%%%%%%%%%%%%%%%%%%%%%%%%%%%%%%%%%%%%%
%%%%%%%%%%%%%%%%%%%%%%%%%%%%%%%%%%%%%%%%%%%%%%%%%%%%%%%%%%%%%%%%%%%%
\section{\textbf{Explicit formula for $\mathscr{M}^r_m$}}
In this section we construct a canonical pairing $(\mathfrak{H})$ below and prove an explicit formula for that pairing, which will be used in the proof of \ref{Main Thm}.
In this section, we use a Sato's arguments of \S8 in \cite{Sa2}. 
\subsection{Setting}
We put $\nu^r_{Y}:=\nu^r_{n,1}$, $\mu':=i^*\psi_*\mu_{p}$ and $\mu:=\mu_p(K)$ for simplicity. Put $\mathscr{M}^q:=\mathscr{M}^q_1=\mathcal{H}^q\big(s_1(q)_{X|D}\big)$, $M^q_Y=i^*R^qj_*\mu_p^{\otimes q}$, $M^q_V=i^{'*}R^qj'_*\mu_p^{\otimes q}$, and let $U^{\bullet}$ be the filtration on $\mathscr{M}^q$ defined in \cite{Sa2}. The purpose of this section is to construct a morphism in\; $D^b(Y_{\acute{e}t}, \mathbb{Z}/p\mathbb{Z})$\;:
\[(\mathfrak{H})\quad\Theta^r_{m,D}:{\{U^1\mathscr{M}^r_m\}}_m\otimes \alpha'_*U^1M^{d-r+1}_V[-d-2]\longrightarrow \mu' \otimes \nu^{d-1}_Y[-d-1]\] and to prove an explicit formula for this morphism (Theorem \ref{ERL} below). 

\subsection{Construction of $\Theta^r_{D}$}

The sheaf $\mu'$ is non-canonically isomorphic to the constant sheaf $\mathbb{Z}/p\mathbb{Z}$, then we will write $\mu'\otimes \mathcal{K}$ ($\mathcal{K} \in  D^b(Y_{\acute{e}t}, \mathbb{Z}/p\mathbb{Z})$) for $\mu' \otimes^{\mathbb{L}} \mathcal{K}$ in $ D^b(Y_{\acute{e}t}, \mathbb{Z}/p\mathbb{Z})$. We consider the following distinguished triangle of pro-system with respect to $m$:
\begin{align*}{\{(\mathscr{M}^r_m/U^1\mathscr{M}^r_m)[-r-1]\}}_m\xrightarrow{g'} {\{\mathbb{B}_m(r)\}}_m\xrightarrow{t'}{\{\tau_{\leq r} s_1(r)_{X|mD}\}}_m\\ \longrightarrow {\{(\mathscr{M}^r_m/U^1\mathscr{M}^r_m)[-r]\}}_m,\end{align*}
where the last morpshim is defined as the composite $\tau_{\leq r} s_1(r)_{X|mD}\rightarrow\mathscr{M}^r_m[-r]\rightarrow (\mathscr{M}^r_m/U^1\mathscr{M}^r_m)[-r]$.
It is easy to see that $\mathbb{B}_m(r)$ is concentrated in $[0,r]$, and the triple $(\mathbb{B}_m(r), g',t')$ is unique up to a unique isomorphism by Lemma . 
We also consider the following distinguished triangle which is induced by \cite{Sa2}:
\begin{multline*} \alpha'_*(M^{d-r+1}_V/U^1M^{d-r+1}_V)[-d+r-2]\xrightarrow{h^{\vee}} \mathbb{B}^{\vee}(d-r+1)\\ \xrightarrow{u^{\vee}}\tau_{\leq d-r+1}\alpha'_*i'^{*}Rj'_*\mu_p^{\otimes d-r+1}\longrightarrow \alpha'_*(M^{d-r+1}_V/U^1M^{d-r+1}_V)[-d+r-1].\end{multline*}
Here the last morphism is defined as the composite
\begin{multline}\tau_{\leq d-r+1}\alpha'_*i'^{*}Rj'_*\mu_p^{\otimes d-r+1} \longrightarrow \alpha'_*M^{d-r+1}_V[-d+r-1]\\ \longrightarrow \alpha'_*(M^{d-r+1}_V/U^1M^{d-r+1}_V)[-d+r-1].\end{multline}
It is easy to see that $\mathbb{B}^{\vee}(d-r+1)$ is concentrated in $[0,d-r+1]$, and the triple $(\mathbb{B}_m^{\vee}(d-r+1),h^{\vee},u^{\vee})$ is unique up to a unique isomorphism by Lemma (3). 

We construct $\Theta^r_{D}$ by decomposing the morphism 
\begin{align*}{\{\mathbb{B}_m(r)\}}_m \otimes^{\mathbb{L}} \mathbb{B}^{\vee}(d-r+1)&\longrightarrow  \{\tau_{\leq r} s_1(r)_{X|mD}\}_m \otimes^{\mathbb{L}} {\tau_{\leq d-r+1}\alpha'_*i'^{*}Rj'_*\mu_p^{\otimes d-r+1}}\\
&\xrightarrow{\cong} \{\tau_{\leq r} s_1(r)_{X|mD}\}_m \otimes^{\mathbb{L}} \{\tau_{\leq d-r+1}\mathfrak{s^{\vee,\bullet}_{{\rm {m'}}D}}\}_{m'} \\
&\longrightarrow \tau_{\leq d+1}\mathcal{S}(d+1)_{(X,M_X)}\xrightarrow{\cong} i^*R\psi_*\mu_p^{\otimes d+1}.\end{align*}

Here the second morphism is a quasi-isomorphism and the complex $\mathfrak{s^{\vee, \bullet}_{{\rm {m'}}D}}$ is defined as follows:
\[ \left[  {s_n(d-r)}^{\vee\;0}_{X|m'D}\longrightarrow \cdots \longrightarrow  {s_n(d-r)}^{\vee\;d-r+1}_{X|m'D} \longrightarrow \cdots \right], \]
where $s_n(q)^{\vee}_{X|m'D}$ is a complex which degree $i$-part defined by {\footnotesize\[{s_n(q)^{\vee}_{X|m'D}}^i:=\left(\omega^i_{Z_n}\otimes \mathscr{O}_{Z_n}((m'+i)\mathscr{D}_n)\otimes J^{[q-i]}_{\mathscr{E}_n}\right)\oplus \left(\omega^{i-1}_{Z_n}\otimes \mathscr{O}_{Z_n}((m'+i)\mathscr{D}_n)\otimes \mathscr{O}_{\mathscr{E}_n} \right).\]}
The third morphism are induced by the product morphism $s_1(r)_{X|mD} \otimes \mathfrak{s^{\vee,\bullet}_{{\rm {m'}}D}} \rightarrow \mathcal{S}(d+1)_{(X,M_X)}$, which is defined as for product structure\;(see \cite{Tsu1}).

By Lemma 7.3.2 \cite{Sa2} and the assumption that $\zeta_p \in K$, there is a morphism 
\[(\mathfrak{o})\quad i^*R\psi_*\mu_p^{\otimes d+1} \cong \mu' \otimes \big(\tau_{\leq d}i^*R\psi_*\mu_p^{\otimes d}\big)\xrightarrow{id\otimes \sigma^d_n[-d]} \mu' \otimes \nu^{d-1}_Y[-d],\]
where the morphism $\sigma^d_n[-d]$ is defined in Proposition 3.6 of \cite{Sa3}. 
Thus we have a morphism 
\[(*1)\quad{\{\mathbb{B}_m(r)\}}_m \otimes^{\mathbb{L}} \mathbb{B}^{\vee}(d-r+1) \longrightarrow \mu' \otimes \nu^{d-1}_Y[-d].\]
Noting that ${\{\mathbb{B}_m(r)\}}_m$ is concentrated in $[0,r]$ with $\mathcal{H}^r\big(\mathbb{B}_m(r)\big) \cong U^1\mathscr{M}^r_m$ and $ \mathbb{B}^{\vee}(d-r+1)$ is concentrated in $[0,d-r+1]$ with $\mathcal{H}^{d-r+1}\big(\mathbb{B}^{\vee}(d-r+1)\big) \cong \alpha'_*U^1M^{d-r+1}_V$, we show the following:
\begin{Lem} There is a unique morphism 
\[(*2)\quad{\{\mathbb{B}_m(r)\}}_m \otimes^{\mathbb{L}}\big(\alpha'_*U^1M^{d-r+1}_V[-d+r-1]\big)\longrightarrow \mu' \otimes \nu^{d-1}_Y[-d]\quad in\;D^-\big(Y_{\acute{e}t}, \mathbb{Z}/p\mathbb{Z}\big)\]
 for each $m$ that the morphism $(*1)$ factors through.
\end{Lem}
\begin{proof}
One can check this by a similar argument as for Lemma 8.2.4 in \cite{Sa2}.
\end{proof}
Applying a similar argument as for this lemma to the morphism $(*2)$, we obtain a morphism 
\[(*3)\quad{\{U^1\mathscr{M}^r_m\}}_m[-r]\otimes^{\mathbb{L}} \alpha'_*U^1M^{d-r+1}_V[-d+r-1]\longrightarrow \mu' \otimes \nu^{d-1}_Y[-d].\]
Because $\mathbb{Z}/p\mathbb{Z}$-sheaves are flat over $\mathbb{Z}/p\mathbb{Z}$, there is a natural isomorphism in $D^-\big(Y_{\acute{e}t}, \mathbb{Z}/p\mathbb{Z}\big)$\;: 
\[{\{U^1\mathscr{M}^r_m\}}_m[-r]\otimes^{\mathbb{L}} \alpha'_*U^1M^{d-r+1}_V[-d+r-1] \cong {\{U^1\mathscr{M}^r_m\}}_m\otimes\alpha'_*U^1M^{d-r+1}_V[-d-1]\]
induced by the identity map on the $(r+(d-r+1) )$-th cohomology sheaves. We thus define the morphism $\Theta^r_{D}$ by composing the inverse of this isomorphism and the morphism 
$(*3)$ with shifting $[-1]$.
\subsection{Explicit formula for $\Theta^r_D$}
We calculate the morphism $\Theta^r_D$ explicitly below. We use the idea of the construction of maps of \cite{Sa2}. 
Let $n$ and $q$ satisfy $1 \leq n \leq d$ and $1 \leq q \leq e'-1$. We put $n':=d+1-n$.  We define 
{\footnotesize\[\Symb^{q,r}_{m,D}:=\{i^*U^q{(1+I_{mD})}^{\times}\otimes(i^*\psi_*\mathscr{O}^{\times}_{U_K})^{\otimes r-1}\}\otimes \{U^{e'-q}_{X_K}\otimes(i^*\psi_*\mathscr{O}^{\times}_{U_K})^{\otimes d-r}\otimes  \mathscr{O}_Y(m'D_s)\}.\]}
Here we remark that there is an isomorphism \[\displaystyle{\alpha'_*U^1M^{d-r+1}_V \cong \lim_{\stackrel{\longrightarrow}{m'}}U^1M^{d-r+1}_Y \otimes \mathscr{O}_Y(m'D)},\] because  $\gr_U^1M^{d-r+1}_Y$ is a quasi-coherent sheaf. The sheaf $U^q\mathscr{M}_m^r \otimes \big(U^{e'-q}M^{d-r+1}_Y\otimes \mathscr{O}_Y(m'D_s) \big)$ is a quotient of $\Symb^{q,r}_{m,D}$:
{\footnotesize \[U^q\mathscr{M}_m^r \otimes \big(U^{e'-q}M^{d-r+1}_Y\otimes \mathscr{O}_Y(m'D_s)\big)=\Imm\big(\Symb^{q,r}_{m,D}\longrightarrow U^1\mathscr{M}_m^r \otimes \big(U^1M^{d-r+1}_Y\otimes \mathscr{O}_Y(m'D_s))\big).\]}
We define the homomorphism of \'etale sheaves 
\[F^{q,r}_{m,D}:\Symb^{q,r}_{m,D}\longrightarrow\omega^{d-1}_Y/\mathcal{B}^{d-1}_Y
\]
by sending a local section $(1+\pi^q\alpha_{1,m})\otimes \Big(\bigotimes_{i=1}^{r-1} \beta_{i,m}\Big) \otimes (1+\pi^{e'-q}\alpha_{2,m'})\otimes \Big(\bigotimes_{i=r}^{d-r} \beta'_{i,m'}\Big)\otimes \gamma_{m'}$ with $\alpha_{1,m} \in i^*\mathscr{O}_X(-mD)$, $\beta_{i,m} \in i^*\psi_*\mathscr{O}^{\times}_{U_K}$, $\alpha_{2,m'} \in i^*\mathscr{O}_X$, $\beta_{i,m} \in i^*j_*\mathscr{O}^{\times}_{X_K}$, $\gamma_{m'} \in  \mathscr{O}_Y(m'D_s)$, to the following:
\begin{multline} q\cdot \overline{\alpha_{1,m}\cdot \alpha_{2,m}} \cdot \gamma_{m'}\cdot \Big(\bigwedge^{r-1}_{i=1}d\log \overline{\beta_{i,m}} \land \bigwedge^{d-1}_{i=r}d\log\overline{\beta'_{i,m'}}\Big)\\ +g^{-1}\Big(\overline{\alpha_{1,m'}}\cdot d(\overline{\alpha_{2,m}}\cdot\gamma_{m'})\land \bigwedge^{r-1}_{i=1}d\log\overline{\beta_{i,m}}\land \bigwedge^{d-1}_{i=r}d\log\overline{\beta'_{i,m'}}\Big)\; \mod \mathcal{B}^{d-1}_{Y}.\end{multline}
Here we put $\overline{x}$ its residue class in $\mathscr{O}_Y$ for $x \in i^*\mathscr{O}_X$ and put $g$ the following $\mathscr{O}_Y$-linear isomorphism which is defined in \cite{Sa2}, :
\[g\;:\;\omega^N_Y\xrightarrow{\cong} \omega^{N+1}_{(Y,M_Y)/(s,N_s)},\quad \omega\mapsto d\log(\overline{\pi})\land \omega. \] 
In what follows, we put $U^{e'-q}M^{d-r+1}_{n,Y, m'D}:=U^{e'-q}M^{d-r+1}_Y\otimes \mathscr{O}_Y(m'D_s)$ for simplicity.
\begin{Lem}(cf.\;\cite{Sa2}, Lemma 8.3.4)
Let $r$ and $q$ be as above. Then $F^{q,r}_{m,D}$ factors through $U^q\mathscr{M}_m^r \otimes U^{e'-q}M^{d-r+1}_{n,Y, m'D}.$
\end{Lem}
\begin{proof}We prove this Lemma by the method of Lemma 8.3.4, \;\cite{Sa2}. We denote $Y^{\sing}$ the singular locus of $Y$, and let $j_Y$ be the open immersion $Y-Y^{\sing} \hookrightarrow Y$. 
Since $\omega^N_Y/ \mathcal{B}^N_Y$ is a locally free $\left(\mathscr{O}_Y\right)^p$-module, we may assume that $Y$ is smooth over $s(=\Spec(k))$. 
We prove that $F^{q,r}_{m,D}$ factors $\gr^q_{U}\mathscr{M}_m^r \otimes \gr^{e'-q}_UM^{d-r+1}_{n,Y, m'D}$ under the assumption that $Y$ is smooth.
We put $\Omega^u_{Y|-m'D}:=\Omega^u_Y \otimes \mathscr{O}_Y(m'D)$ below. Let 
\[(1)\quad\rho^{l,u}_{mD}:\Omega^{u-2}_{Y|mD} \oplus \Omega^{u-1}_{Y|mD} \longrightarrow {\{\gr^l_{U}\mathscr{M}_m^u\}}_m, \]
\[(2)\quad\rho^{l,u}_{m'}:\Omega^{u-2}_{Y|-m'D} \oplus \Omega^{u-1}_{Y|-m'D} \longrightarrow \gr^l_{U}M^{u}_{n,Y, m'D},\]
be the map\;(cf. \cite{BK}, (4.3)) defined as follows:
{\tiny\[(1')\;\begin{cases}
({\alpha}_m\cdot d \log \beta_1\land \cdots \land d\log \beta_{u-2},0)\mapsto& \{1+\pi^l\tilde{{\alpha}}_m, \tilde{\beta}_1, \cdots, \tilde{\beta}_{u-2},\pi\}\mod U^{l+1}\mathscr{M}_m^u,\\
(0, {\alpha}_m\cdot d \log \beta_1\land \cdots \land d\log \beta_{u-1})\mapsto& \{1+\pi^l\tilde{{\alpha}}_m, \tilde{\beta}_1, \cdots, \tilde{\beta}_{u1},\pi\}\mod U^{l+1}\mathscr{M}_m^u,
\end{cases}
\]

\[(2')\;\begin{cases}
({\eta}_m\cdot d \log \beta_1\land \cdots \land d\log \beta_{u-2},0)\mapsto& \{1+\pi^l\tilde{{\eta}}_m, \tilde{\beta}_1, \cdots, \tilde{\beta}_{u-2},\pi\}\mod U^{l+1}M^{u}_{n,Y, m'D},\\
(0, {\eta}_m\cdot d \log \beta_1\land \cdots \land d\log \beta_{u-1})\mapsto& \{1+\pi^l\tilde{{\eta}}_m, \tilde{\beta}_1, \cdots, \tilde{\beta}_{u1},\pi\}\mod U^{l+1}M^{u}_{n,Y, m'D},
\end{cases}
\] }for ${\alpha}_m \in \mathscr{O}_Y(-mD_s)$, ${\eta}_m \in \mathscr{O}_Y(m'D)$ and each $\beta_i \in \mathscr{O}^{\times}_Y$, where $\tilde{\alpha}_m \in \mathscr{O}_X(-mD)\;(\resp. \tilde{\beta}_i \in \mathscr{O}^{\times}_X)$ denotes a lift of $\alpha_m\;(\resp. \beta_i)$. 
We have the following short exact sequences which is considered in \cite{BK} Lemma 4.5: 
\[(a)\quad\quad0 \longrightarrow\Omega^{u-2}_{Y|mD} \xrightarrow{\theta^{l,u}_{mD}} \Omega^{u-2}_{Y|mD}\oplus\Omega^{u-1}_{Y|mD} \xrightarrow{\rho^{l,u}_{mD}}\gr^l_{U}\mathscr{M}_m^u\longrightarrow 0\quad(p \not | l), \]
\[(b)\quad\quad0 \longrightarrow\mathcal{Z}^{u-2}_{Y|mD}\oplus\mathcal{Z}^{u-1}_{Y|mD} \longrightarrow \Omega^{u-2}_{Y|mD}\oplus\Omega^{u-1}_{Y|mD} \xrightarrow{\rho^{l,u}_{mD}}\gr^l_{U}\mathscr{M}_m^u\longrightarrow 0\quad(p | l), \]
\[(c)\quad\quad0 \longrightarrow\Omega^{u-2}_{Y|-m'D} \xrightarrow{\theta^{l,u}_{m'}} \Omega^{u-2}_{Y|-m'D}\oplus\Omega^{u-1}_{Y|-m'D} \xrightarrow{\rho^{l,u}_{m'}}\gr^l_{U}M^{u}_{n,Y, m'D}\longrightarrow 0\quad(p \not | l), \]
\[(d)\quad\quad0 \longrightarrow\mathcal{Z}^{u-2}_{Y|-m'D}\oplus\mathcal{Z}^{u-1}_{Y|-m'D} \longrightarrow \Omega^{u-2}_{Y|-m'D}\oplus\Omega^{u-1}_{Y|-m'D} \xrightarrow{\rho^{l,u}_{m'}}\gr^l_{U}M^{u}_{n,Y, m'D}\longrightarrow 0\quad(p | l).\]
Here $\theta^{l,u}_{mD}$ and $\theta^{l,u}_{m'}$ are given by $\omega \mapsto \left((-1)^u\cdot l \cdot \omega,\;d\omega \right)$.
We define the following maps: 
\[h^{l,u}_{mD}:i^*U^0\left(1+I_{mD}\right)^{\times} \otimes \left(i^*\psi_*\mathscr{O}_{U_K}^{\times}\right)^{\otimes u-1}\longrightarrow \Omega^{u-2}_{Y|mD}\oplus\Omega^{u-1}_{Y|mD};\]
{\footnotesize\[(1+\pi^l\alpha_m)\otimes\left(\bigotimes^{u-1}_{i=1}\beta_i\right) \mapsto \begin{cases} \left(0, \overline{\alpha}_m\cdot \bigwedge_{1 \leq i \leq u-1}d\log\overline{\beta}_i\right)\quad({\rm if}\;\beta_i \in i^*\mathscr{O}_X^{\times}\;{\rm for\;all}\;i ),\\
\left((-1)^{u-1-i'}\cdot \overline{\alpha}_m\cdot \bigwedge_{1 \leq i \leq u-1, i \neq i'} d\log \overline{\beta}_i, 0 \right)\quad({\rm if}\;\beta_i=\pi\;{\rm for \;exactly\; one}\; i \neq i'),\\
(0,0)\quad({\rm otherwise}),
\end{cases}\]}

\[h^{l,u}_{m'}:U_{X_K}^l \otimes \left(i^*\psi_*\mathscr{O}_{U_K}^{\times}\right)^{\otimes u-1}\otimes \mathscr{O}_Y(m'D_s)\longrightarrow \Omega^{u-2}_{Y|-m'D}\oplus\Omega^{u-1}_{Y|-m'D};\]
{\footnotesize\[(1+\pi^l\alpha)\otimes\left(\bigotimes^{u-1}_{i=1}\beta_i\right)\otimes \gamma_{m'} \mapsto \begin{cases} \left(0, \gamma_{m'}\cdot\overline{\alpha}\cdot \bigwedge_{1 \leq i \leq u-1}d\log\overline{\beta}_i\right)\quad({\rm if}\;\beta_i \in i^*\mathscr{O}_X^{\times}\;{\rm for\;all}\;i ),\\
\left((-1)^{u-1-i'}\cdot  \gamma_{m'}\cdot\overline{\alpha}\cdot \bigwedge_{1 \leq i \leq u-1, i \neq i'} d\log \overline{\beta}_i, 0 \right)\quad({\rm if}\;\beta_i=\pi\;{\rm for \;exactly\; one}\; i \neq i'),\\
(0,0)\quad({\rm otherwise}),
\end{cases}\]}
with $\alpha \in i^*\mathscr{O}_X$, $\alpha_m \in  i^*\mathscr{O}_X(-mD)$, $\gamma_m\in \mathscr{O}_Y(m'D_s)$ and $\beta_i \in  i^*\mathscr{O}_X^{\times}\cup\{\pi\}$.
Here for $x \in i^*\mathscr{O}_X\;(resp. x \in i^*\mathscr{O}_X^{\times})$, $\overline{x}$ denotes its residue class in $\mathscr{O}_Y\;(resp. i^*\mathscr{O}_Y^{\times})$.
We consider a commutative diagram\;(cf.\;\S8.3,\;\cite{Sa2})\\

{\tiny \[
\xymatrix@C=66pt@R=30pt{   
\Symb^{q,r}_{mD}  \ar@/^46pt/@{>}[rr]_{{\rm symbol\;map}}\ar[rd]_-{F^{q,r}_{m,D}} \ar[r]^-{h^{q,r} \otimes h^{e'-q, d-r+1}} & \left(\Omega^{u-2}_{Y|mD}\oplus\Omega^{u-1}_{Y|mD}\right) \oplus \left( \Omega^{u-2}_{Y|-m'D}\oplus\Omega^{u-1}_{Y|-m'D}\right)\ar[d]^-{\varphi^{q,r}_{D}} \ar[r]^-{\rho^{q,r}_{mD}\otimes\rho^{e'-q,d-r+1}_{m'}} &\gr^q_{U}\mathscr{M}_m^r \otimes\gr^{e'-q}_{U}M^{d-r+1}_{n,Y, m'D}&\\
 &\Omega^N_Y/\mathcal{B}^N_Y,&&
 } \]}

where $\varphi^{q,r}_D$ is defined as for \S8.3, \cite{Sa2}\;:
\begin{align*}&(\omega_1\otimes \gamma_{1,m},\;\omega_2\otimes \gamma_{2,m})\oplus
(\omega_3\otimes \delta_{1,m},\; \omega_4\otimes \delta_{2,m}) \\
&\mapsto \left(q\cdot \gamma_{2,m}\cdot \delta_{2,m}(\omega_2\land \omega_4)+(-1)^{r-1}\gamma_{1,m}\cdot\delta_{2,m}d(\omega_1)\land \omega_4+(-1)^{d-r} \gamma_{2,m}\cdot \delta_{1,m}(\omega_2\land d(\omega_3)       \right)\\
&\mod \mathcal{B}^N_Y.\end{align*}
We check the composition map $(G:=)\varphi^{q,r}_{D}\circ (h^{q,r} \otimes h^{e'-q, d-r+1})=F^{q,r}_{mD}$.
We take any local section $\mathfrak{x}:=(1+\pi^q\alpha_{1,m})\otimes \Big(\bigotimes_{i=1}^{r-1} \beta_{i,m}\Big) \otimes (1+\pi^{e'-q}\alpha_{2,m'})\otimes \Big(\bigotimes_{i=r}^{d-r} \beta'_{i,m'}\Big)\otimes \gamma_{m'}$ in $\Symb^{q,r}_{mD} $.  We can consider $(3\times3)$-cases, but we calculate the following cases(other cases left to the reader):\\
$\bullet$\;If $\beta_{i,m} \in i^*\mathscr{O}_X^{\times}$ for all $i$ and $\beta_{i,m'}=\pi$ for exactly one $i'$, we denoted by $k$.  
{\tiny\begin{align*}
&G(\mathfrak{x})=\varphi^{q,r}_{D}\left(\left(0,\; \overline{\alpha}_{1,m}\cdot \bigwedge_{1 \leq i \leq r-1}d\log\overline{\beta}_{i,m}\right),\;\left((-1)^{d-r-i'}\cdot  \gamma_{m'}\cdot\overline{\alpha}_{2,m'}\cdot \bigwedge_{r \leq i \leq d-r, i \neq i'} d\log \overline{\beta}_{i,m'},\;0 \right)                                            \right)\\
&=(-1)^{r'}\left(\overline{\alpha}_{1,m}\cdot \bigwedge_{1 \leq i \leq r-1}d\log\overline{\beta}_{i,m} \right)\land d\left(  (-1)^{d-r-i'}\cdot  \gamma_{m'}\cdot\overline{\alpha}_{2,m'}\cdot \bigwedge_{r \leq i \leq d-r, i \neq i'} d\log \overline{\beta}_{i,m'}\right)\\
&=\overline{\alpha}_{1,m}d( \gamma_{m'}\cdot\overline{\alpha}_{2,m'})\land  \bigwedge_{1 \leq i \leq r-1}d\log\overline{\beta}_{i,m} \land \bigwedge_{r \leq i \leq d-r, i \neq i'} d\log \overline{\beta}_{i,m'}\\
&=F^{q,r}_{mD}(\mathfrak{x})\quad\quad \mod \mathcal{B}^{d-1}_Y.
\end{align*}}
Hence it suffices to prove that the subsheaf  $\Ker(\rho^{q,r}_{mD}\otimes\rho^{e'-q,d-r+1}_{m'})$ of {\tiny$\left(\Omega^{u-2}_{Y|mD}\oplus\Omega^{u-1}_{Y|mD}\right) \oplus \left( \Omega^{u-2}_{Y|-m'D}\oplus\Omega^{u-1}_{Y|-m'D}\right)$} has trivial image under $\varphi_D^{q,r}$.
We only consider the case (i)\;$p \nmid l$ for simplicity. For any element $x=(x_1,x_2) \in \Ker(\rho^{q,r}_{mD}\otimes\rho^{e'-q,d-r+1}_{m'})$, there exists an element $\eta_1\otimes \gamma_{1,m'}  \in \Omega^{u-2}_{Y|mD}$ such that $\theta^{l.u}_{mD}(\eta_1\otimes \gamma_{1,m'} )=x_1$ and there exists an element $\eta_2\otimes \gamma_{2,m'} \in \Omega^{u-2}_{Y|-m'D}$ such that $\theta^{l.u}_{m'}(\eta_2\otimes \gamma_{2,m'})=x_2$ by the exact sequences $(a)$ and $(c)$. Then we can compute $\varphi^{q,r}_D(x)$ as follows:
{\footnotesize\begin{align*}
\varphi^{q,r}_D(x)&=\varphi^{q,r}_D\left(\big((-1)^rq\eta_1\otimes \gamma_{1,m'},\;d\eta_1\otimes \gamma_{1,m'}\big),\big((-1)^{d-r+1}(e'-q)\cdot\eta_2\otimes \gamma_{2,m'},\;d\eta_2\otimes \gamma_{2,m'}\big)\right)\\
&=((-1)^rq \gamma_{1,m'}\gamma_{2,m'} d\eta_1 \land d\eta_2+(-1)^{r-1} \gamma_{1,m'} \gamma_{2,m'}d((-1)^r q\eta_1)\land d\eta_2\\
&(-1)^{d-r} \gamma_{1,m'} \gamma_{2,m'}(d\eta_1\land d((-1)^{d-r+1}(e'-q)\cdot\eta_2)) )\quad \mod \mathcal{B}^N_{Y}\\
&=((-1)^rq \gamma_{1,m'}\gamma_{2,m'} d\eta_1 \land d\eta_2+(-1)q \gamma_{1,m'} \gamma_{2,m'}d(\eta_1)\land d\eta_2\\
&(-1)(e'-q)\gamma_{1,m'} \gamma_{2,m'}(d\eta_1\land d\eta_2) )\quad \mod \mathcal{B}^N_{Y}
\end{align*}}
Since $d(\eta_1\land d(\eta_2))=d\eta_1 \land d\eta_2$,  we obtain the assertion in the case $(i)$. This completes the proof.
\end{proof}

\begin{Def}(cf.\;\cite{Sa2}, Definition 8.3.6)\;For $\zeta(\neq 1) \in \mu_p(K)$, let $v(\zeta) \in k^{\times}$ be the residue class of $(1-\zeta)/\pi^{e(p-1)} \in \mathscr{O}_K^{\times}$. We put $w:=\zeta \otimes v(\zeta)^{-p} \in \mu \otimes k$(see Definition 8.3.6 in \cite{Sa2}).
We define the homomorphism as for $f^{q,n}$ in Definition 8.3.6 in \cite{Sa2}
 \begin{align*}\xymatrix@M=10pt{f^{q,r}_{m,D}:U^q\mathscr{M}^r_m \otimes U^{e'-q}M^{d-r+1}_{n,Y, m'D} \ar[r]& (\mu' \otimes \underbar{k}) \otimes_{\underbar{k}} \omega^N_Y/\mathcal{B}^N_Y\\ \ar[r]^-{\cong}& \mu' \otimes \omega^N_Y/\mathcal{B}^N_Y  }\end{align*} as $w \otimes (-1)^{N+n}\Tilde{F^{q,r}_{m,D}}$. 
 Here $\underbar{k}$ is the constant sheaf on $Y_{\acute{e}t}$ associated with $k$.  We regard $w \in \mu \otimes k$ as a global section of $\mu' \otimes \underbar{k}$, and $\Tilde{F^{q,r}_{m,D}}$ denotes the map induced by $F^{q,r}_{m,D}$. The map $f^{q,r}_{m,D}$ is independent of the choice of $\pi$ by the definition of $w$ and $\Tilde{F^{q,r}_{m,D}}$.
\end{Def}

%%%%%%%%%%%%%%%%%%%%%%%%%%%%%%%%%%%%%%%%%%%%%%%%%%%%%%%%%%%%%%%%%%%%%%%
We state the main result of this section. 
\begin{Thm}\label{ERL}(Explicit reciprocity law, cf.Theorem 8.3.8, \cite{Sa2})
We have the following commutative square of pro-systems in $D^b\big(Y_{\acute{e}t}, \mathbb{Z}/p^n\mathbb{Z}\big)$
\[
\xymatrix@M=10pt{   
{\{U^q\mathscr{M}_m^r\}}_m\otimes {\{U^{e'-q}M^{d-r+1}_{n,Y, m'D}\}}_{m'}\ar[d]^{\rm canonical}\ar[r]^-{f^{q,r}_{m,D}}& \mu'\otimes \omega^N_Y/\mathcal{B}^N_Y\ar[d]^{\chi} &\\
 {\{U^1\mathscr{M}_m^r\}}_m\otimes {\{U^{1}M^{d-r+1}_{n,Y, m'D}\}}_{m'} \ar[r]^-{\Theta^{r}_{m,D}[d+2]} & \mu'\otimes \nu^N_{Y,n}[1],&
 } 
\]for $(q, r)$ with $1\leq q \leq e'-1$ and $1 \leq r \leq d$. Here $\chi$ is defined in \S8.3, \cite{Sa2}.
\end{Thm}

We will prove Theorem \ref{ERL} in two steps  below by the same arguments as for the proof of Theorem 8.3.8, \cite{Sa2}.\\
{\bf I.}\;\;Reduction to cohomology groups(cf. \S8.4, \cite{Sa2})\;:\mbox{}

Without loss of generality, we may assume that $X$ is connected. Thus by Proposition 8.4.1, \cite{Sa2} for $i=N$, we have 
\begin{align*}
\Hom&_{D^b(Y_{\acute{e}t}, \mathbb{Z}/p\mathbb{Z})}\left({\{U^q\mathscr{M}_m^r\}}_m\otimes {\{U^{e'-q}M^{d-r+1}_{n,Y, m'D}\}}_{m'},  \mu'\otimes \nu^N_{Y,n}[1] \right)\\
\cong&\Hom\left(H^N(Y, {\{U^q\mathscr{M}_m^r\}}_m\otimes {\{U^{e'-q}M^{d-r+1}_{n,Y, m'D}\}}_{m'}), \mu \otimes H^{N+1}(Y, \nu^n_Y) \right).
\end{align*}
Therefore we are reduced to the equality of maps on cohomology groups
\[(\circledast 1)\quad\quad H^N(Y, \Theta^{q,r}_{m,D})=H^N(Y, \chi \circ f^{q,r}_{m,D}),\]
where we denote $\Theta^{q,r}_{m,D}$ the composite morphism 
{\footnotesize\begin{align*}
\xymatrix@M=10pt{\Theta^{q,r}_{m,D}:{\{U^q\mathscr{M}_m^r\}}_m\otimes {\{U^{e'-q}M^{d-r+1}_{n,Y, m'D}\}}_{m'}\ar[r]^-{{\rm canonical}}&{\{U^1\mathscr{M}_m^r\}}_m\otimes {\{U^{1}M^{d-r+1}_{n,Y, m'D}\}}_{m'}\\\ar[r]^-{\Theta^{r}_{m,D}[d+2]}&  \mu'\otimes \nu^N_{Y,n}[1]. }\end{align*}}
{\bf IId.}\;\; Reduction to higher local fileds(cf. \S8.5, \cite{Sa2})\;:\mbox{}

In this subsection, we reduce $(\circledast 1)$ to $(\circledast 2)$ below by Sato's trick used in \S8.5, \cite{Sa2}. 
For $\mathcal{K}\in D^b(Y_{{\acute e}t},\mathbb{Z}/p)$, we can consider the map 
\[
\xymatrix@=10pt{\delta_{Y}(\mathcal{K}):\bigoplus_{(y_0,y_1,\cdots,y_N) \in \Ch(Y)} H^0(y_N, \mathcal{K}) \ar[r]&H^N(Y, \mathcal{K}).}\]
Here $(y_0,y_1,\cdots,y_N)\in\Ch(Y)$ is a chain on $Y$(see \S 8.5, \cite{Sa2} for details).
We obtain the following Lemma by Sublemma 8.5.2, \cite{Sa2}:
\begin{Lem}
The map $\delta_Y\left({\{U^q\mathscr{M}_m^r\}}_m\otimes {\{U^{e'-q}M^{d-r+1}_{n,Y, m'D}\}}_{m'} \right)$ is surjective.
\end{Lem}
 By the above lemma, $(\circledast 1)$ is reduced to the formula 
 \[(\circledast 2)\quad\quad H^0(y_N, \Theta^{q,r}_{m,D})=H^0(y_N, \chi \circ f^{q,r}_{m,D}),\]
 for all chains $(y_0,y_1,\cdots,y_N) \in \Ch(Y)$, which will be proved in next subsection below.

%%%%%%%%%%%%%%%%%%%%%%%%%%%%%%%%%%%%%%%%%%%%%%%%%%%%%%%%%%%%%%%%%%%%%%%%%% 
\subsection{Proof of $(\circledast 2)$}\mbox{}
We begin the proof of  the $(\circledast 2)$ by the same manner of the proof of $(8.5.3)$ in \S8.7, \cite{Sa2}.
We fix an arbitrary chain $(y_0,y_1,\cdots,y_N) \in \Ch(Y)$. Put $F_N:=\kappa(y_N)$ and $L_{N+1}:=\Frac\left[(\cdots ((\mathscr{O}^h_{X,y_0})^h_{y_1})\cdots)^h_{y_N} \right]$
which is a henselian discrete valuation field of characteristic $0$ with residue field $F_N$.\;Let $F/F_N$ be a finite separable field extension. We put $y:=\Spec(F)$ and 
{\footnotesize\[\xymatrix@M=10pt{\mathscr{A}^{q,r}_m(F):=H^0(y,U^q\mathscr{M}_m^r) \otimes H^0(y, U^{e'-q}M^{d-r+1}_{n,Y, m'D}) \subset H^0(y, U^q\mathscr{M}_m^r \otimes U^{e'-q}M^{d-r+1}_{n,Y, m'D}).  }\]}
We use Lemma 8.6.1, \cite{Sa2} for the subfield $(F_N)^p \subset F_N$, the formula $(\circledast 2)$ is reduced to the formula
 \[(\circledast 3)\quad\quad H^0(y, \Theta^{q,r}_{m,D})|_{\mathscr{A}^{q,r}_m(F)}=H^0(y, \chi \circ f^{q,r}_{m,D})|_{\mathscr{A}^{q,r}_m(F)}.\]
Here we have $H^0(y, U^q\mathscr{M}_m^r)\cong H^0(y, U^qM^r)$ and $H^0(y, U^{e'-q}M^{d-r+1}_{n,Y, m'D})\cong H^0(y, U^{e'-q}M^{d-r+1}_{n,Y})$. Thus we obtain the $(\circledast 3)$ by the same argument as for \S $8.5$ in \cite{Sa2}. This completes the proof of Theorem \ref{ERL}.

%%%%%%%%%%%%%%%%%%%%%%%%%%%%%%%%%%%%%%%%%%%%%%%%%%%%%%%%%%%%%%%%%%%%%%%%%%%%
\section{\textbf{Duality for $\mathscr{M}^r_m$}}
In this section, we prove the following Theorem \ref{D-U} which will be used in \S8. 
\begin{Thm}\label{D-U}(cf. Theorem 9.1.1, \cite{Sa2})Let $r$ be $1 \leq r \leq d$. Then for an integer $i$, the following pairing induced by  $\Theta^{r}_{m,D}$ and $\tr_Y$:
\[\lim_{\stackrel{\longleftarrow}{m}}H^i(Y, U^1\mathscr{M}^r_m)  \times H^{2d+1-i}(Y, \alpha'_*U^1M^{d-r+1}_{n,V})\xrightarrow{\Theta^r{m,D}} \mu \otimes H^{d}(Y, \nu^{d-1}_{Y,n})\xrightarrow{\id \otimes \tr} \mu\]
is a non-degenerate pairing. 
\end{Thm}
We will calculate the map $f^{q,r}_m$ constructed in \S5.  We recall and define some filtrations on $M^{d-r+1}_{n,Y}$ and $\mathscr{M}^r_{m}$.
\begin{Def}(\S9 in \cite{Sa2} and \S3 in \cite{Y})
\begin{enumerate}
\item\; We define the subesheaf $T^qM^r_Y \subset U^qM^r_Y\;(q \geq 1)$ as the part generated by $V^qM^r_Y$ and symbols of the form 
\[\{1+\pi^q\alpha^p,\beta_1,\cdots,\beta_{r-1}\}\]
with $\alpha \in i^*\mathscr{O}_X$ and each $\beta_i \in i^*j_*\mathscr{O}_{X_K}^{\times}$.
\item\; We similally define the subesheaf $T^q\mathscr{M}^r_{m}\subset U^q\mathscr{M}^r_{m}\;(q \geq 1)$ as the part generated by $V^q\mathscr{M}^r_{m}$ and symbols of the form 
\[\{1+\pi^q\alpha^p,\beta_1,\cdots,\beta_{r-1}\}\]
with $\alpha \in i^*\mathscr{O}_X(-mD)$ and each $\beta_i \in i^*\psi_*\mathscr{O}_{U_K}^{\times}$ for each $m$.
\end{enumerate}
\end{Def}
From the definition, we have 
$U^{q+1}\subset V^q\subset T^q\subset U^{q}$ for $M^r_{Y, m'D}$ and $\mathscr{M}^r_{m}$ respectively.
For $q \geq 1$, we put 
\[\gr^q_{U/T}:=U^q/T^q,\quad \gr^q_{T/V}:= T^q/V^q,\quad\gr^q_{V/U}:= V^q/U^{q+1}\]
for $M^r_{Y, m'D}$ and $\mathscr{M}^r_{m}$ respectively.

By Theorem 3.3.7 in \cite{Sa2} $(3)$ and $(4)$, the sheaf $U^{e'}M^r_{Y, m'D}=0$, and for $q$ with $1 \leq q \leq e'-1$
we have isomorphisms 
\begin{align*}
\rho^{q,l}_{1,m'}:\gr^q_{U/T}M^l_{Y, m'D} \xrightarrow{\cong}& \left(\omega^{l-1}_Y/\mathcal{Z}^{l-1}_Y\right)\otimes_{ \mathscr{O}_Y} \mathscr{O}_Y(m'D),\\
\{1+\pi^q\alpha, \beta_1, \cdots, \beta_{l-1}\}\otimes \gamma_{m'}\mod T^qM^r_{Y, m'D} \mapsto \;&\overline{\alpha}\cdot\left(\bigwedge^{l-1}_{i=1} d\log \overline{{\beta}}_i\right)\otimes \gamma_{m'} \mod \mathcal{Z}^{r-1}_Y \\
\rho^{q,l}_{2,m'}:\gr^q_{T/V}M^l_{Y, m'D}  \xrightarrow{\cong}& \begin{cases} \left(\mathcal{Z}^{l-1}_Y/\mathcal{B}^{l-1}_Y\right)\otimes_{ \mathscr{O}_Y} \mathscr{O}_Y(m'D)\quad(p\nmid q),\\
0\quad(p |q),
\end{cases}\\
\{1+\pi^q\alpha^p, \beta_1, \cdots, \beta_{l-1}\}\otimes \gamma_{m'}\mod V^qM^r_{Y, m'D} \mapsto \overline{\alpha}^p&\cdot\left(\bigwedge^{l-1}_{i=1} d\log \overline{{\beta}}_i\right)\otimes \gamma_{m'}\mod \mathcal{B}^{r-1}_Y\;(p\nmid q) \\
\rho^{q,l}_{3,m'}:\gr^q_{V/U}M^l_{Y, m'D} &\xrightarrow{\cong} \left(\omega^{l-2}_Y/\mathcal{Z}^{l-2}_Y\right)\otimes_{\mathscr{O}_Y} \mathscr{O}_Y(m'D),\\
\{1+\pi^q\alpha, \beta_1, \cdots, \beta_{l-2},\pi\}\otimes \gamma_{m'}\mod U^{q+1}M^r_{Y, m'D} \mapsto &\overline{\alpha}\cdot\left(\bigwedge^{l-2}_{i=1} d\log \overline{{\beta}}_i\right)\otimes \gamma_{m'} \mod \mathcal{Z}^{r-2}_Y.
\end{align*}

On the other hands, we get the following isomorphisms by Thorem 3.5 in \cite{Y}:
\begin{align*}
\rho^{q,j}_{1,mD}:{\{\gr^q_{U/T}\mathscr{M}^j_{m}\}}_m&\xrightarrow{\cong} {\left\{\left(\omega^{j-1}_{Y|mD}/\mathcal{Z}^{j-1}_{Y|mD}\right)\right\}}_m,\\
\{1+\pi^q\alpha, \beta_1, \cdots, \beta_{j-1}\}_m&\mod T^q\mathscr{M}^r_{m}  \mapsto \overline{\alpha}\cdot\left(\bigwedge^{j-1}_{i=1} d\log \overline{{\beta}}_i\right)\mod \mathcal{Z}^{j-1}_{Y|mD} \\
\rho^{q,j}_{2,mD}:{\{\gr^q_{T/V}\mathscr{M}^j_{m}\}}_m& \xrightarrow{\cong} \begin{cases} {\left\{\left(\mathcal{Z}^{j-1}_{Y|mD}/\mathcal{B}^{j-1}_{Y|mD}\right)\right\}}_m\quad(p\nmid q),\\
0\quad(p |q),
\end{cases}\\
\{1+\pi^q\alpha^p, \beta_1, \cdots, \beta_{j-1}\}_m&\mod V^q\mathscr{M}^r_{m} \mapsto \overline{\alpha}^p\cdot\left(\bigwedge^{j-1}_{i=1} d\log \overline{{\beta}}_i\right)\mod \mathcal{B}^{j-1}_{Y|mD}\quad(p \nmid q)\\
\rho^{q,j}_{2,mD}:{\{\gr^q_{V/U}\mathscr{M}^j_{m}\}}_m&\xrightarrow{\cong} {\left\{\left(\omega^{j-2}_{Y|mD}/\mathcal{Z}^{j-2}_{Y|mD}\right)\right\}}_m,\\
\{1+\pi^q\alpha, \beta_1, \cdots, \beta_{j-2},\pi\}_m&\mod U^{q+1}\mathscr{M}^r_{m} \mapsto \overline{\alpha}\cdot\left(\bigwedge^{j-2}_{i=1} d\log \overline{{\beta}}_i\right)\mod \mathcal{Z}^{j-2}_{Y|mD}.
\end{align*}

We put $e'=pe/(p-1)$ an integer and $p|e'$. 
\begin{Lem}\label{Lem maps}(cf. Lemma 9.1.4, \cite{Sa2})Assume that $1 \leq q \leq e'-1$. Then:
\begin{enumerate}
 \item\; $f^{q,r}_{m,D}$ annihilates the subsheaf of ${\{U^{q}\mathscr{M}^r_{m}\}}_m \otimes {\{U^{q}M^{d-r+1}_{Y, m'D}\}}_{m'}$
 generated by {\footnotesize${\{U^{q+1}\mathscr{M}^r_{m}\}}_m \otimes {\{U^{e'-q}M^{d-r+1}_{Y, m'D}\}}_{m'}$, ${\{U^{q}\mathscr{M}^r_{m}\}}_m \otimes {\{U^{e'-q+1}M^{d-r+1}_{Y, m'D}\}}_{m'}$, ${\{V^{q}\mathscr{M}^r_{m}\}}_m \otimes {\{T^{e'-q}M^{d-r+1}_{Y, m'D}\}}_{m'}$ and ${\{T^{q}\mathscr{M}^r_{m}\}}_m \otimes {\{V^{e'-q}M^{d-r+1}_{Y, m'D}\}}_{m'}$.}
 \item\; The composite map 
{\footnotesize \begin{align*}&{\{\omega^{r-1}_{Y|mD}/\mathcal{Z}^{r-1}_{Y|mD}\}}_m \otimes {\{\left(\omega^{d-r}_Y/\mathcal{Z}^{d-r}_Y\right)\otimes_{ \mathscr{O}_Y} \mathscr{O}_Y(m'D)\}}_{m'}\\
   &\xrightarrow{ (\rho^{q,r}_{1,mD}\otimes \rho^{e'-q, d-r+1}_{3,m'})^{-1} } {\{\gr^q_{U/T}\mathscr{M}^r_{m}\}}_m \otimes {\{\gr^{e'-q}_{V/U}M^l_{Y, m'D}\}}_{m'}\\
   &\xrightarrow{f^{q,r}_{m, D}} \mu' \otimes \left(\omega^{N}_Y/\mathcal{B}^N_Y\right)
\end{align*}}
sends a local section ${\{x_m \otimes \gamma_{1, m}\}}_m \otimes {\{y_{m} \otimes  \gamma_{2, m}\}}_m$ to $w\otimes_k (-1)^r $.
Similarly, the composite map 
{\footnotesize\begin{align*}&{\{\omega^{r-2}_{Y|mD}/\mathcal{Z}^{r-2}_{Y|mD}\}}_m \otimes {\{\left(\omega^{d-r+1}_Y/\mathcal{Z}^{d-r+1}_Y\right)\otimes_{ \mathscr{O}_Y} \mathscr{O}_Y(m'D)\}}_{m'}\\
   &\xrightarrow{ (\rho^{q,r}_{3,mD}\otimes \rho^{e'-q, d-r+1}_{1,m'})^{-1} } {\{\gr^q_{V/U}\mathscr{M}^r_{m}\}}_m \otimes {\{\gr^{e'-q}_{U/T}M^l_{Y, m'D}\}}_{m'}\\
   &\xrightarrow{f^{q,r}_{m, D}} \mu' \otimes \left(\omega^{N}_Y/\mathcal{B}^N_Y\right)
\end{align*}}
sends a local section ${\{x_m \otimes  \gamma_{1, m}\}}_m \otimes {\{y_{m} \otimes  \gamma_{2, m}\}}_m$ to $w\otimes_k (-1)^N$.
\item\;If $q$ is prime to $p$, then the composite map 
{\footnotesize\begin{align*}&{\{\omega^{r-1}_{Y|mD}/\mathcal{Z}^{r-1}_{Y|mD}\}}_m \otimes {\{\left(\omega^{d-r+1}_Y/\mathcal{Z}^{d-r+1}_Y\right)\otimes_{ \mathscr{O}_Y} \mathscr{O}_Y(m'D)\}}_{m'}\\
   &\xrightarrow{ (\rho^{q,r}_{2,mD}\otimes \rho^{e'-q, d-r+1}_{2,m'})^{-1} } {\{\gr^q_{T/V}\mathscr{M}^r_{m}\}}_m \otimes {\{\gr^{e'-q}_{T/V}M^l_{Y, m'D}\}}_{m'}\\
   &\xrightarrow{f^{q,r}_{m, D}} \mu' \otimes \left(\omega^{N}_Y/\mathcal{B}^N_Y\right)
\end{align*}}
sends a local section ${\{x_m \otimes \gamma_{1, m}\}}_m \otimes {\{y_{m} \otimes  \gamma_{2, m}\}}_m$ to $w\otimes_k (-1)^{N+r}q$.
\end{enumerate}
\end{Lem}
\begin{proof} 
The proof of these assertions are straight forward. \end{proof}

\subsection{Proof of Theorem $6.1$}\mbox{}

In this subsection, we prove  Theorem \ref{D-U} by a similar arguments as for \S9.2, \cite{Sa2}.  Since the sheaves $U^1\mathscr{M}^r_{m}$ and $U^1M^{d-r+1}_{Y, m'D}$ are finitely sucessive extension of coherent $\left(\mathscr{O}_Y\right)^p$-modules, then we have the finiteness of the groups in the pairing in Theorem \ref{D-U} for each $m$. We introduce a descending filtration $Z^{\bullet}\mathscr{M}^q_{m}$ on $U^1\mathscr{M}^q_{m}$ and $Z^{\bullet}M^q_{Y, m'D}$ on $U^1M^q_{Y, m'D}$ defined as follows:
\[ Z^s\mathscr{M}^q_{m}:=\begin{cases}U^n\mathscr{M}^q_{m}\quad({\rm if}\;s \equiv1\mod 3, n=(s+2)/3), \\T^n\mathscr{M}^q_{m}\quad({\rm if}\;s \equiv2\mod 3, n=(s+1)/3),\\V^n\mathscr{M}^q_{m}\quad({\rm if}\;s \equiv0\mod 3, n=s/3), \end{cases}\]
\mbox{}\\
\[ Z^sM^q_{Y, m'D}:=\begin{cases}U^nM^q_{Y, m'D}\quad({\rm if}\;s \equiv1\mod 3, n=(s+2)/3),\\T^nM^q_{Y, m'D}\quad({\rm if}\;s \equiv2\mod 3, n=(s+1)/3)\\V^nM^q_{Y, m'D}\quad({\rm if}\;s \equiv0\mod 3, n=s/3).\end{cases}\]
\mbox{}\\
We have that the following Proposition.
\begin{Prop}(\cite{Sa2}, \S9.2)
Let $l$ be an integer $1 \leq l \leq 3e'-3$.
We have the following:

\noindent
(i)\;\hbox{There is a map}\quad \[H^N(Y, (U^1\mathscr{M}^l_{m}/Z^{l+1}\mathscr{M}_m)\otimes Z^{3e'-2-l}M^{l'}_{Y,m'D}\longrightarrow \mu.\]
Here this map is induced by a map 
\begin{align*}(\alpha)\quad\quad&H^N(Y,\; U^1\mathscr{M}^r_{m}\otimes Z^{3e'-2-r}M^{r'}_{Y,m'D})\\
&\longrightarrow H^N(Y, U^1\mathscr{M}^r_{m}\otimes U^1M^{r'}_{Y,m'D})\\&\longrightarrow \mu \otimes H^{N+1}(Y,\nu^N_Y)\cong \mu,
\end{align*}

\noindent
(ii)\;\hbox{There is a map}\[H^N(Y, \gr_{Z}\mathscr{M}^l_{m}\otimes \gr_Z^{3e'-2-l}M^{l'}_{Y,m'D})\longrightarrow \mu.\]

Here this map is induced by a map 

\begin{align*}&H^N(Y,\; \gr_Z^l\mathscr{M}^r_{m}\otimes Z^{3e'-2-r}M^{r'}_{Y,m'D})\\
&\longrightarrow H^N\left(Y, (U^1\mathscr{M}^r_{m}/Z^{r+1}\mathscr{M}_m)\otimes Z^{3e'-2-r}M^{r'}_{Y,m'D}\right)\longrightarrow \mu.
\end{align*}

(iii)\;\hbox{There is a map}\[H^N\left(Y,\;\;\frac{(U^1\mathscr{M}^r_{m}/Z^{l+1}\mathscr{M}_m)\otimes Z^{3e'-2-l}M^{l'}_{Y,m'D}}{\gr_{Z}\mathscr{M}^l_{m}\otimes Z^{3e'-2-l}M^{l'}_{Y,m'D}}\right)\longrightarrow \mu.\]

We put $\mathfrak{I}:= \frac{(U^1\mathscr{M}^r_{m}/Z^{l+1}\mathscr{M}_m)\otimes Z^{3e'-2-l}M^{l'}_{Y,m'D}}{\gr_{Z}\mathscr{M}^l_{m}\otimes Z^{3e'-2-l}M^{l'}_{Y,m'D}}$ for simplicity.\;
If $r \geq 2$, there is a commutative diagram:

{\footnotesize\[ \xymatrix@M=9pt{H^N(Y,\; \gr_{Z}\mathscr{M}^l_{m}\otimes \gr_Z^{3e'-2-l}M^{l'}_{Y,m'D})\oplus H^N(Y,\; \frac{U^1\mathscr{M}^r_{m}}{Z^{l+1}\mathscr{M}_m}\otimes Z^{3e'-2-l}M^{l'}_{Y,m'D})\ar[d]^-{(ii)\oplus(i)}\ar[r]&H^N(Y,\;\mathfrak{I})\ar[d]^-{(iii)}\\
\mu\oplus\mu\ar[r]^-{{\rm product}}&\mu.
}\]}

The map $(iii)$ is induced by the map $(i)$.

\end{Prop}
\begin{proof}
The proof is the same way as for Lemma 9.2.1 in \cite{Sa2}. We show only the case (i). One can show that the cases (ii) and (iii) are similar way.
There is a map
\begin{align*}
F:\;H^N\left(Y,\;Z^l \mathscr{M}_m^r\otimes Z^{3e'-2-l}M^{r'}_{Y,m'D}\right)&\longrightarrow H^N\left(Y,\;U^1\mathscr{M}_m^r\otimes Z^{3e'-2-l}M^{r'}_{Y,m'D}\right)\\
&\xrightarrow{(\alpha)} \mu. 
\end{align*}
Theorem \ref{ERL} implies that $F=\chi\circ f^{q,r}_D$.
We consider the following short exact sequence
\begin{align*}0\longrightarrow Z^{l+1} \mathscr{M}_m^r\otimes Z^{3e'-2-l}M^{r'}_{Y,m'D}\longrightarrow &U^1 \mathscr{M}_m^r\otimes Z^{3e'-2-l}M^{r'}_{Y,m'D}\\
& \longrightarrow \left(U^1\mathscr{M}_m^r/Z^{l+1} \mathscr{M}_m^r\right)\otimes Z^{3e'-2-l}M^{r'}_{Y,m'D}\longrightarrow 0.\end{align*}

By this exact sequence, we have the following diagram:

{\tiny\[ \xymatrix@M=4pt{
H^N\left(Z^{l+1} \mathscr{M}_m^r\otimes Z^{3e'-2-l}M^{r'}_{Y,m'D}\right)
\ar[r]^-{f_1}& H^N\left(U^1 \mathscr{M}_m^r\otimes Z^{3e'-2-l}M^{r'}_{Y,m'D}\right)\ar[r]^-{f_2}\ar[d]^-{g_1}& H^N\left(U^1 \mathscr{M}_m^r\otimes Z^{3e'-2-l}M^{r'}_{Y,m'D}\right)\ar[ld]^-{g_2}\\
&\mu.&
}\]}
Here the map $f_1$ is surjective by SubLemma 8.5.2 (2) in \cite{Sa2} and $g_1 \circ f_1$ is a zero map by $f^{q,r}_D$ annihilates the subsheaf $Z^{l+1} \mathscr{M}_m^r\otimes Z^{3e'-2-l}M^{r'}_{Y,m'D}$  by Lemma \ref{Lem maps}\;(1). Then we get the map $g_2$. 
\end{proof}

%%%%%%%%%%%%%%%%%%%%%%%%%%%%%%%%%%%%%%%%%%%%%%%%%%%%%%%%%%%%%%%%%%%
 By the same arguments as in \cite{Sa2}, $\S9.2$, we obtain the following pairings and Lemma  below:
\[b^{i,l}_D : H^i(Y, U^1\mathscr{M}^r_{m}/Z^{l+1}\mathscr{M}^r_{m}) \times H^{N-i}(Y, Z^{3e'-2-l}M^{d-r+1}_{Y, m'D}) \longrightarrow \mu,\]
\[c^{i,l}_D : H^i(Y, \gr^l_Z\mathscr{M}^r_{m}) \times H^{N-i}(Y, \gr^{3e'-2-l}_Z M^{d-r+1}_{Y, m'D}) \longrightarrow \mu,\]
for $i$ and $l$ with $1 \leq l \leq 3e'-3$.

\begin{Lem}(cf. Lemma 9.2.7, \cite{Sa2}) The pairing $c^{i, l}_D$ is non-degenerate for any $i$ and $l$ with $1 \leq l \leq 3e'-3$. 
\end{Lem}
\begin{proof} See the proof of Lemma 9.2.7 in \cite{Sa2} for the brief explanation of a linear Cartier isomorphism $C_{{\rm lin}}$. We consider the following maps:
\[{F_{Y/s}}_*(\omega^{r-1}_{Y|mD}/\mathcal{Z}^{r-1}_{Y|mD}) \times {F_{Y/s}}_*(\omega^{d-r}_{Y}/\mathcal{Z}^{d-r}_{Y} \otimes \mathscr{O}_Y(mD))\longrightarrow \omega^{d}_{\tilde{Y}},\]\[(x\otimes \gamma_1,y\otimes \gamma_2)\mapsto C_{\lin}(\gamma_1 \cdot \gamma_2\cdot dx \land y),\]
 \[{F_{Y/s}}_*(\omega^{r-1}_{Y|mD}/\mathcal{Z}^{r-1}_{Y|mD}) \times {F_{Y/s}}_*(\omega^{d-r+1}_{Y}/\mathcal{Z}^{d-r+1}_{Y} \otimes \mathscr{O}_Y(mD))\longrightarrow \omega^{d}_{\tilde{Y}},\]\[(x\otimes \gamma_1,y\otimes \gamma_2)\mapsto C_{\lin}(\gamma_1 \cdot \gamma_2\cdot x \land y),\] 
  \[{F_{Y/s}}_*(\omega^{r-2}_{Y|mD}/\mathcal{Z}^{r-2}_{Y|mD}) \times {F_{Y/s}}_*(\omega^{d-r+1}_{Y}/\mathcal{Z}^{d-r+1}_{Y} \otimes \mathscr{O}_Y(mD))\longrightarrow \omega^{d}_{\tilde{Y}},\]\[(x\otimes \gamma_1,y\otimes \gamma_2)\mapsto C_{\lin}(\gamma_1 \cdot \gamma_2\cdot x \land dy).\]
  By 3.1, \cite{Hy} and the Serre-Hartshorne duality, $\omega^d_{\tilde{Y}}$ is a dualizing sheaf on $\tilde{Y}$ in the sense of Definition, p.241, \cite{Har}. We put $q$ the maximal integer such that $3(q-1)<l$. 
 Then the pairing 
{\footnotesize\[ \xymatrix@M=10pt{ H^N(Y, \gr^l_Z \mathscr{M}^r_m) \times H^{N-i}(Y, \gr^{3e'-2-l}_Z M^{d-r+1}_{Y, m'D})\ar[r]^-{f^{q, l}_{mD}}& \mu\otimes H^N(Y, \omega^N_Y/ \mathcal{B}^N_Y)\ar[r]^-{\id \otimes \tr'_Y}&\mu \otimes k }\]}
 is a non-degenerate pairing. Here $\tr'_Y$ is the $k$-linear trace map (see \S9.3, \cite{Sa2}). We show the non-degeneracy of $c^{i, l}_D$ by the commutative diagrams:
 {\footnotesize\[ \xymatrix@M=10pt{ H^N(Y, \gr^l_Z\mathscr{M}^r_m \otimes \gr^{3e'-2-l}_Z M^{d-r+1}_{Y, m'D})  \ar[r]^-{f^{q,l}_{mD} }\ar[d]^-{(*)}&  \mu \otimes H^N(Y, \omega^N_Y/ \mathcal{B}^N_Y) \ar[r]^-{\id \otimes \tr'_Y} \ar[d]^-{\chi} & \mu \otimes k \ar[d]^-{\id \otimes \tr_{k/ \mathbb{F}_p} }\\
  \mu & \mu \otimes H^{N+1}(Y, \nu^N_Y) \ar[l]^-{\id \otimes \tr_Y} \ar[r]^-{\id \otimes \tr_Y}& \mu \otimes \mathbb{F}_p, } \]}
where the left square commutes by Theorem \ref{ERL} and the right square commutes by Remark 2.2.6(4), \cite{Sa2} and 3.4.1, \cite{Sa1}. The morphism $(*)$ is given by the same arguments as for Lemma 9.2.1 (2), \cite{Sa2}. 
This completes the proof of Lemma  and Theorem \ref{D-U}. 
\end{proof}

%%%%%%%%%%%%%%%%%%%%%%%%%%%%%%%%%%%%%%%%%%%%%%%%%%%%%%%%%%%%%%%%%%%%%%%%%%%%%%%%5
\section{\textbf{Proof of Theorem $\ref{Main Thm}$}}
First, we define the following pairing, which is a generalization of Milne's pairing of two-term complexes:
\begin{Def}(cf.  \cite{Mil}, p.175, \cite{Mil2}, p.271)\\
Let \[\mathscr{F}^{\bullet}:=\left[\mathscr{F}^0\xrightarrow{d^{0}_{\mathscr{F}}}\mathscr{F}^1\xrightarrow{d^{1}_{\mathscr{F}}}\cdots \xrightarrow{d^{n-1}_{\mathscr{F}}}\mathscr{F}^{n}\rightarrow \cdots\right],\]
\[ \mathscr{G}^{\bullet}:=\left[\mathscr{G}^0\xrightarrow{d^{0}_{\mathscr{G}}}\mathscr{G}^1\xrightarrow{d^{1}_{\mathscr{G}}}\cdots \xrightarrow{d^{n-1}_{\mathscr{G}}}\mathscr{G}^{n}\rightarrow \cdots\right],\]
and  \[\mathscr{H}^{\bullet}:=\left[\mathscr{H}^0\xrightarrow{d^{0}_{\mathscr{H}}}\mathscr{H}^1\xrightarrow{d^{1}_{\mathscr{H}}}\cdots \xrightarrow{d^{n-1}_{\mathscr{H}}}\mathscr{H}^{n}\rightarrow \cdots\right]\]
be complexes of sheaves. A pairing of complexes $\langle\cdot, \cdot \rangle:\mathscr{F}^{\bullet} \times \mathscr{G}^{\bullet} \longrightarrow \mathscr{H}^{\bullet}$
is a system of pairings
\[\{\langle\cdot, \cdot \rangle_{i, j}^{k}: \mathscr{F}^{i} \times \mathscr{G}^{j} \longrightarrow \mathscr{H}^{k}\}_{i,j,k}\]
such that
{\footnotesize\[(\diamond)\quad\quad \sum_{i+j=n} d_{\mathscr{H}}^n\left(\left\langle x_i, y_j\right\rangle^n_{i, j} \right)=\sum_{i+j=n}\left\{\left\langle d^i_{\mathscr{F}}(x_i), y_j \right\rangle^{n+1}_{i+1, j}+(-1)^i \left\langle x_i, d^j_{\mathscr{G}}(y_j) \right\rangle^{n+1}_{i, j+1}\right\}\]}
for all $n\geq 0$, $x_i \in \mathscr{F}^{i}$, and $y_j \in \mathscr{G}^{j} $.
\end{Def}

\begin{Rem}
Milne defines a pairing of two term complexes in \cite{Mil}, p.$175$ or \cite{Mil2}, p.$217$. In \cite{JSZ}, they have used  Milne's two-term pairing to prove their results.

\end{Rem}
\noindent
The aim of this section is to construct the following pairing
\begin{equation}\lim_{\stackrel{\longleftarrow}{m}}H^i(Y, \mathfrak{T}_n(r)_{X|mD}^{syn})  \times H_{Y \cap U}^{2d+1-i}(U, \mathfrak{T}_n(d-r)_U) \longrightarrow \mathbb{Z}/p^n, \end{equation}
which  is equivalent to giving the following pairing by the  proper base change theorem:
\begin{equation}\lim_{\stackrel{\longleftarrow}{m}}H^i(Y, \mathfrak{T}_n(r)_{X|mD}^{syn}) \times H^{2d+1-i}(Y, Ri^{!}R\alpha_*\mathfrak{T}_n(d-r)_U) \longrightarrow \mathbb{Z}/p^n.\end{equation}
For this purpose, we construct a morphism 
\[(\Upsilon)\quad{\{\mathfrak{T}_n(r)_{X|mD}^{syn}\}}_m \otimes^{\mathbb{L}} Ri^{!}R\alpha_*\mathfrak{T}_n(d-r)_U \longrightarrow \nu^{d-1}_{Y,n}[-d-1].\]

\noindent
We first consider a presentation as complex of $Ri^{!}R\alpha_*\mathfrak{T}_n(d-r)_U$.
We have $Ri^{!}R\alpha_*\mathfrak{T}_n(d-r)_U=R\alpha'_*Ri'^{!}\mathfrak{T}_n(d-r)_U$ since the diagram in {\textbf Setting} ($\clubsuit$) is commutative. There is a spectral sequence $E_2^{p,q}=R^p\alpha'_*\mathcal{H}^q(Ri'^{!}\mathfrak{T}_n(d-r)_U) \Longrightarrow \mathcal{H}^{p+q}R\alpha'_*Ri'^{!}\mathfrak{T}_n(d-r)_U$ and we have $E_2^{p,q}=0$ for all $p \geq 2$. We introduce an augmented complex $s_n(q)^{\vee,\;\bullet}_{X|m'D}$ to construct a pairing $(\Upsilon)$:
\begin{Def}(cf.\;\cite{Tsu0},\;\cite{Tsu1},\;\cite{Tsu2})\\
We define a complex $s_n(q)^{\vee,\;\bullet}_{X|m'D}$, which degree $i$-part defined by {\footnotesize\[{s_n(q)^{\vee,\;i}_{X|m'D}}:=\left(\omega^i_{Z_n}\otimes \mathscr{O}_{Z_n}((m'+i)\mathscr{D}_n)\otimes J^{[q-i]}_{\mathscr{E}_n}\right)\oplus \left(\omega^{i-1}_{Z_n}\otimes \mathscr{O}_{Z_n}((m'+i)\mathscr{D}_n)\otimes \mathscr{O}_{\mathscr{E}_n} \right).\] }Its differential is given by $d^i(x,y):=(dx, (1-\varphi_q)x-dy)$. Here we take
$m'$ satisfies $m'+i-m \leq 0$. \end{Def}

We consider a complex\;(as a pro-system)
{\tiny\[\mathscr{G}_{m'}^{\bullet}:= \left[ \frac{ \CoKer\Big({s_n(d-r)}^{\vee\;d-r-1}_{X|m'D}\longrightarrow {s_n(d-r)}^{\vee\;d-r}_{X|m'D}\Big)}{{\rm \mathfrak{Q}}}\xrightarrow{\frak{d}}{s_n(d-r)}^{\vee\;d-r+1}_{X|m'D}\longrightarrow  {s_n(d-r)}^{\vee\;d-r+2}_{X|m'D}\longrightarrow \cdots \right], \]}where $\mathfrak{Q}$ is $\Ker\left(\mathcal{H}^{d-r}({s_n(d-r)}^{\vee\;\bullet}_{X|m'D})\longrightarrow \nu^{d-r-1}_{Y,n}\otimes_{\mathscr{O}_Y}\mathscr{O}_{Y}(m'D_s)\right)$. Here the first term of $\mathscr{G}_{m'}^{\bullet}$ is a degree $d-r+1$. The transition map of $\mathscr{G}_{m'}^{\bullet}$ is a differential $d$ of ${s_n(d-r)}^{\vee\;\bullet}_{X|m'D}$. We check the well-definedness of the map $\frak{d}$. It is easy to see that the map $d^{d-r}$ induces $\CoKer\Big({s_n(d-r)}^{\vee\;d-r-1}_{X|m'D}\longrightarrow {s_n(d-r)}^{\vee\;d-r}_{X|m'D}\Big) \longrightarrow {s_n(d-r)}^{\vee\;d-r+1}_{X|m'D}$. Then we obtain the following  representation (quasi-isomorphism) of $R\alpha'_*Ri'^{!}\mathfrak{T}_n(d-r)_U$ by using the calculation of the cohomology sheaf of $Ri'^{!}\mathfrak{T}_n(d-r)_U$ in \cite{Sa2}:
\[R\alpha'_*Ri'^{!}\mathfrak{T}_n(d-r)_U
\cong \lim_{\stackrel{\longrightarrow}{m'}}\mathscr{G}_{m'}^{\bullet}.\]
Taking the cohomology of the right hand side, one can check this quasi-isomorphism.

Next we consider $\mathfrak{T}_n(r)_{X|mD}^{syn}$. By the definition  and Lemma . 
We have the following presentation of $\mathfrak{T}_n(r)_{X|mD}^{syn}$:
{\footnotesize \begin{align*}
\mathfrak{T}_n(r)_{X|mD}^{syn}=\Big[s_n(r)_{X|mD}^0\rightarrow s_n(r)_{X|mD}^1&\rightarrow \cdots
\rightarrow  s_n(r)_{X|mD}^{r-1}\\
& \rightarrow \Ker\left(\Ker(s_n(r)_{X|mD}^r\rightarrow s_n(r)_{X|mD}^{r+1}) \rightarrow \nu_{Y|mD,n}^{r-1}\Big) \right].\end{align*}}
We define the following maps by a similar way as for product structure of syntomic complex (cf.\cite{Tsu1}, \S2.2):
{\footnotesize\begin{multline}s_n(r)^i_{X|mD} \otimes^{\mathbb{L}} {s_n(d-r)}^{\vee\;i'}_{X|m'D}\\ \longrightarrow \left(\omega^{i+i'}_{Z_n}\otimes \mathscr{O}_{Z_n}((m'+i'-m)\mathscr{D}_n)\otimes J^{[d-i-i']}_{\mathscr{E}_n}\right)\oplus \left(\omega^{i+i'-1}_{Z_n}\otimes \mathscr{O}_{Z_n}((m'+i'-m)\mathscr{D}_n)\otimes \mathscr{O}_{\mathscr{E}_n}\right).\end{multline}}
Since $m'+i'-m \leq 0$ by assumption of $m'$,  the target of this map is contained in ${\mathcal{S}_n(d)}_{(X,M_X)}^{i+i'}$.

\begin{Lem} We have the morphism 
\[{\{\mathfrak{T}_n(r)_{X|mD}^{syn}\}}_m \otimes^{\mathbb{L}} Ri^{!}R\alpha_*\mathfrak{T}_n(d-r)_U \longrightarrow \mathscr{H}^{\bullet}.\]
\end{Lem}
\begin{proof}
We put 
{\tiny\[\mathscr{F}_m^{\bullet}:=\left[s_n(r)_{X|mD}^0 \rightarrow s_n(r)_{X|mD}^1\rightarrow \cdots \rightarrow  s_n(r)_{X|mD}^{r-1} \rightarrow \Ker\left(\Ker(s_n(r)_{X|mD}^r\rightarrow s_n(r)_{X|mD}^{r+1}) \rightarrow \nu_{Y|mD,n}^{r-1}\right)\right],\]}
{\tiny\[\mathscr{G}_{m'}^{\bullet}:= \left[ \frac{ \CoKer\Big({s_n(d-r)}^{\vee\;d-r-1}_{X|m'D}\longrightarrow {s_n(d-r)}^{\vee\;d-r}_{X|m'D}\Big)}{{\rm \mathfrak{Q}}}\longrightarrow {s_n(d-r)}^{\vee\;d-r+1}_{X|m'D}\longrightarrow \cdots \right], \]}
and
\[\mathscr{H}^{\bullet}:=\left[\frac{ \CoKer\left({\mathcal{S}_n(d)}_{(X,M_X)}^{d-1}\rightarrow {\mathcal{S}_n(d)}_{(X,M_X)}^{d}\right) }{{FM}_{X}^d}\rightarrow {\mathcal{S}_n(d)}_{(X,M_X)}^{d+1}\rightarrow \cdots \right].\]
Here ${\mathcal{S}_n(d)}_{(X,M_X)}^{\bullet}$ is syntomic complex defined by Tsuji (see \cite{Tsu0}, \cite{Tsu1} and \cite{Tsu2}).
The pairing $\langle \cdot, \cdot \rangle: \mathscr{F}_m \times \mathscr{G}_{m'} \longrightarrow \mathscr{H}$ is a product map of syntomic complexes defined in Definition .
We have to check the pairing $\langle,\cdot, \cdot \rangle: \mathscr{F}_m \times \mathscr{G}_{m'} \longrightarrow \mathscr{H}$ satisfies the relation $(\diamond)$ for each $m$.
It is enough to show that the equation \[(*)\quad d_{\mathscr{H}}^n (\left\langle x_i, y_j\right\rangle^n_{i, j}) =\left\langle d^i_{\mathscr{F}_m}(x_i), y_j \right\rangle^{n+1}_{i+1, j}+(-1)^i \left\langle x_i, d^j_{\mathscr{G}_{m'}}(y_j) \right\rangle^{n+1}_{i, j+1}.\]
We will check the relation $(*)$ for the following two cases {\textbf (I), (II)}.  \\

{\textbf(I)}\; The case $s_n(r)_{X|mD}^i \otimes {s_n(d-r)}^{\vee\;j}_{X|m'D} \longrightarrow {\mathcal{S}_n(d)}_{(X,M_X)}^{i+j}$.\\
The calculation of this case is similar as for product structure of syntomic complexes\;(\cite{Tsu1}). 
We denote by $x_i=(x_{i1}, x_{i2}) \in s_n(r)_{X|mD}^i$, $y_j=(y_{j1}, y_{j2}) \in {s_n(d-r)}^{\vee\;j}_{X|m'D}$.
We have 
\begin{align*}
d&_{\mathscr{H}}^n (\left\langle x_i, y_j\right\rangle^n_{i, j})=d_{\mathscr{H}}^n \left(x_{i1}y_{j1},\quad(-1)^ix_{i1}y_{j2}+x_{i2}\varphi_{d-r}(y_{j1})\right)\\
&=\left(d(x_{i1}y_{j1}),\;(1-\varphi_d)(x_{i1}y_{j1})-d\{(-1)^ix_{i1}y_{j2}+x_{i2}\varphi_{d-r}(y_{j1})\}\right)\\
&=\left((-1)^ix_{i1}dy_{j1}+y_{j1}dx_{i1},\quad(1-\varphi_d)(x_{i1}y_{j1})-(-1)^id(x_{i1}y_{j2})-d(x_{i2}\varphi_{d-r}(y_{j1}))\right).
\end{align*}
On the other hand, we can calculate the RHS of $(*)$ as follows:
{\footnotesize\begin{align*}
&\left\langle d^i_{\mathscr{F}_m}(x_i), y_j \right\rangle^{n+1}_{i+1, j}+(-1)^i \left\langle x_i, d^j_{\mathscr{G}_{m'}}(y_j) \right\rangle^{n+1}_{i, j+1}\\
&=\left\langle\left(dx_{i1},\;\;(1-\varphi_r)(x_{i1})-dx_{i2}\right),\quad y_j \right\rangle^{n+1}_{i+1,j}+(-1)^i\left\langle x_i,\quad (dy_{j1},\;\;(1-\varphi_{d-r})(y_{j1})-dy_{j2}) \right\rangle^{n+1}_{i,j+1}\\
&=\left((-1)^ix_{i1}dy_{j1}+y_{j1}dx_{i1},\quad(1-\varphi_d)(x_{i1}y_{j1})-(-1)^id(x_{i1}y_{j2})-d(x_{i2}\varphi_{d-r}(y_{j1}))\right).
\end{align*}}
Then this case, the map $s_n(r)_{X|mD}^i \otimes {s_n(d-r)}^{\vee\;j}_{X|m'D} \longrightarrow {\mathcal{S}_n(d)}_{(X,M_X)}^{i+j}$ satisfies the relation $(*)$.

{\textbf(II)}\; The case  {\footnotesize\begin{align*}&\Ker\left(\Ker(s_n(r)_{X|mD}^r\rightarrow s_n(r)_{X|mD}^{r+1}) \rightarrow \nu_{Y|mD,n}^{r-1}\right)\otimes \frac{\CoKer\Big({s_n(d-r)}^{\vee\;d-r-1}_{X|m'D}\longrightarrow {s_n(d-r)}^{\vee\;d-r}_{X|m'D}\Big)}{\mathfrak{Q}}\\
 & \longrightarrow \frac{\CoKer\left({\mathcal{S}_n(d)}_{(X,M_X)}^{d-1}\rightarrow {\mathcal{S}_n(d)}_{(X,M_X)}^{d}\right)}{FM^d_{X,n}}.\end{align*}}
First we check the well-definedness of the map (we denote by this map $(\bigstar)$)
{\footnotesize \begin{align*}&\Ker\left(\Ker(s_n(r)_{X|mD}^r\rightarrow s_n(r)_{X|mD}^{r+1}) \rightarrow \nu_{Y|mD,n}^{r-1}\right)\otimes \CoKer\Big({s_n(d-r)}^{\vee\;d-r-1}_{X|m'D}\longrightarrow {s_n(d-r)}^{\vee\;d-r}_{X|m'D}\Big)\\
 & \longrightarrow \CoKer\left({\mathcal{S}_n(d)}_{(X,M_X)}^{d-1}\rightarrow {\mathcal{S}_n(d)}_{(X,M_X)}^{d}\right).\end{align*}}
The element $(z_1, z_2)$ is $\Big(d(z'_1), \;\;(1-\varphi_{d-r})(z'_1)-d(z'_2)\Big)$ for some $(z'_1,z'_2) \in {s_n(d-r)}^{\vee\;d-r-1}_{X|m'D}$. We calculate the image of $(x_1,x_2)\otimes (z_1, z_2)$ under the product map:
\begin{align*}&(x_1,x_2)\otimes (z_1, z_2)=(x_1,x_2)\otimes \Big(d(z'_1), \;\;(1-\varphi_{d-r})(z'_1)-d(z'_2)\Big)\\
& \mapsto\Big(x_1d(z'_1),\;\;(-1)^r x_1\cdot\{(1-\varphi_{d-r})z'_1-d(z'_2)\}+x_2\varphi_{d-r}(dz'_1) \Big).
\end{align*}
It is enough to find a solution $(\eta_1,\;\eta_2) \in $ satisfies the following two equations:
\begin{align*}
&(1)\;d\eta_1=x_1d(z'_1),\\
 &(2)\;(1-\varphi_d)\eta_1-d\eta_2=(-1)^r x_1\cdot\{(1-\varphi_{d-r})z'_1-d(z'_2)\}+x_2\varphi_{d-r}(dz'_1).
\end{align*}
We take $\eta_1:=x_1z'_1$, this element satisfies the condition $(1)$ because $d(x_1)=0$. Next we put $\eta_2:=(-1)^r\{x_1z'_2-x_2\varphi_{d-r}(z'_1)\}$, this satisfies the condition $(2)$ by the direct calculation. Then the pair $(\eta_1, \eta_2)$ is the solution of  the equation $(1)$ and $(2)$. Lastly, we have the following map, which induced by the map $(\bigstar)$: 
 {\footnotesize\begin{align*}&\Ker\left(\Ker(s_n(r)_{X|mD}^r\rightarrow s_n(r)_{X|mD}^{r+1}) \rightarrow \nu_{Y|mD,n}^{r-1}\right)\otimes \frac{\CoKer\Big({s_n(d-r)}^{\vee\;d-r-1}_{X|m'D}\longrightarrow {s_n(d-r)}^{\vee\;d-r}_{X|m'D}\Big)}{\mathfrak{Q}}\\
 & \longrightarrow \frac{\CoKer\left({\mathcal{S}_n(d)}_{(X,M_X)}^{d-1}\rightarrow {\mathcal{S}_n(d)}_{(X,M_X)}^{d}\right)}{FM^d_{X,n}}.\end{align*}}
Here we used the fact that $FM^d_{X,n} \cong \Ker\left(\mathcal{H}^d\left({\mathcal{S}_n(d)}_{(X,M_X)}^{\bullet}\right)\rightarrow \nu^{d-1}_{Y,n}\right)$\;(see \cite{Sa3}, Theorem 3.4).
We check that this map $(\bigstar)$ satisfies the relation $(*)$.  It suffices to show that the relation 
$d_{\mathscr{H}}^n\left(\left\langle x_i, y_j\right\rangle^n_{i, j} \right)= (-1)^i \left\langle x_i, d^j_{\mathscr{G}_{m'}}(y_j) \right\rangle^{n+1}_{i, j+1}$. This relation follows from straightforward computations.
Then we have a morphism ${\{\mathfrak{T}_n(r)_{X|mD}^{syn}\}}_m \otimes^{\mathbb{L}} Ri^{!}R\alpha_*\mathfrak{T}_n(d-r)_U \longrightarrow \mathscr{H}^{\bullet}$. For the complex $\mathscr{H}^{\bullet}$, we have $\mathcal{H}^{d+1}(\mathscr{H}^{\bullet})[-d-1]\cong \nu_{Y,n}^{d-1}[-d-1]$.  Thus we obtain the map $(\Upsilon)$ by the following composition:
\[{\{\mathfrak{T}_n(r)_{X|mD}^{syn}\}}_m \otimes^{\mathbb{L}} Ri^{!}R\alpha_*\mathfrak{T}_n(d-r)_U \longrightarrow \mathscr{H}^{\bullet} \rightarrow \mathcal{H}^{d+1}(\mathscr{H}^{\bullet})[-d-1]\cong \nu_{Y,n}^{d-1}[-d-1]\]\end{proof} We obtain a pairing \begin{equation}\lim_{\stackrel{\longleftarrow}{m}}H^i(Y, \mathfrak{T}_n(r)_{X|mD}^{syn})  \times H_{Y \cap U}^{2d+1-i}(U, \mathfrak{T}_n(d-r)_U) \longrightarrow \mathbb{Z}/p^n.\end{equation} This completes the proof of the main result \ref{Main Thm}\;(1). In the next section, we will prove that the main result \ref{Main Thm}\;(2).  

%%%%%%%%%%%%%%%%%%%%%%%%%%%%%%%%%%%%%%%%%%%%%%%%%%%%%%%%%%%%%%%%%%%%%%%%%%%%%%%%%%%%%%

%%%%%%%%%%%%%%%%%%%%%%%%%%%%%%%%%%%%%%%%%%%%%%%%%%%%%%%%%%%%%%%%%%%%%%%%%%%%%%%%
\section{\textbf{Proof of non-degeneracy of the pairing in Theorem $\ref{Main Thm}$}}
We start the proof of Theorem \ref{Main Thm}(2) by the same strategy of \cite{Sa2} in \S $10$.  We can reduce the problem to the case $n=1$ by the distinguished triangle in Proposition .

\subsection{Descending induction on $r$}\mbox{}
We assume that $\zeta_p \in K$. We obtain the following Lemma by the same arguments as in  \cite{Sa2}, Proposition 4.3.1 (3), and by Lemma \ref{Lem1}.
\begin{Lem}(Bockstein triangle,\;cf. \cite{Sa2}, Proposition 4.3.1 (3))
There is a distinguished triangle as a projective system with respect to the multiplicities of an irreducible components of $D$. 
\[\left\{\mathfrak{T}_{n}(r)^{{\rm syn}}_{X|D}\right\}_D \xrightarrow{\mathcal{R}}\left\{\mathfrak{T}_{n-1}(r)^{{\rm syn}}_{X|D}\right\}_D \xrightarrow{\delta_{n,1}}\left\{ \mathfrak{T}_{1}(r)^{{\rm syn}}_{X|D}[1]\right\}_D\xrightarrow{\underline{p}^n[1]}\left\{\mathfrak{T}_{n}(r)^{{\rm syn}}_{X|D}[1]\right\}_D \]
\end{Lem} 
\mbox{}\\
Therefore it suffices to show that the case $n=1$.
We show this case of (2) in Theorem \ref{Main Thm} by induction on $r$. 
If $r=0$, the pairing in \ref{Main Thm} $(2)$ is isomorphic to the pairing 
\[\plim[m]H^i(Y, \mathfrak{T}_1(0)_{X|mD}^{syn}) \times H^{2d+1-i}(Y, \alpha'_*\nu^{d+1}_{V,1}) \longrightarrow \mathbb{Z}/p\mathbb{Z}\]
by the proper-base change theorem and Lemma 7.3.3, \cite{Sa2}.
We put $\nu_Y^r:={\nu_Y^r}_1$, $\mathfrak{T}(r)_{X|mD}^{syn}:={\mathfrak{T}_1(r)}_{X|mD}^{syn}$, $\mu:=\mu_p(K)$, $\mu':=i^*\psi_*\mu_{p}$ and $\mu'':=i'^*j'_*\mu_{p}$ for simplicity. Note that $\mu'$ is the constant \'etale sheaf on $Y$ associated with the abstract group $\mu (\cong \mathbb{Z}/p\mathbb{Z})$. We define the morphism
\begin{equation}
\ind_n: \mu' \otimes^{\mathbb{L}} \mathfrak{T}(r-1)_{X|mD}^{syn}\longrightarrow  \mathfrak{T}(r)_{X|mD}^{syn}
\end{equation}
\noindent
by restricting the product structure (cf. Proposition ) {\footnotesize $ \mathfrak{T}(1)_{X|mD}^{syn} \otimes^{\mathbb{L}} \mathfrak{T}(r-1)_{X|mD}^{syn}\longrightarrow  \mathfrak{T}(r)_{X|mD}^{syn}$} to the $0$-th cohomology sheaf $\mu'$ of  $ \mathfrak{T}(1)_{X|mD}^{syn}$.  We have the following Lemma (cf. \cite{Sa2}, {\rm Lemma 10.4.1}).

\begin{Lem}
Let
\begin{equation}\label{Kdist}
\mathbb{K}(r)_m\stackrel{b_r}{ \longrightarrow }\mu' \otimes^{\mathbb{L}}\mathfrak{T}(r-1)_{X|mD}^{syn} \longrightarrow \mathfrak{T}(r)_{X|mD}^{syn}\stackrel{a_r}{ \longrightarrow }\mathbb{K}(r)_m[1]
\end{equation}
be a distinguished triangle in $D^b\big(Y_{\acute{e}t}, \mathbb{Z}/p\mathbb{Z}\big)$. Then:\\
{\rm (1)} The triple $\big(\mathbb{K}(r), a_r, b_r\big)$ is unique up to a unique isomorphism in $D^b\big(Y_{\acute{e}t}, \mathbb{Z}/p\mathbb{Z}\big)$, and $b_r$ is determined by the pair $\big(\mathbb{K}(r), a_r\big)$.\\
{\rm (2)} $\mathbb{K}(r)_m$ is concentrated in $[r, r+1]$ and $a_r$ induces isomorphism
\begin{equation}
\mathcal{H}^q\big(\mathbb{K}(r)_m\big) \cong \begin{cases}
\mu' \otimes \nu_{Y|mD}^{r-2}& (q=r),\\
\Ker\Big(\mathcal{H}^r\big(s_1(r)_{X|mD}\big)\rightarrow \nu_{Y|mD}^{r-1}\Big)& (q=r+1).
\end{cases}
\end{equation}
\end{Lem}
\begin{proof}
The assertion {\rm (2)} follows from the long exact sequence of cohomology sheaves associated with (\ref{Kdist}). The details are straightforward and left to reader.
We have $\Hom(\mu' \otimes^{\mathbb{L}}\mathfrak{T}(r-1)_{X|mD}^{syn},\;\mathbb{K}(r)_m[-1])=0$ by (2) and Lemma 2.1.1 in \cite{Sa2}. Then we obtain the assertion {\rm (1)} by Lemma 2.1.2 (3) in \cite{Sa2}. 
\end{proof}
 In what follows, we fix a pair $\big(\mathbb{K}(r), a_r\big)$ fitting into () for each $r$ with $0 \leq r \leq p-2$.
 \begin{Lem}(\cite{Sa2}, {\rm Lemma 10.4.1})\mbox{}
 
 There is a distinguished triangle in $D^b\big(V_{\acute{e}t}, \mathbb{Z}/p\mathbb{Z}\big)$
 \begin{equation}
 \mathbb{M}(d-r)[-1]\stackrel{c_r}{\longrightarrow }\mathbb\mu''\otimes^{\mathbb{L}}Ri'^!\mathfrak{T}_1(d-r)_U \longrightarrow Ri'^!\mathfrak{T}_1(d-r+1)_U  \stackrel{d_r}{\longrightarrow }\mathbb{M}(d-r).
\end{equation}
{\rm (1)} The triple $\big(\mathbb{M}(d-r), c_r, d_r\big)$ is unique up to a unique isomorphism in $D^b\big(Y_{\acute{e}t}, \mathbb{Z}/p\mathbb{Z}\big)$, and $b_r$ is determined by the pair $\big(\mathbb{K}(r), c_r\big)$.\\
{\rm (2)} $\mathbb{M}(d-r)_m$ is concentrated in $[d-r, d-r+1]$ and $c_r$ induces isomorphism
\begin{equation}
\mathcal{H}^q\big(\mathbb{M}(d-r)\big) \cong \begin{cases}
\mu'' \otimes \nu_{V, 1}^{d-r-1}& (q=d-r),\\
FM_1^{d-r+1}& (q=d-r+1).
\end{cases}
\end{equation}
\end{Lem}\mbox{}

\noindent
Let us note that for objects $\mathcal{K}_1, \mathcal{K}_2 \in  D^-(V_{\acute{e}t}, \mathbb{Z}/p^n\mathbb{Z})$ , and $ \mathcal{K}_3 \in  D^+(V_{\acute{e}t}, \mathbb{Z}/p^n\mathbb{Z})$,
we have 
\begin{equation}
\Hom_{D(V_{\acute{e}t}, \mathbb{Z}/p\mathbb{Z})}\Big(\mathcal{K}_1\otimes^{\mathbb{L}}\mathcal{K}_2, \mathcal{K}_3\Big) \cong \Hom_{D(V_{\acute{e}t}, \mathbb{Z}/p\mathbb{Z})}\Big(\mathcal{K}_1,R\mathcal{H}om_{V,\mathbb{Z}/p\mathbb{Z}}\big(\mathcal{K}_2, \mathcal{K}_3\big) \Big).
\end{equation}
\noindent
We now introduce some notations. 
For $\mathcal{K} \in D^-(V_{\acute{e}t}, \mathbb{Z}/p^n\mathbb{Z})$, we define
\begin{equation}
\mathbb{D}(\mathcal{K}):=R\mathcal{H}om_{Y, \mathbb{Z}/p\mathbb{Z}}\big(\mathcal{K},  \mu'' \otimes \nu_V^{d-1}[-d-1]\big) \in D^+(V_{\acute{e}t}, \mathbb{Z}/p^n\mathbb{Z}).
\end{equation}
This notation is due to \cite{Sa2}.
\begin{Lem}\label{KLEM}
Then there is a unique morphism 
\begin{equation}
\Big(``\plim[m]"\mathbb{K}(r)_m\Big) \otimes^{\mathbb{L}} R\alpha'_*\mathbb{M}(d-r) \longrightarrow \mu' \otimes \nu_Y^{d-1}[-d-1] \quad in D^-(Y_{\acute{e}t}, \mathbb{Z}/p^n\mathbb{Z})
\end{equation}
whose adjoint morphism $``\plim[m]"\mathbb{K}(r)_m \rightarrow R\alpha'_*\mathbb{D}( \mathbb{M}(d-r))$ fits into a commutative diagram with with distinguished rows:
{\tiny\begin{equation*}
\xymatrix@M=9pt@C=6pt{
``\plim[m]"\mathbb{K}(r)_m\ar[r] \ar[d] &``\plim[m]"\mu' \otimes^{\mathbb{L}}T(r-1)_{X|mD}^{syn} \ar[r]\ar[d]_{(\bigtriangledown)}&``\plim[m]"T(r)_{X|mD}^{syn}\ar[r]\ar[d]_{(\bigtriangledown)}&``\plim[m]"\mathbb{K}(r)_m[1] \ar[d]\\
R\alpha'_*\mathbb{D}\big(\mathbb{M}(d-r)\big) \ar[r]&R\alpha'_*\mathbb{D}\big(\mathfrak{T}_1(d-r+1)_U\big) \ar[r]& R\alpha'_*\mathbb{D}\big(\mu' \otimes ^{\mathbb{L}}\mathfrak{T}_1(d-r)_U\big) \ar[r]& R\alpha'_*\mathbb{D}\big(\mathbb{M}(d-r)\big)[1].}
\end{equation*}}
Here lower raw arises from a distinguished triangle obtained by truncation and taking $R\alpha'_*$, and the vertical arrows $(\bigtriangledown)$ comes from the pairing $(\Upsilon)$ .
\end{Lem}
\begin{proof}
The assertion follows from Lemma 2.1.2 $(1)$ and the fact that 
{\footnotesize\begin{align*} \Hom_{D^{+}(Y_{\acute{e}t}, \mathbb{Z}/p\mathbb{Z})}&\left(\mathbb{K}(r)_m,\;\mathbb{D}(\mu' \otimes ^{\mathbb{L}}\mathfrak{T}_1(d-r)_U)\right) \cong\\
&\Hom_{D^{-}(Y_{\acute{e}t}, \mathbb{Z}/p\mathbb{Z})}\left(\mathbb{K}(r)_m\otimes ^{\mathbb{L}}(\mu' \otimes ^{\mathbb{L}}\mathfrak{T}_1(d-r)_U),\;\mu' \otimes \nu_Y^{d-1}[-d-1] \right)\\
&=0, 
\end{align*}}
where the last equality follows from Lemma 8.2 $(2)$ and Lemma 2.1.1 in \cite{Sa2}.
\end{proof}
We turn to the proof of the Theorem \ref{Main Thm}(2) and claim the following:
\begin{Prop}\label{Ms}
Let $0 \leq r \leq p-2$. Then for $i \in \mathbb{Z}$, the pairing 
\begin{equation}
\plim[m] H^{i}(Y, \mathbb{K}(r)_m)\times H^{2d+1-i}(Y, \alpha'_*\mathbb{M}(d-r)) \rightarrow \mu \otimes H^{d}(Y, \nu_Y^{d-1})\xrightarrow{id\otimes \tr_Y}\mu
\end{equation}
is a non-degenerate pairing of compact group and discrete group.
\end{Prop}
 We will prove this proposition \ref{Ms}  in the next subsection. We first complete the proof of Theorem \ref{Main Thm} (2) by descending induction on $n \leq d$, admitting Proposition.
We get the Theorem 1.3 (2) by applying \cite{Sa2}, {\rm Lemma 10.4.10} to the  diagram in Lemma \ref{KLEM}.

%%%% %%%%%%% %%% %%%%%%%%%%%% %%% %%% %%% %%%%% %%%%%% %%%%%% %%%%% %%%%%% %%%%%%% %%%%% %%%5
\subsection{Proof of Proposition} \mbox{} 

Let $\mathbb{V}(d-r)$ be an object of $D^b(V_{\acute{e}t}, \mathbb{Z}/p^n\mathbb{Z})$ fitting into a distinguished triangle
\begin{equation}
\lambda_{V,1}^{d-r+1}[-d+r-2]  \rightarrow\mathbb{V}(d-r) \rightarrow \mathbb{M}(d-r) \rightarrow \lambda_{V,1}^{d-r+1}[-d+r-1],
\end{equation}
where the last morphism is defined as the composite
\begin{equation}
\mathbb{M}(d-r) \rightarrow \mathcal{H}^{d-r+1}(\mathbb{M}(d-r))[-d+r-1]\cong {FM_{1,V}}^{d-r+1}[-d+r-1]\rightarrow  \lambda_{V,1}^{d-r+1}[-d+r-1].
\end{equation}
Here the last map is due to \cite{Sa2}.
By Lemma 8.2 (2) and Lemma 2.1.2(3) in \cite{Sa2}, $\mathbb{V}(d-r)$ is concentrated in $[d-r, d-r+1]$ and unique up to a unique isomorphism. We have
\begin{equation}
\mathcal{H}^q(\mathbb{V}(d-r)) \cong \begin{cases}
\mu' \otimes \nu_{V,1}^{d-r-1} & (q=d-r)\\
U^1M^{d-r+1}_V & (q=d-r+1).
\end{cases}
\end{equation}
\begin{Lem}(cf. \cite{Sa2}, Lemma 10.5.3)\label{LLEM1} There is a unique morphism
\begin{equation}
R\alpha'_*\mathbb{V}(d-r)[-1]\longrightarrow \mathbb{D}\big(``\plim[m]"F\mathscr{M}^{r}_m[-r]\big) \quad in\; D^+(Y_{\acute{e}t}, \mathbb{Z}/p\mathbb{Z})
\end{equation}
fitting into a commutative diagram with distinguished rows
{\tiny\begin{equation}\label{CD1}
\xymatrix@C=8pt@M=8pt{
R\alpha'_*\mathbb{V}(d-r)[-1]\ar[r]\ar[d]^{(1)} &R\alpha'_*\mathbb{M}(d-r)[-1]\ar[r]\ar[d]^{(2)} \ar@{}[rd]|{(\mathfrak{W})}& R\alpha'_*\lambda_{V,1}^{d-r+1}[-d+r-2] \ar[r]\ar[d]^{(3)}& R\alpha'_*\mathbb{V}(d-r)\ar[d]^{(4)}\\
 \mathbb{D}\big(``\plim[m]"F\mathscr{M}^{r}_m[-r]\big) \ar[r] & \mathbb{D}\big(``\plim[m]"\mathbb{K}(r)_m\big)  \ar[r]& \mathbb{D}\big(``\plim[m]"\mu' \otimes \nu_{Y|mD}^{r-2}[-r+1]\big)  \ar[r]& \mathbb{D}\big(``\plim[m]"F\mathscr{M}^{r}_m[-r]\big) [1] 
}\end{equation}}
Here the morphism $(2)$ obtained in Lemma \ref{KLEM} and the morphism $(3)$ obtained in section \S4. 
\end{Lem}
\begin{proof}
We have $\Hom\left(R\alpha'_*\mathbb{V}(d-r),\;\mathbb{D}\big(``\plim[m]"\mu' \otimes \nu_{Y|mD}^{r-2}[-r+1]\big)\right)=0$. It suffices to show that the commutativity of the central square of $(\ref{CD1})$ by Lemma 2.1.2 (1) in \cite{Sa2}. We consider the composition of the morphisms 
\begin{align*}\upsilon:\;&\left(R\alpha'_*\mathbb{M}(d-r)[-1]\right)\otimes^{\mathbb{L}}(``\plim[m]"\mu' \otimes \nu_{Y|mD}^{r-2}[-r+1])[-r+1]\\ &\longrightarrow \left(R\alpha'_*\mathbb{M}(d-r)[-1]\right)\otimes^{\mathbb{L}} (``\plim[m]"\mathbb{K}(r)_m) \longrightarrow  \mu'' \otimes \nu_V^{d-1}[-d-1]\end{align*}, 
where the last morphism follows from the morphism $(2)$. It remains to show that the commutativity of the central square $(\mathfrak{W})$. To prove the commutativity of the square $(\mathfrak{W})$, we first prove the following Lemma:
\begin{Lem}
We have the following commutative diagram\;(we denote by (CSQ)) in $D(Y_{\acute{e}t},\;\mathbb{Z}/p\mathbb{Z})$:
\[
\xymatrix@C=30pt@M=9pt{
R\alpha'_*i'^*\mathfrak{T}(d-r)_U[-1]\otimes^{\mathbb{L}} \{\mathfrak{T}(r)^{{\rm syn}}_{X|mD}\}_m\ar[r]^-{{\rm product}}\ar[d]^{a_{d-r}[-1] \otimes^{\mathbb{L}} a_{r}}\ar@{}[rd]|{(CSQ)}& i^*R\psi_*\mu_p^{\otimes d+1}[-1]\ar[d]\\
R\alpha'_*\mathbb{M}(d-r)[-1]\otimes^{\mathbb{L}}\{\mathbb{K}(r)_m\}_m\ar[r]&\mu'\otimes \nu_Y^{d-1}[-d-1].
}
\]
\end{Lem}
\begin{proof}
We have a commutative diagram in $D^b(Y_{\acute{e}t},\mathbb{Z}/p\mathbb{Z})$ by using an anti-commutative diagram $(10.4.4)$ in \cite{Sa2}:
 {\tiny\[
\xymatrix@C=15pt@M=8pt{
R\alpha'_*i'^*\mathfrak{T}(d-r)_U[-1]\otimes^{\mathbb{L}} \{\mathfrak{T}(r)^{{\rm syn}}_{X|mD}\}_m\ar[r]^-{{\rm product}}\ar[d]^{a_{d-r}[-1] \otimes^{\mathbb{L}} \id}\ar@{}[rd]|{(I)}& R\alpha'_*\mu''\otimes i^*R\psi_*\mu_p^{\otimes d-r-1}[-1]\otimes^{\mathbb{L}} \{\mathfrak{T}(r)^{{\rm syn}}_{X|mD}\}_m\ar[d]\\
R\alpha'_*\mathbb{M}(d-r)[-1]\otimes^{\mathbb{L}} \{\mathfrak{T}(r)^{{\rm syn}}_{X|mD}\}_m\ar[r]& R\alpha'_*\mu''\otimes Ri^!R\alpha_*\mathfrak{T}(d-r)_U\otimes^{\mathbb{L}} \{\mathfrak{T}(r)^{{\rm syn}}_{X|mD}\}_m.
}
\]}
By Lemma 10.5.6 in \cite{Sa2}, we obtain the following commutative diagram:
 \[
\xymatrix@C=26pt@M=8pt{
 R\alpha'_*\mu''\otimes i^*R\psi_*\mu_p^{\otimes d-r-1}[-1]\otimes^{\mathbb{L}} \{\mathfrak{T}(r)^{{\rm syn}}_{X|mD}\}_m\ar[d]\ar[r]\ar@{}[rd]|{(II)}&i^*R\psi_*\mu_p^{\otimes d+1}[-1]\ar[d]\\
 R\alpha'_*\mu''\otimes Ri^!R\alpha_*\mathfrak{T}(d-r)_U\otimes^{\mathbb{L}} \{\mathfrak{T}(r)^{{\rm syn}}_{X|mD}\}_m\ar[r]&\mu'\otimes \nu_Y^{d-1}[-d-1].
}
\]
Thus we combine the diagram $(I)$ with $(II)$, we obtain the following commutative diagram:
{\tiny \[
\xymatrix@C=7pt@M=7pt{
R\alpha'_*i'^*\mathfrak{T}(d-r)_U[-1]\otimes^{\mathbb{L}} \{\mathfrak{T}(r)^{{\rm syn}}_{X|mD}\}_m\ar[r]^-{{\rm product}}\ar[d]^{a_{d-r}[-1] \otimes^{\mathbb{L}} \id}\ar@{}[rd]|{(I)}& R\alpha'_*\mu''\otimes i^*R\psi_*\mu_p^{\otimes d-r-1}[-1]\otimes^{\mathbb{L}} \{\mathfrak{T}(r)^{{\rm syn}}_{X|mD}\}_m\ar[d]\ar@{}[rd]|{(II)}\ar[r]&i^*R\psi_*\mu_p^{\otimes d+1}[-1]\ar[d]\\
R\alpha'_*\mathbb{M}(d-r)[-1]\otimes^{\mathbb{L}} \{\mathfrak{T}(r)^{{\rm syn}}_{X|mD}\}_m\ar[r]\ar[d]\ar@{}[rrd]|{(III)}& R\alpha'_*\mu''\otimes Ri^!R\alpha_*\mathfrak{T}(d-r)_U\otimes^{\mathbb{L}} \{\mathfrak{T}(r)^{{\rm syn}}_{X|mD}\}_m\ar[r]&\mu'\otimes \nu_Y^{d-1}[-d-1]\ar[d]^-{\id}\\
R\alpha'_*\mathbb{M}(d-r)[-1]\otimes^{\mathbb{L}}\{\mathbb{K}(r)_m\}_m\ar[rr]&&\mu'\otimes \nu_Y^{d-1}[-d-1].
}
\]}
The commutativity of $(III)$ follows from the construction of the pairing.
This completes the proof.

\end{proof}
We return to the proof of the commutativity of the square $(\mathfrak{W})$. There is a commutative diagram ($\mathfrak{Diag}$):
{\tiny \[
\xymatrix@C=6pt@M=6pt{
\left(R\alpha'_*\mathbb{M}(d-r)[-1]\right)\otimes^{\mathbb{L}}\{\mu' \otimes \nu_{Y|mD}^{r-2}\}_m[-r+1] \ar[r]\ar[dd]& \left(R\alpha'_*\mathbb{M}(d-r)[-1]\right)\otimes^{\mathbb{L}} \{\mathbb{K}(r)_m\}_m \ar[r] \ar@/_90pt/@{>}[dd]\ar@{}[rd]|{(CSQ)}& \mu'' \otimes \nu_V^{d-1}[-d-1]\\
&R\alpha'_*i'^*\mathfrak{T}(d-r)_U[-1]\otimes^{\mathbb{L}} \{\mathfrak{T}(r)^{{\rm syn}}_{X|mD}\}_m\ar[u]\ar[r]^-{{\rm product}}\ar[d]& i'^*R\psi_*\mu_p^{\otimes d+1}[-1]\ar[u]\ar[d]\\
 R\alpha'_*\lambda_{V,1}^{d-r+1}[-d+r-2]\otimes^{\mathbb{L}}\{\mu' \otimes \nu_{Y|mD}^{r-2}\}_m[-r+1] \ar[r]&R\alpha'_*\lambda_{V,1}^{d-r+1}[-d+r-2]\otimes^{\mathbb{L}}\{\mathbb{K}(r)_m\}_m\ar[r]^-{(*)}& \mu'' \otimes \nu_V^{d-1}[-d-1].}\]}
Here the morphism $(*)$ correspondings to the pairing in section \S4, which follows from the fact that {\footnotesize \begin{align*}\Hom\left(R\alpha'_*\lambda_{V,1}^{d-r+1}[-d+r-2]\otimes^{\mathbb{L}}\{\mathbb{K}(r)_m\}_m,\; \mu'' \otimes \nu_V^{d-1}[-d-1]\right) \cong \\ \Hom\left(R\alpha'_*\lambda_{V,1}^{d-r+1}\otimes^{\mathbb{L}}\left(\mu' \otimes \nu^{r-2}_{Y|mD_s}\right),\; \mu' \otimes \nu_V^{d-1}\right)\end{align*}} by Lemma 2.1.1 in \cite{Sa2}. In the right hand side $\Hom$, there is a product morphism constructed in $\S4$. Then we get the morphism $R\alpha'_*\lambda_{V,1}^{d-r+1}[-d+r-2]\otimes^{\mathbb{L}}\{\mathbb{K}(r)_m\}_m\longrightarrow \mu'' \otimes \nu_V^{d-1}[-d-1]$.
The above diagram $(\mathfrak{Diag})$ commutes by using the commutativity of (CSQ)\;(the upside of the right square). Then we obtain the middle commutative square of $(\ref{CD1})$.
\end{proof}
\begin{Lem}\label{LLEM2}
{\tiny \begin{equation}\label{CD2}
\xymatrix@C=5pt@M=5pt{
 \mu' \otimes \alpha'_*\nu_{V}^{d-r-1}[-d+r] \ar[d]^-{\mathfrak{f}_1}  \ar[r]&R\alpha'_*\mathbb{V}(d-r) \ar[r]\ar[d]^-{\mathfrak{f}_2} &R\alpha'_*U^1M^{d-r+1}_V[-d+r-1]\ar[r]\ar[d]^-{\mathfrak{f}_3}& \mu' \otimes \alpha'_*\nu_{V}^{d-r-1}[-d+r+1]\ar[d]^-{\mathfrak{f}_1[1]}.\\
\mathbb{D}(``\plim[m]"\lambda^r_{Y|mD}[-r])\ar[r]&\mathbb{D}(``\plim[m]"F{\mathscr{M}^r_m}[-r]) \ar[r]&\mathbb{D}(``\plim[m]"U^1{\mathscr{M}^r_m}[-r])\ar[r]&
\mathbb{D}(``\plim[m]"\lambda^r_{Y|mD}[-r])[1].
}\end{equation}}
\end{Lem}
\begin{proof}
We have $\Hom\left( \mu' \otimes \alpha'_*\nu_{V}^{d-r-1}[-d+r+1],\; \mathbb{D}(``\plim[m]"U^1{\mathscr{M}^r_m}[-r])  \big)\right)=0$, and the left square commutes by a similar argument in $(1)$. Then there is a unique morphism $\mathfrak{f}'_3: R\alpha'_*U^1M^{d-r+1}_V[-d+r] \rightarrow \mathbb{D}(``\plim[m]"U^1{\mathscr{M}^r_m}[-r-1])$ fitting into $(\ref{CD2})$ by Lemma 2.1.2 (2) in \cite{Sa2}. We have $\mathfrak{f}'_3=\mathfrak{f}_3$ by the commutativity of (CSQ) and the construction of maps.
\end{proof}

\noindent
{\bf proof of the main result Theorem \ref{Main Thm}(2):}\\
In Lemma \ref{LLEM2}, since morphisms $\mathfrak{f}_1$ and $\mathfrak{f}_3$ are isomorphisms, the morphism $\mathfrak{f}_2$ is an isomorphism by Theorem and . 
Thus in Lemma \ref{LLEM1},  the morphism $(1)$ is an isomorphism. This completes the proof of the main result Theorem \ref{Main Thm}(2).$\square$\\
Take $i=2d$, $r=d$ in Theorem $\ref{Main Thm}$, we obtain the following:
 \begin{Cor}(cf. \cite{JSZ})
We obtain a natural isomorphism 
\[\plim[m]H^{2d}(X, \mathfrak{T}_n(d)_{X|mD})\xrightarrow{\cong}\pi_1^{ab}(U)/p^n. \]
\end{Cor}

\section{Acknowledgements}
I am great thankful to Professor Kanetomo Sato for his great help and discussion. This problem was suggested to me by him. I would like to express my thanks to my parents for their support.

\end{document}